\documentclass{amsart}
\usepackage{amsfonts,amssymb,mathtools,mathabx,mathrsfs}

\usepackage[usenames,dvipsnames]{xcolor}

\usepackage{hyperref}
\hypersetup{
    unicode=false,          		% non-Latin characters in Acrobat’s bookmarks
    pdftoolbar=true,        		% show Acrobat’s toolbar?
    pdfmenubar=true,        	% show Acrobat’s menu?
    pdffitwindow=false,     		% window fit to page when opened
    pdfstartview={FitH},    		% fits the width of the page to the window
    pdftitle={Fields Lectures},    		% title
    pdfauthor={Peter Perry},  % author
     linktocpage=true,			% link to page not section title in TOC
     pdfnewwindow=true,      	% links in new PDF window
    colorlinks=true,       			% false: boxed links; true: colored links
    linkcolor=blue,          		% color of internal links (change box color with linkbordercolor)
    citecolor=PineGreen,    	% color of links to bibliography
    filecolor=magenta,      		% color of file links
    urlcolor=cyan           		% color of external links
}

\theoremstyle{plain}
\newtheorem{theorem}{Theorem}[section]

\newtheorem{lemma}[theorem]{Lemma}
\newtheorem{proposition}[theorem]{Proposition}

\theoremstyle{definition}

\newtheorem{RHP}[theorem]{Riemann-Hilbert Problem}

\theoremstyle{remark}
\newtheorem{remark}[theorem]{Remark}

\newtheorem{exercise}{Exercise}[section]

\usepackage{palatino}

\numberwithin{figure}{section}
\numberwithin{equation}{section}

\allowdisplaybreaks

\DeclareMathOperator{\ad}{ad}
\DeclareMathOperator{\adj}{adj}
\DeclareMathOperator{\real}{Re}
\DeclareMathOperator{\imag}{Im}

\DeclareMathOperator{\Ran}{Ran}
\DeclareMathOperator{\tr}{tr}
\DeclareMathOperator{\Tr}{Tr}

\DeclareFontFamily{U}{mathx}{\hyphenchar\font45}
\DeclareFontShape{U}{mathx}{m}{n}{
      <5> <6> <7> <8> <9> <10>
      <10.95> <12> <14.4> <17.28> <20.74> <24.88>
      mathx10
      }{}
\DeclareSymbolFont{mathx}{U}{mathx}{m}{n}
\DeclareFontSubstitution{U}{mathx}{m}{n}
\DeclareMathAccent{\widecheck}{0}{mathx}{"71}
\DeclareMathAccent{\wideparen}{0}{mathx}{"75}

\begin{document}

\title[Inverse Scattering and Global Well-Posedness]{Inverse Scattering and Global Well-Posedness in One and Two Space Dimensions}
\author{Peter A. Perry}
\date{\today}
\maketitle
\tableofcontents

%%%%%%%%%%%%%%%%%%%%%%%
%
%		Hyphenation rules for foreign words
%
%%%%%%%%%%%%%%%%%%%%%%%

\hyphenation{Kol-mo-gor-ov}

%%%%%%%%%%%%%%%%%%%%%%%
%
%		For marginal notes while MS is 
%		being developed
%
%%%%%%%%%%%%%%%%%%%%%%%

\newcommand{\sidenote}[1]{\marginpar{\scriptsize{\color{purple} #1 }}}

%%%%%%%%%%%%%%%%%%%%%%%%%
%
%		macros.tex		
%		10.12.17
%
%%%%%%%%%%%%%%%%%%%%%%%%%

%% 	boldface letters - mostly for Lecture 3

\newcommand{\bfs}{\mathbf{s}}				%%	Scattering transform
\newcommand{\bfS}{\mathbf{S}}				%%	Beurling transform
\newcommand{\barS}{\overline{\bfS}}		%%	Conjugate Beurling transform

%%	script letters

\newcommand{\scrS}{\mathscr{S}}

%%  caligraphic letters

\newcommand{\calB}{\mathcal{B}}
\newcommand{\calC}{\mathcal{C}}
\newcommand{\calF}{\mathcal{F}}
\newcommand{\calI}{\mathcal{I}}
\newcommand{\calL}{\mathcal{L}}
\newcommand{\calM}{\mathcal{M}}
\newcommand{\calR}{\mathcal{R}}
\newcommand{\calS}{\mathcal{S}}

%% 	blackboard bold

\newcommand{\C}{\mathbb{C}}
\newcommand{\I}{\mathbb{I}}
\newcommand{\R}{\mathbb{R}}
\newcommand{\Z}{\mathbb{Z}} 

%% 	boldface - lower case

\newcommand{\bfn}{\mathbf{n}}

%%	boldface

\newcommand{\bfM}{\mathbf{M}}
\newcommand{\bfN}{\mathbf{N}}
\newcommand{\bfQ}{\mathbf{Q}}
\newcommand{\bfV}{\mathbf{V}}

%%	breve	-	lower case

\newcommand{\ba}{\breve{a}}
\newcommand{\bb}{\breve{b}}
\newcommand{\bm}{\breve{m}}
\newcommand{\bq}{\breve{q}}
\newcommand{\br}{\breve{r}}

%%	breve 	-	capital

\newcommand{\bQ}{\breve{Q}}

%%	breve	-	Greek

\newcommand{\bmu}{\breve{\mu}}

%%	shorthand	-	Greek letters and symbols

\newcommand{\dbar}{\overline{\partial}}
\newcommand{\dee}{\partial}
\newcommand{\eps}{\varepsilon}
\newcommand{\lam}{\lambda}
\newcommand{\sig}{\sigma}

%%	shorthand	-	arrows

\newcommand{\rarr}{\rightarrow}
\newcommand{\uarr}{\uparrow}
\newcommand{\darr}{\downarrow}

%%	barred letters (representing complex conjugates)

\newcommand{\fbar}{\overline{f}}
\newcommand{\gbar}{\overline{g}}
\newcommand{\kbar}{\overline{k}}
\newcommand{\qbar}{\overline{q}}
\newcommand{\rbar}{\overline{r}}
\newcommand{\ubar}{\overline{u}}
\newcommand{\vbar}{\overline{v}}
\newcommand{\wbar}{\overline{w}}
\newcommand{\zbar}{\overline{z}}

\newcommand{\Pbar}{\overline{P}}

\newcommand{\mubar}{\overline{\mu}}
\newcommand{\nubar}{\overline{\nu}}
\newcommand{\zetabar}{\overline{\zeta}}

%%	miscellaneous

\newcommand{\dotarg}{\, \cdot \,}

%%	shorthand - analysis

\newcommand{\dint}{\displaystyle{\int}}
\newcommand{\norm}[2][ ]{\left\Vert #2 \right\Vert_{#1}}
\newcommand{\bigO}[2][ ]{\mathcal{O}_{#1}\left( #2 \right)}
\newcommand{\littleo}[2][ ]{o_{{#1}}\left( #2 \right)}

%% 	shorthand - misc

\newcommand{\lin}{\mathrm{lin}}

%% modified overline command from 
%% tex.stackexchange.com
%% http://tex.stackexchange.com/questions/87609/setting-the-vertical-distance-in-overline

\newcommand{\ovmath}[1]
	{\overline{\mbox{#1}\raisebox{3mm}{}}}

%%%%%%%%%%%%%%%%%%%%%%%%%
%
%		Matrices and Vectors
%
%%%%%%%%%%%%%%%%%%%%%%%%%

\newcommand{\twomat}[4]
{
\begin{pmatrix}
#1		&		#2	\\
#3		&		#4
\end{pmatrix}
}

\newcommand{\Twomat}[4]
{
        \begin{pmatrix}
                #1      &       #2      \\[5pt]
                #3      &       #4
        \end{pmatrix}
}

\newcommand{\offmat}[2]{\twomat{0}{#1}{#2}{0}}
\newcommand{\diagmat}[2]{\twomat{#1}{0}{0}{#2}}

\newcommand{\Offmat}[2]{\Twomat{0}{#1}{#2}{0}}

\newcommand{\idmat}{\diagmat{1}{1}}

\newcommand{\twovec}[2]
{
	\begin{pmatrix}
	#1		\\	#2	
	\end{pmatrix}
}

%%%%%%%%%%%%%%%%%%%%%%%%%
%
%		Heading for problems that won't show up
%		in TOC
%
%%%%%%%%%%%%%%%%%%%%%%%%%

\newcommand{\Problems}
{
\medskip
\noindent
\textbf{Exercises.}
\medskip }

					%%	\newcommands etc.

%%%%%%%%%%%%%%%%%%%
%
%		More macros to be added
%
%%%%%%%%%%%%%%%%%%%

%%

\newcommand{\inv}{\mathrm{inv}}

%%%%%%%%%%%%%%%%%%%

\section*{Preface}

These notes are a considerably revised and expanded version of lectures given at the Fields Institute workshop on 
``Nonlinear Dispersive Partial Differential Equations and Inverse Scattering'' in August 2017. These lectures, together with lectures of Walter Craig, Patrick Gerard, Peter Miller, and Jean-Claude Saut, constituted a week-long introduction to recent developments in inverse scattering and dispersive PDE intended for students and postdoctoral researchers working in these two areas. 

The goal of my lectures was to give a complete and mathematically rigorous exposition of the inverse scattering method for two dispersive equations: the defocussing cubic nonlinear Schr\"{o}dinger equation (NLS) in one space dimension, and the defocussing Davey-Stewartson II (DS II) equation in two space dimensions.  Each is arguably the simplest example in its class since neither admits solitons; moreover, the long-time behavior of solutions to each equation has been rigorously deduced from inverse scattering \cite{DZ:2003,NRT:2017,Perry:2016}. Inverse scattering for the  defocussing NLS provides an introduction to the Riemann-Hilbert method further discussed in the contribution of Dieng, McLaughlin and Miller in this volume \cite{DMM:2018}. Inverse scattering for the defocussing  Davey-Stewartson II equation provides an introduction to the $\dbar$-methods used extensively in two-dimensional inverse scattering (see, for example, the surveys \cite{BC:1989,FA:1983b,Grinevich:2000}, the monograph \cite{Ablowitz:2011}, and references therein).

In Lecture 1, I give an overview of the inverse scattering method for these two equations, focusing on the formal (i.e., algebraic) aspects of the theory.  I motivate the solution formulas given by inverse scattering, seen as a composition of nonlinear maps and a linear time-evolution of scattering data. 

In Lecture 2, I analyze inverse scattering for the defocussing cubic NLS in one dimension in depth, based on  the seminal paper of Deift and Zhou \cite{DZ:2003}.  We study the direct map via Volterra integral equations for the Jost solutions for the operator \eqref{ZS-AKNS}, and the inverse map via the Riemann-Hilbert Problem \ref{NLS:RHP} using the approach of Beals and Coifman \cite{BC:1984}.

In Lecture 3, I discuss the inverse scattering method for the Davey-Stewartson II equation in depth. 
This lecture has been completely rewritten in light of the recent work of Nachman, Regev, and Tataru \cite{NRT:2017} which introduced a number of new ideas and techniques from harmonic analysis and pseudodifferential operators to the study of scattering maps. Using these techniques, the authors proved that the defocussing DS II equation is globally well-posed in $L^2(\R)$ and that all solutions scatter to solutions of the associated linear problem. Their results significantly improve earlier work of mine \cite{Perry:2016}, which proved global well-posedness of the defocussing DS II equation in the weighted Sobolev space $H^{1,1}(\R)$ and obtained large-time (dispersive) asymptotics of solutions in $L^\infty$-norm.  In the revised lecture 3, I give a pedagogical proof of the results in \cite{Perry:2016} using some of the ideas of \cite{NRT:2017} to streamline proofs  significantly.

At the end of each lecture, I've added exercises which supplement the text and develop key ideas. 

 I hope that these lectures will appeal to a wide audience interested in recent progress in this rapidly developing field.

\section*{Acknowledgments}

I am grateful to the University of Kentucky for sabbatical support during part of the time these lectures were prepared, and to Adrian Nachman and Idan Regev for many helpful discussions about their work. I thank the participants in a Fall 2017 working seminar on \cite{NRT:2017}--Russell Brown, Joel Klipfel, George Lytle, and Mihai Tohaneanu--for helping me understand this paper in greater depth. I have benefited from conversations with Deniz Bilman about Lax representations and Riemann-Hilbert problems. I have also benefited from course notes for Percy Deift's 2008 course at the Courant Institute on the defocussing NLS equation,\footnote{I am grateful to Percy Deift for kindly sharing his handwritten notes for this course.}  his 2015 course there on integrable systems,\footnote{Handwritten notes may be found at \url{https://www.math.nyu.edu/faculty/deift/RHP/} } and Peter Miller's Winter 2018 course at the University of Michigan on integrable systems and Riemann-Hilbert problems.\footnote{Please see the course website \url{http://www.math.lsa.umich.edu/~millerpd/CurrentCourses/651_Winter18.html}. }
For supplementary material, I have drawn on several excellent sources for material on dispersive equations, Riemann-Hilbert problems, and $\dbar$-problems, namely 
the monograph of Astala, Iwaniec, and Martin \cite{AIM:2009},
the textbook of Ponce and Linares \cite{LP:2015}, and 
the monograph of Trogdon and Olver \cite{TO:2016}.

\newpage

\section*{Notation Index}

%%%%%%%%%%%%%%%%%%%%%%%%%%%%%%
%
%		Use itemize with math items and give
%		page and line reference
%
%%%%%%%%%%%%%%%%%%%%%%%%%%%%%%

\newcommand{\mathitem}[1]{\item[${#1}$]}
\newcommand{\fullref}[1]{(see \eqref{#1}, page \pageref{#1})}

\begin{itemize}

%%%%		Lower case roman

%% a
\mathitem{\ad \sig_3}		Operator $A \mapsto [\sig_3,A]$ on $2 \times 2$ matrices

%% e
\mathitem{e_k(z)} Unimodular phase function \fullref{DSII:ek} where $k,z \in \C$

\mathitem{e^{it\Delta}} Solution operator for linear Schr\"{o}dinger equation \eqref{LSE}

%% m
\mathitem{m^1,\, m^2} Solutions to \eqref{DSII:m.x1}--\eqref{DSII:m.x2} obeying the asymptotic condition \eqref{DSII:m.x.asy}

%%n
\mathitem{\bfn(x,z)}	Null vector for a Riemann-Hilbert problem (see Proposition \ref{NLS:prop.null})

%% p
\mathitem{p^*} Sobolev conjugate exponent $(2-p)/2p$ 

%% q

\mathitem{q_\inv} Solution of DS II by inverse scattering \fullref{DSII:qinv}
\mathitem{q_\lin} Solution of linearized DS II equation \fullref{DSII:lin.bis}

%% r

\mathitem{r}			``Right'' reflection coefficient  for NLS $r(\lam) = -b(\lam)/\overline{a(\lam)}$ \fullref{BC.Jump.left}
\mathitem{\br}		``Left'' reflection coefficient for NLS $\br(\lam)= -b(\lam)/a(\lam)$ \fullref{BC.Jump.right}

%% s
\mathitem{\bfs} Scattering transform  $\bfs = \calS q$ for DSII equation \fullref{DSII:scatt}

%%%% Upper case roman
\bigskip

%% B
\mathitem{\calB(Y)}  The Banach algebra of bounded operators on a Banach space $Y$

%% C
% line below is apparently unused in current version
%\mathitem{\calC }		Cauchy transform acting on $L^2(\R)$
\mathitem{C_\pm}		Cauchy projectors for $L^2(\R)$ \fullref{Cpm}
\mathitem{\calC_w}	  Beals-Coifman integral operator  \fullref{NLS:BC.op} with weights $w=(w^+,w^-)$

%% F
\mathitem{\calF}		Fourier transform \eqref{Fourier} (lectures 1 and 2)
\mathitem{\calF_a}		Antilinear Fourier transform \fullref{DSII:scatt.lin} 

%% H		
\mathitem{H^1(\R)}  Sobolev space \fullref{NLS:H1}
\mathitem{H^{1,1}(\R)}  Weighted Sobolev space \fullref{NLS:H11}
\mathitem{H^{1,1}_1(\R)} Open subset of $H^{1,1}(\R)$ consisting of those $r \in H^{1,1}(\R)$ with
									$\norm[\infty]{r} < 1$. 
\mathitem{H^{1,1}(\R^2)}		Weighted Sobolev space \fullref{DSII:H11}
%% I
\mathitem{\I}		The $2 \times 2$ identity matrix

\mathitem{\calI}		Inverse scattering map for NLS (see RHP \ref{NLS:RHP.I} and \eqref{NLS:M.q})

%% L

\mathitem{\calL}	Linear operator for NLS spectral problem  \fullref{ZS-AKNS}  or DS II spectral problem \fullref{DSII:LS}
%% M

\mathitem{\bfM}			Normalized solution $\bfM(x,z) = \psi(x,z) e^{-ixz\sig_3}$ \fullref{NLS:M}
\mathitem{\bfM^\pm}	Normalized Jost solution defined by $\Psi^\pm = \bfM^\pm e^{-ix\lam \sig_3}$
\mathitem{\bfM^\ell, \, \bfM^r}  Left and right Beals-Coifman solutions (see Theorems \ref{thm:NLS.BC.left} and \ref{thm:NLS.BC.right})
\mathitem{\bfM_\pm} 	Boundary values of a Beals-Coifman solution

\mathitem{\calM f}		Hardy-Littlewood maximal function of $f \in L^p(\R^n)$
%%  N

\mathitem{\bfN}			Normalized Jost solution \fullref{NLS:N}

%%  Q

\mathitem{\bfQ}			Potential matrix in $\calL$ \fullref{ZS-AKNS} (NLS) or (\eqref{DSII:LS}, page \pageref{DSII:LS}) (DS II)
\mathitem{\bfQ_1, \bfQ_2} 		Matrices in Lax representation for NLS \fullref{NLS:Lax}

%% R

\mathitem{\calR}		Direct scattering map for NLS \fullref{NLS:R}

%% S

\mathitem{\scrS(\R), \scrS(\R^2)}	Schwartz class of $C^\infty$ functions of rapid decrease on $\R$, $\R^2$
\mathitem{\scrS_1(\R)} Functions $r \in \scrS(\R)$ with $\norm[\infty]{r} < 1$
\mathitem{S} Model compact operator \fullref{DBAR:op}
\mathitem{\calS}		Scattering map for DS II equation \fullref{DSII:S}
\mathitem{\bfS} Beurling transform \fullref{Beurling}
\mathitem{\overline{\bfS}} Conjugate Beurling transform \fullref{Conjugate-Beurling}

%% T
\mathitem{T(\lam)} Transition matrix \fullref{NLS:T.sym}

%% U
\mathitem{V(t)}  Solution operator for linearized DS II equation \eqref{DSII:lin.bis}

%% V

\mathitem{ \bfV^\ell, \bfV^r}  Jump matrices for Riemann-Hilbert problems satisfied by left and right Beals-Coifman solutions (see respectively \eqref{NLS:Jump.left} and \eqref{NLS:Jump.right})

%%%% Greek
\bigskip

%% theta

\mathitem{\theta} Phase function for the RHP that solves NLS \fullref{NLS:RHP.phase}

%% phi
\mathitem{\varphi} Phase function for the $\dbar$-problem that solves DS II \fullref{DSII:phase}

%% psi
\mathitem{\Psi^\pm}	Jost solutions to $\calL \psi = \lam \psi$ \fullref{NLS:Jost.asym}
\mathitem{\psi}			Generic solution to $\calL \psi = z \psi$

%% mu

\mathitem{\mu}		Solution to Beals-Coifman integral equation \fullref{NLS:BC}
\mathitem{\mu^\sharp} $\mu-\I$

%% sigma
\mathitem{\sigma_1, \sigma_2, \sigma_3}  Pauli matrices 
								$\begin{pmatrix}
								0 & 1 \\ 1 & 0
								\end{pmatrix}$,
								$\begin{pmatrix}
								0	&	-i	\\	i	&	0
								\end{pmatrix}$,
								$\begin{pmatrix}
								1	&	0	\\	0	&	-1
								\end{pmatrix}$

%%%%		Misc mathematics
\bigskip

\mathitem{\dee_{\zbar}{-1}} Solid Cauchy transform \fullref{DSII:solid-Cauchy} 
\mathitem{\dee_z^{-1}} Conjugate solid Cauchy transform \fullref{DSII:conjugate-solid-Cauchy}
\mathitem{\dee C(L^2)} The linear space of pairs $(h_+,h_-)$ with $h_\pm = C_\pm h$ for some $h \in L^2(\R)$
\mathitem{\dee_z,\, \dee_{\zbar}} Wirtinger ($\zbar$ and $z$) derivatives \fullref{DSII:dee-dbar}

\end{itemize}

\newpage

%%%%%%%%%%%%%%%%%%%

\section{Introduction to Inverse Scattering}

Among dispersive PDE's that describe wave propagation are the \emph{completely integrable} PDE's. These equations--which include the Korteweg-de Vries and cubic NLS equations in one space dimension, and the Davey-Stewartson and Kadomtsev-Petviashvilli equation in two space dimensions--are equivalent to simple linear flows by conjugation with an invertible, nonlinear map adapted to the PDE and very strongly dependent on its special structure. This nonlinear map is called a \emph{scattering transform} and serves the same function for  these equations that the Fourier transform does for the linear Schr\"{o}dinger equation. 

To understand what this means, let's consider the Cauchy problem for  linear Schr\"{o}dinger equation in one dimension.
\begin{equation}
\label{LSE}
\left\{
\begin{aligned}
i \frac{\dee q}{\dee t} + \frac{\dee^2 q}{\dee x^2} &= 0	,	\\
q(x,0) 		&=	q_0(x).
\end{aligned}
\right.
%}
\end{equation}
assuming for simplicity that $q_0 \in \scrS(\R)$. 
The Fourier transform
$$ \left(\calF q\right)(\xi) = \widehat{q}(\xi) = \int e^{-ix\xi} f(x) \, dx $$
reduces the Cauchy problem \eqref{LSE} to the trivial flow
$$ i \frac{\dee}{\dee t} \widehat{q}(\xi,t)  =  |\xi|^2 \widehat{q}(\xi,t) $$
leading to the solution formula
\begin{equation}
\label{LSE:Fourier-sol}
q(x,t) = \frac{1}{2\pi} \int_\R e^{it \theta} \widehat{q_0}(\xi) \, d\xi, \quad \theta_0(\xi;x,t) = \xi x/t - \xi^2.
\end{equation}
We can also write the solution as $q(t) = e^{it\Delta}q_0$ where $e^{it\Delta}$ is 
the solution operator
\begin{equation}
\label{LSE:SO}
\left( e^{it\Delta} f \right)(x) = \calF^{-1} \left( e^{-it(\dotarg)^2} \left( \calF f \right) \right)(x). 
\end{equation}

Since the phase function $\theta_0$ in \eqref{LSE:Fourier-sol} has a single, nondegenerate critical point at $\xi_0 = x/2t$, the solution has large-time asymptotics
$$ 
q(x,t) \sim  \frac{1}{\sqrt{4 \pi i t}}e^{ix^2/(4t)} \widehat{q_0}\left( \frac{x}{2t} \right) + \bigO{t^{-3/4}}.$$

The map $\calF$ is linear, has a well-behaved and explicit inverse, and yields a well-behaved solution formula that extends to initial data in Sobolev spaces. The representation formula, combined with stationary phase methods, leads to a complete description of long-time asymptotic behavior.

Similar results may be obtained for integrable systems provided that the scattering transforms are well-controlled and have well-behaved inverses. In these lectures we will discuss two examples in depth: the defocussing, cubic nonlinear Schr\"{o}dinger equation in one space dimension, and the defocussing Davey-{Stew\-art\-son} II equation in two space dimensions. Neither of these equations admits solitons, so that the dynamics are purely dispersive. 

In each case, the scattering transform is the ``next best thing to linear'': it is  a  diffeomorphism when restricted to appropriate function spaces (and, in each case, its Fr\'{e}ch\'{e}t derivative at zero is a Fourier-like transform--see Remarks \ref{NLS:Direct.lin} and \ref{NLS:Inverse.lin} for the NLS scattering maps, and see \eqref{DSII:scatt.lin} and the accompanying discussion for the DS II scattering map). It also has global Fourier-like properties which facilitate the analysis of large-time asymptotics of the solution.

\subsection{The Defocussing Cubic Nonlinear Schr\"{o}dinger Equation}

Zakharov and Shabat \cite{ZS:1972} showed that the Cauchy problem for the cubic nonlinear Schr\"{o}dinger equation
\begin{equation}
\label{NLS}
\left\{
\begin{aligned}
i \frac{\dee q}{\dee t} + \frac{\dee^2 q }{\dee x^2} - 2 |q|^2 q &= 0	\\	
q(x,0)	&=	q_0(x)
\end{aligned}
\right.
%}
\end{equation}
is integrable by inverse scattering. To describe the direct and inverse scattering maps, we will follow the conventions of \cite{DMM:2018}. These conventions differ slightly but inessentially from those of Deift \cite{Deift:2018} and  Deift-Zhou \cite{DZ:2003}.\footnote{Deift and Zhou write the ZS-AKNS equation as $\psi_x = i\lam \sig \psi + \bfQ_1 \psi$ where $\sig=(1/2)\sig_3$. This results in various sign changes and changes in factors of $2$ throughout. The conventions of \cite{DMM:2018} also make the scattering maps linearize to antilinear Fourier-type transforms, whereas those of \cite{DZ:2003} linearize to the usual Fourier transform.}

Equation \eqref{NLS} is the consistency condition for the overdetermined system
\begin{equation}
\label{NLS:Lax}
\left\{
\begin{aligned}
\Psi_x		&=	-i\lam \sig_3 \Psi + \bfQ_1 \Psi,	\\
\Psi_t		&=	\left(-2i\lam^2 \sig_3 +2\lam \bfQ_1 + \bfQ_2 \right) \Psi 
\end{aligned}
\right.
%}
\end{equation}
where
$$ 
\bfQ_1 = 
\begin{pmatrix}
0					&	q	\\[5pt]
\overline{q}	&	0	
\end{pmatrix},
\quad
\bfQ_2	=
\begin{pmatrix}
-i|q|^2				&	iq_x	\\[5pt]
-i\overline{q_x}	&	i|q|^2
\end{pmatrix}
$$
and $\Psi$ is an unknown $2 \times 2$ matrix-valued function of $(x,t)$. 
Note that the first of  equations \eqref{NLS} is an eigenvalue problem for the self-adjoint operator
\begin{equation}
\label{ZS-AKNS}
\calL \coloneqq i\sigma_3 \frac{d}{dx} + \bfQ(x), \quad
\sigma_3 \coloneqq \begin{pmatrix} 1 & 0 \\ 0 & -1 \end{pmatrix}, \quad
\bfQ(x) = \begin{pmatrix} 0 & -iq(x) \\ i\overline{q(x)} & 0 \end{pmatrix} 
\end{equation}
acting on  $L^2(\R,M_2(\C))$, the square-integrable, $2 \times 2$-matrix valued functions (see Exercise \ref{ex:ZS-AKNS.sa}). This equation is sometimes called the ZS-AKNS equation after the fundamental papers of Ablowitz, Kaup, Newell, and Segur \cite{AKNS:1974}  and 
Zakharaov-Shabat \cite{ZS:1972}. Note that, if $q=0$, the operator $\calL$ has continuous spectrum on the real line and bounded, matrix-valued eigenfunctions $\Psi_0(x,\lam) = e^{-i\lam x \sig_3}$.

The scattering transform of $q \in L^1(\R)$ is defined as follows. There exist unique $2 \times 2$ matrix-valued solutions $\Psi^\pm(x,\lam)$ of $\calL \psi = \lam \psi $ with
\begin{equation}
\label{NLS:Jost.asym}
\lim_{x \to \pm \infty} \Psi^\pm(x,\lam) e^{ix\lam \sig_3}= \I,
\end{equation}
where $\I$ denotes the identity matrix. Matrix-valued solutions of $\calL \psi  = \lam \psi$ have the properties that (i) $\det \psi(x)$ is independent of $x$ and (ii) any two nonsingular solutions $\psi_1$ and $\psi_2$ are related by $\psi_1 = \psi_2 A$
for a constant matrix $A$ (see Exercises \ref{ex:NLS.det1} and \ref{ex:NLS.matmult}). For this reason, there is a matrix $T(\lam)$ with
\begin{equation}
\label{NLS:T}	
\Psi^+(x,\lam)	=	\Psi^-(x,\lam) T(\lam). 
\end{equation}
Clearly $\det T(\lam) = 1$ and, by a symmetry argument (see Exercises \ref{ex:NLS.sigma1} and \ref{ex:NLS.mu.bc}),
\begin{equation}
\label{NLS:T.sym}
T(\lam) = 
\begin{pmatrix}
a(\lam)	&	\overline{b(\lam)}	\\
b(\lam)	&	\overline{a(\lam)}
\end{pmatrix},
\quad
|a(\lam)|^2 - |b(\lam)|^2 = 1
\end{equation}
The \emph{direct scattering map} is the map $\calR : q \to r$ where \begin{equation}
\label{NLS:R}
r(\lam) = -b(\lam)/\overline{a(\lam)}.
\end{equation}
The direct scattering map linearizes the flow \eqref{NLS} in the sense that, if 
$q(x,t)$ solves the initial value problem \eqref{NLS} with initial data $q_0$ and $r_0 = \calR(q_0)$, then
$$
\calR \left( q(\dotarg,t) \right)(\lam)  =  e^{4it\lam^2} r_0(\lam).
$$
We give a heuristic proof of this law of evolution, based on the Lax representation \eqref{NLS:Lax}, at the beginning of section \ref{NLS:sec.sol}.

The inverse of $\calR$ is determined as follows. Denote by $\psi(x,z)$ a solution to $\calL \psi = z \psi$, written as
$$
\frac{d}{dx} \psi = -iz \sigma_3 \psi + \bfQ_1(x) \psi, 
\quad 
\bfQ_1(x) = 
\begin{pmatrix} 
0 & q(x) \\ 
\overline{q(x)} & 0 
\end{pmatrix},
$$
where $z \in C$, and factor $\psi(x,z) = \bfM(x,z) e^{-ixz \sig_3}$. Then $\bfM$ obeys the differential equation
\begin{equation}
\label{NLS:M}
\frac{d}{dx} \bfM(x,z)  = -iz \ad \sig_3 \left( \bfM(x,z) \right) + \bfQ_1(x) \bfM(x,z)
\end{equation}
where, for a $2 \times 2$ matrix $A$,
$$
\ad \sig_3 (A)  = \left[ \sig_3, A \right].
$$
One can show that \eqref{NLS:M} admits special solutions, the \emph{Beals-Coifman solutions}, which are piecewise analytic in $\C \setminus \R$,  have distinct boundary values $\bfM_\pm(x,\lam)$ on $\R$, and for each $z$ obey the asymptotic conditions
$$ \bfM(x,z) \to \I \,\, \text{ as } x \to +\infty, \quad \bfM(x,z) \text{ is bounded as } x \to -\infty.$$ 
The Beals-Coifman solutions are unique and, moreover, $q(x)$ can be recovered from their asymptotic behavior:
\begin{equation}
\label{NLS:M.q}
q(x) = \lim_{z \to \infty} 2iz \bfM_{12}(x,z).
\end{equation}
The boundary values  satisfy a jump relation 
\begin{equation}
\label{NLS.M.Jump}
\left\{
\begin{aligned}
\bfM_+(x,\lam) &= \bfM_-(x,\lam) \bfV(\lam;x),	\\[5pt]
\bfV(\lam;x) &= 
\begin{pmatrix}
1-|r(\lam)|^2			&	-\overline{r(\lam)}e^{-2i\lam x}	\\[5pt]
r(\lam) e^{2i\lam x}	&	1 		
\end{pmatrix}.
\end{aligned}
\right.
%}
\end{equation}
The asymptotics of $\bfM(x,z)$ together with the jump relation \eqref{NLS.M.Jump} define a Riemann-Hilbert problem for $\bfM(x,z)$, which we write as $\bfM(z;x)$ to emphasize that $x$ plays the role of a parameter in the RHP.

\begin{RHP}
\label{NLS:RHP.I}
For given $r$, and $x \in \R$, find $\bfM(z;x)$ so that
\begin{enumerate}
\item[(i)]		$\bfM(z;x)$ is analytic in $\C \setminus \R$ for each $x$,
\item[(ii)]	$\lim_{z \to \infty} \bfM(z;x) = \I$,
\item[(iii)]	$\bfM(z;x)$ has continuous boundary values $\bfM_\pm(\lam;x)$ on $\R$
\item[(iv)]	The jump relation
		$$ \bfM_+(\lam;x) = \bfM_-(\lam;x) V(\lam;x), \quad 
			\bfV(\lam;x) = 
			\begin{pmatrix}
			1-|r(\lam)|^2 			&	-\overline{r(\lam)} e^{-2i\lam x}\\[5pt]
			r(\lam) e^{2i\lam x}	&	 1		
			\end{pmatrix}
		$$
		holds.
\end{enumerate}
\end{RHP}

\begin{remark}
\label{NLS:remark.lim.z}
In condition (ii) above, the limit is meant to be uniform in proper subsectors of the upper and lower half planes. That is, for any $\eps>0$,
$$ \lim_{R \to \infty} \sup_{\substack{|z| \geq R \\  \arg(z) \in (\alpha,\beta)}} \left( \bfM(z;x) - \I \right) = 0 $$
if $(\alpha,\beta)$ is a proper subinterval of $(0,\pi)$ or $(\pi,2\pi)$. 
\end{remark}

The \emph{inverse scattering map}  $\calI: r \to q$ determined by Riemann-Hilbert Problem \ref{NLS:RHP.I} and the reconstruction formula \eqref{NLS:M.q}.

Thus, to implement the solution formula
\begin{equation}
\label{NLS:IST-sol}
q(x,t) = \calI \left( e^{4i(\dotarg)^2 t} \calR(q_0)(\dotarg)\right)(x), 
\end{equation}
we compute the scattering transform $r_0 = \calR(q_0)$ and solve the following Riemann-Hilbert problem (RHP).

\begin{RHP}
\label{NLS:RHP}
For given $r_0$ and parameters $x$, $t$, find $\bfM(z;x,t)$ so that
\begin{enumerate}
\item[(i)]		$\bfM(z;x,t)$ is analytic in $\C \setminus \R$ for each $x,t$,
\item[(ii)]	$\lim_{z \to \infty} \bfM(z;x,t) = \I$,
\item[(iii)]	$\bfM(z;x,t)$ has continuous boundary values $\bfM_\pm(\lam;x,t)$ on $\R$
\item[(iv)]	The jump relation
		\begin{equation*} 
		\begin{aligned}
			\bfM_+(\lam;x,t) &= \bfM_-(\lam;x,t) \bfV(\lam;x,t), \\
			\bfV(\lam;x,t) &= 
			\begin{pmatrix}
			1-|r_0(\lam)|^2	&	-\overline{r_0(\lam)} e^{-2it\theta}\\[5pt]
			r_0(\lam) e^{2it\theta}	&	1
			\end{pmatrix}
		\end{aligned}
		\end{equation*}
		holds, where 
		\begin{equation}
		\label{NLS:RHP.phase}
		\theta(\lam;x,t) = 2\lam^2 + x\lam/t. 
		\end{equation}
\end{enumerate}
\end{RHP}

\medskip
Given the solution of RHP \ref{NLS:RHP}, we can then compute
$$
q(x,t)	=	\lim_{z \to \infty} 2iz \bfM_{12}(z;x,t)
$$
where the limit is meant in the sense of Remark \ref{NLS:remark.lim.z}.

Denote by $\scrS(\R)$ the Schwartz class functions on $\R$ and let
$$ \scrS_1(\R) = \left\{ r \in \scrS(\R): \norm[\infty]{r} < 1 \right\}. $$
Beals and Coifman \cite{BC:1984} proved:
\begin{theorem}
\label{thm:NLS.BC}
Suppose that $q_0 \in \scrS(\R)$. Then $r_0 \in \scrS_1(\R)$, RHP \ref{NLS:RHP} has a unique solution for each $x,t \in \R$, and \eqref{NLS:IST-sol} defines a classical solution of the Cauchy problem \eqref{NLS}.
\end{theorem}

We will give a complete proof of Theorem \ref{thm:NLS.BC} in section \ref{NLS:sec.sol}.

The solution formula \eqref{NLS:IST-sol} defines a continuous solution map provided that its component maps are continuous. To describe the mapping properties of $\calR$ and $\calI$, we define
\begin{equation}
\label{NLS:H11}
H^{1,1}(\R) \coloneqq \left\{ f \in L^2(\R): f', \, xf \in L^2(\R) \right\} 
\end{equation}
and
$$
H^{1,1}_1(\R) \coloneqq \left\{ f \in H^{1,1}(\R): \norm[\infty]{f} < 1 \right\}.
$$
Note that $H^{1,1}_1(\R)$ is an open subset of $H^{1,1}(\R)$ since $\norm[\infty]{f} \leq c \norm[H^{1,1}]{f}$ (see Exercise \ref{ex:H10}). The following result is proved by Deift and Zhou in \cite[\S 3]{DZ:2003} and also follows from Zhou's analysis \cite{Zhou:1998} of Sobolev mapping properties of the scattering transform.

\begin{theorem}\cite{DZ:2003}
\label{thm:NLS.maps} 
The maps $\calR : H^{1,1}(\R) \to H^{1,1}_1(\R)$ and $\calI: H^{1,1}_1(\R) \to H^{1,1}(\R)$ are Lipschitz continuous maps with $\calR \circ \calI = \calI \circ \calR = I$.
\end{theorem}

A consequence of Theorems \ref{thm:NLS.BC},  \ref{thm:NLS.maps}, and local well-posedness theory for the NLS is:

\begin{theorem}
\label{thm:NLS.GWP}
The Cauchy problem \ref{NLS} is globally well-posed in $H^{1,1}(\R)$ with solution \eqref{NLS:IST-sol}.
\end{theorem}

Theorem \ref{thm:NLS.GWP} is of interest not because of the global well-posedness result: far superior results are available through PDE methods--see, for example, \cite{LP:2015} and \cite{Tao:2006} and references therein--and most recently through a very different approach to complete integrability pioneered by Koch-Tataru \cite{KT:2018},  Killip-Visan-Zhang \cite{KVZ:2018}, and Killip-Visan   \cite{KV:2018} which give conserved quantities and well-posedness results in the presence of very rough initial data.  Rather, Theorem \ref{thm:NLS.GWP} is of interest  because the solution map so constructed can be used to study large-time asymptotics of solutions with initial data in $H^{1,1}(\R)$. Deift and Zhou \cite{DZ:2003} gave a rigorous proof of long-time dispersive behavior for the solution of \eqref{NLS}, motivated by formal results of Zakharov and Manakov \cite{ZM:1976}. Their proof is an application of the Deift-Zhou steepest descent method \cite{DZ:1993}. Dieng and McLaughlin \cite{DM:2008} gave a different proof using the `$\dbar$-steepest descent method' and obtained a sharp remainder estimates. This result is discussed in a companion paper by Dieng, McLaughlin and Miller in this volume \cite{DMM:2018}. 

\begin{theorem} \cite{DMM:2018}
The unique solution to \eqref{NLS}  with initial data $q_0 \in H^{1,1}(\R)$ has the asymptotic behavior
$$
q(x,t) \sim t^{-1/2}\alpha(z_0) e^{ix^2/(4t) - i\nu(z_0) \log (2t)} + \bigO{t^{-3/4}}
$$
where
$z_0 = -x/(4t)$, $\nu(z) = -\dfrac{1}{2\pi} \log\left(1-|r(z)|^2 \right)$, $|\alpha(z)|^2 = \nu(z)/2$, and 
$$
\arg \alpha(z) = %\\
\frac{1}{\pi}\int_{-\infty}^z \log(z-s) \, d\left( \log\left(1-|r(s)|^2 \right)\right) +
\frac{\pi}{4} + \arg(i\nu(z)) + \arg r(z).
$$
The remainder term is uniform in $x \in \R$. 
\end{theorem}

\subsection{The Defocussing Davey-Stewartson II Equation}
The Cauchy problem for the defocussing Davey-Stewartson II (DSII) equation is
\begin{equation}
\label{DSII}
\left\{
\begin{aligned}
iq_t + 2(\dee_z^2 + \dee_{\zbar}^2) q + (g + \gbar) q &=0	,\\
\dee_{\zbar} g	&=	-4 \dee_z \left( |q|^2 \right),\\
q(z,0)	&=	q_0(z).
\end{aligned}
\right.
%}
\end{equation}
Here $z=x_1+ix_2$ and
$$ \dee_{\zbar} = \frac{1}{2} \left( \frac{\dee}{\dee x_1} + i \frac{\dee }{\dee x_2} \right),  \dee_z = \frac{1}{2} \left( \frac{\dee}{\dee x_1} - i \frac{\dee }{\dee x_2} \right). $$
Here and in what follows, the notation $f(z)$ for a function of $z=x_1+ix_2$ does \emph{not} imply that $f$ is an analytic function of $z$.

Ablowitz and Segur \cite[Chapter 2, \S 2.1.d]{AS:1981} showed that the Davey-Stewartson II equation is completely integrable. The solution of DS II by inverse scattering was developed by Beals-Coifman \cite{BC:1998,BC:1985a,BC:1986} and Fokas-Ablowitz \cite{FA:1983c,FA:1984a,FA:1983b}. A rigorous analysis of the scattering maps, including the case $q_0 \in \scrS(\R^2)$ was carried out by Sung in a series of three papers \cite{Sung:1994a,Sung:1994b,Sung:1994c}. 

The DSII flow is linearized by a zero-energy spectral problem for the operator
\begin{equation}
\label{DSII:LS}
\calL = 
\begin{pmatrix}
\dee_{\zbar}		&	0 	\\
0			&	\dee_z
\end{pmatrix} - \bfQ(z), \quad
\bfQ(z) =
\begin{pmatrix}
0						&	q(z)	\\
\overline{q(z)}	&	0
\end{pmatrix}
.
\end{equation}
To define the scattering transform for $q \in \scrS(\R^2)$, we look for solutions 
$$ 
\begin{pmatrix}
\psi_1(z,k)	\\[5pt]
\psi_2(z,k)
\end{pmatrix}
=
\begin{pmatrix}
m^1 (z,k) e^{ikz}	\\[5pt]
m^2 (k,z) e^{ikz}
\end{pmatrix}
$$
of $\calL \psi = \mathbf{0}$, where $kz$ denotes complex multiplication of $k=k_1+ik_2$ by $z=x_1+ix_2$.   We assume that
$m^1 (z,k) \to 1$ and $m^2 (z,k) \to 0$ for each $k \in \C$ as $|z| \to \infty$. Such unbounded solutions $\psi_1, \psi_2$  are sometimes called \emph{complex geometric optics} (CGO) solutions and were introduced in scattering theory by Faddeev \cite{Faddeev:1965}.  An easy computation shows that
\begin{equation}
\label{DSII:M}
\left\{
\begin{aligned}
\dee_{\zbar} m^1 (z,k)										&=	q(z) m^2 (z,k),	\\
\left(\dee_z + ik \right) m^2 (z,k)					&=	\overline{q(z)} m^1 (z,k), \\
\lim_{|z| \to \infty}		(m^1 (z,k),m^2 (z,k))	&=	(1,0).
\end{aligned}
\right.
%}
\end{equation}
The system \eqref{DSII:M} is formally equivalent\footnote
{
Convolution with $z^{-1}$ 
(resp.\ $\zbar^{-1}$) is a formal inverse to 
$\dee_{\zbar}$ 
(resp.\ $\dee_z$). See section \ref{DSII:subsec.prelim}, equations \eqref{DSII:solid-Cauchy} and \eqref{DSII:conjugate-solid-Cauchy} and the accompanying discussion and references.
} 
to a system of integral equations:
\begin{equation}
\label{DSII:M.int}
\left\{
\begin{aligned}
m^1 (z,k)	&=	1	+	 \frac{1}{\pi} \int_{\C} \frac{1}{z-w} q(w) m^2 (w,k) \, dw\\
m^2 (z,k)	&=			\frac{1}{\pi} \int_{\C} \frac{e_k(z-w)}{\zbar - \wbar} \overline{q(w)} m^1 (z,k) \, dw
\end{aligned}
\right.
%}
\end{equation}
where
\begin{equation}
\label{DSII:ek}
e_k(z) = e^{i(kz + \kbar \zbar)}. 
\end{equation}
For $q \in \scrS(\R^2)$, $m^1 $ and $m^2 $ admit large-$z$ expansions of the form
\begin{align*}
m^1 (z,k)	&\sim	1 + 	\sum_{j\geq 1}	\frac{a_j(k)}{z^j},\\
m^2 (z,k)	&\sim			e_{-k}(z) \sum_{j \geq 1} 	\frac{b_j(k)}{\zbar^j}.
\end{align*}
The scattering transform $\calS q$ of $q \in \scrS(\R^2)$ is $-ib_1(k)$. From the integral equations \eqref{DSII:M.int} we can see that
\begin{equation}
\label{DSII:S}
\left( \calS q \right)(k)	=	-\frac{i}{\pi}	\int_{\C} e_k(z) \overline{q(z)} m^1 (z,k) \, dz.
\end{equation}
This map is a perturbation of the antilinear `Fourier transform'
$$ \left( \calF_a q \right) = -\frac{i}{\pi} \int_{\C} e_k(z) \overline{q(z)} \, dz $$
which satisfies $\calF_a \circ \calF_a = I$. We will see that, remarkably, the same holds for $\calS$. The scattering transform $\calS$ linearizes the DSII equation \eqref{DSII} in the following sense: if $q(z,t)$ solves \eqref{DSII} and $q(\dotarg,t) \in \calS(\R^2)$ for each $t$, then
$$ \calS(q(\dotarg,t))(k) = e^{2i(k^2 + \kbar^2) t} \calS(q_0))(k).$$
Thus, a putative solution by inverse scattering is given by
\begin{equation}
\label{DSII:IST-sol}
q(z,t) 	=	\calS^{-1} \left( e^{2i((\dotarg)^2 + (\overline{\dotarg})^2)} \left(\calS q_0 \right)(k) \right)(z).
\end{equation}

To implement the solution formula \eqref{DSII:IST-sol}, we compute the scattering transform $\bfs(k)=\calS(q_0)(k)$ and, for each $t$, solve the system
\begin{equation}
\label{DSII:M.dual}
\left\{
\begin{aligned}
\dee_{\kbar} n^1 (z,k,t)				&=	\bfs(k) e^{2i(k^2+\kbar^2)t} n^2 (z,k,t)	\\
\left( \dee_k + iz \right) n^2 (z,k,t)	&=	\overline{\bfs(k)} e^{-2i(k^2+\kbar^2)t} n^1 (z,k,t).
\end{aligned}
\right.
%}
\end{equation}
The solution $q(z,t)$ is given by
\begin{equation}
\label{DSII.recon}
q(z,t) = -\frac{i}{\pi} \int_\C e^{it\varphi} \overline{\bfs(k)} n^1 (z,k,t) \, dk.
\end{equation}
where
$$ \varphi(k;z,t) = 2\left(k^2 + \kbar^2\right) - \frac{kz+\kbar \zbar}{t}. $$

The results of Beals-Coifman, Fokas-Ablowitz, and Sung imply the following analogue of Theorem \ref{thm:NLS.BC}.

\begin{theorem}
\label{DSII:Sung}
Suppose that $q_0 \in \calS(\R^2)$. Then $\calS q_0 \in \calS(\R^2)$. Moreover, the system \eqref{DSII:M.dual} has a unique solution for each $(z,t)$, and \eqref{DSII.recon} defines a classical solution of the Cauchy problem \ref{DSII}.
\end{theorem}

Nachman, Regev, and Tataru \cite{NRT:2017} proved the following remarkable result on the scattering transform $\calS$. Recall that the Hardy-Littlewood maximal function of $f \in L^p(\R^n)$, $1 \leq p \leq \infty$, is given by
$$ \left( \calM f \right)(x) = \sup_{r>0} \frac{1}{|B(x,r)|} \int_{B(x,r)} |f(y)| \, dy, $$
where $B(x,r)$ denotes the ball of radius $r$ about $x \in \R^n$ and $|A|$ denotes the Lebesgue measure of the measurable set $A \subset \R^n$. For $q \in L^2(\R^2)$, set
$$ \widehat{q}(k) = \frac{1}{\pi} \int e_k(z) q(z) \, dz. $$

\begin{theorem}\cite{NRT:2017}
\label{DSII:NRT} The scattering transform $\calS$ extends to a diffeomorphism from $L^2(\R^2)$ onto itself with $\calS \circ \calS = I$. Moreover, the following estimates hold:
\begin{enumerate}
\item[(i)]		$\norm[L^2(\R^2)]{\calS q} = \norm[L^2(\R^2)]{q}$
\smallskip
\item[(ii)]	$\left| \left( \calS(q) \right) (k) \right| \leq C\left(\norm[L^2]{q} \right) \left| \calM \widehat{q}(k) \right|$
\end{enumerate}
\end{theorem}

Theorem \ref{DSII:NRT} considerably extends earlier work of Brown \cite{Brown:2001} and Perry \cite{Perry:2016}, who considered the scattering map respectively for small data in $L^2(\R^2)$ and data in a weighted space $H^{1,1}(\R^2)$ analogous to the space $H^{1,1}(\R)$ for the NLS. It also illuminates  other work of Astala-Faraco-Rogers \cite{AFR:2015} and Brown-Ott-Perry \cite{BOP:2016} on the Fourier-like mapping properties of $\calS$.
The maximal function estimate is particularly important for the analysis of scattering since it implies that the solution of DSII by inverse scattering is bounded \emph{pointwise} by a maximal function for the solution of the linear problem. This means, for example, that Strichartz-type estimates for the linear problem imply Strichartz-type estimates for the nonlinear problem. 

As a consequence of Theorems \ref{DSII:Sung} and \ref{DSII:NRT}, Nachman, Regev, and Tataru obtain a complete characterization of the dynamics for DSII.  Denote by $U(t)$ the (nonlinear) solution operator for \eqref{DSII}, and by $V(t)$ the solution operator for the linearization of \eqref{DSII} at $q=0$, i.e.,
\begin{equation}
\label{DSII:lin}
\left\{
\begin{aligned}
v_t + 2\left( \dee_z^2 + \dee_{\zbar}^2 \right) v &= 0,	\\
v(z,0)	&=	v_0(z).
\end{aligned}
\right.
%}
\end{equation}

\begin{theorem} \cite{NRT:2017}
The Cauchy problem for \eqref{DSII} is globally well-posed in $L^2(\R^2)$ with $$\norm[L^2(\R^2)]{q(t)} = \norm[L^2(\R^2)]{q(0)}$$ for all $t$. Moreover, all solutions scatter in the sense that, for any $q_0 \in L^2(\R^2)$, there is a function 
$v_0 \in L^2(\R^2)$ so that
$$ \lim_{t \to \pm \infty} \norm[L^2]{U(t)q_0 - V(t) v_0} = 0. $$
The function $v_0$ is given by
$$ v_0 = \calF_a \calS q_0. $$
\end{theorem}

Note that the scattering is \emph{trivial} because the $t \to -\infty$ (past) and $t \to +\infty$ (future) asymptotes are the same.

Perry \cite{Perry:2016} obtained pointwise asymptotics under somewhat more restrictive conditions on the initial data.  Let
\begin{equation}
\label{DSII:H11}
 H^{1,1}(\R^2) = \left\{ q \in L^2(\R^2): \nabla q, |\dotarg| q(\dotarg) \in L^2(\R^2) \right\}. 
\end{equation}

\begin{theorem}
Suppose that $q_0 \in L^1(\R^2) \cap H^{1,1}(\R^2)$ and that $\left(\calS q_0\right)(0) \neq 0$. Then
$$ q(x,t) \sim v(z,t) + \littleo{t^{-1}} $$
where $v(z,t)$ solves the linearized equation \eqref{DSII:lin} with initial data $v_0 = \calF_a \calS q_0$. 
\end{theorem}

\subsection*{Exercises for Lecture 1}

In the following exercises, the Fourier transforms $\calF$ and $\calF^{-1}$ are defined by
\begin{align}
\label{Fourier}
\left(\calF f \right)(\xi) 			&= \int e^{-ix\xi} f(x) \, dx	\\
\label{Fourier.inverse}
\left(\calF^{-1} g \right)(x) 	&=	\frac{1}{2\pi} \int e^{ix\xi} g(\xi) \, d\xi.
\end{align} 

\medskip

\begin{exercise}
\label{ex:Fourier.conv}
Show that, with the conventions \eqref{Fourier} and \eqref{Fourier.inverse}, 
$$ \calF (f*g) (\xi) = \left( \calF f \right)(\xi) \left(\calF g \right)(\xi). $$
\end{exercise}

\begin{exercise}
\label{ex:Fourier.gaussian}
Suppose that $f(x) = e^{-zx^2}$ for some $z$ with $\real z >0$. Show that
$$
\left( \calF f \right)(\xi) =
        \sqrt{\frac{\pi}{z}} e^{-\xi^2/4z}.
$$
Use the formula
$$ \int_{-\infty}^\infty e^{-zx^2} \, dx = \sqrt{\frac{\pi}{z}} $$
and made a contour shift in the integration.
\end{exercise}

\begin{exercise}
\label{ex:Schrodinger.prop1}
The distribution inverse Fourier transform of $e^{\pm it \xi^2}$ may be computed as
$$ 
\calF^{-1} \left( e^{\pm it \xi^2} \right) = \lim_{\eps \to 0^+} \calF^{-1} \left( e^{\pm it \xi^2} e^{-\eps \xi^2} \right). 
$$
Using the result of Exercise \ref{ex:Fourier.gaussian}, show that
$$
\calF^{-1} \left( e^{\pm it \xi^2} \right) = \frac{1}{\sqrt{\mp 4\pi i t}}e^{\mp ix^2/(4t)} 
$$
where we take the principal branch of the square root function.
\end{exercise}

\begin{exercise}
Use the result of Exercise \ref{ex:Schrodinger.prop1} and the convolution theorem from Exercise \ref{ex:Fourier.conv} to show from the solution formula \eqref{LSE:Fourier-sol} that
$$ q(x,t) = \frac{1}{\sqrt{4\pi i t}} \int_{-\infty}^\infty e^{i(x-y)^2/(4t)} q_0(y) \, dy $$
for $q_0 \in \scrS(\R)$.
\end{exercise}

\begin{exercise}
\label{ex:NLS.Lax.pre}
Suppose that $\psi$ is a twice continuously differentiable, $N \times N$ matrix-valued solution to the system
\begin{align*}
\psi_x	&=	A(x,t) \psi\\
\psi_t		&=	B(x,t) \psi
\end{align*}
where $A(x,t)$ and $B(x,t)$ are continuously differentiable $N \times N$ matrix-valued functions of $x,t$. Suppose further that $\det \psi(x,t) \neq 0$ for all $(x,t)$. 
Show that 
$$ A_t  - B_x + [A,B] = 0. $$
\emph{Hint}: cross-differentiate the equations and use the equality $\psi_{tx} = \psi_{tx}$ (Clairaut's Theorem).
\end{exercise}

\begin{exercise}
A \emph{fundamental} solution of \eqref{NLS:Lax} is a twice-differentiable $2\times 2$ matrix-valued solution $\psi(x,t)$ with $\det \psi(x,t) \neq 0$ for all $(x,t)$. 
Using the result of Exercise \ref{ex:NLS.Lax.pre}, show that if \eqref{NLS:Lax} admits a fundamental solution for a given smooth function $q(x,t)$, then $q(x,t)$ solves \eqref{NLS}.
\end{exercise}

\begin{exercise}
\label{ex:ZS-AKNS.sa}
Let $\calL$ be the operator \eqref{ZS-AKNS}. Show that, for any smooth, compactly supported, $2 \times 2$ matrix-valued functions $\psi(x)$ and $\varphi(x)$, the identity $(\psi, \calL \varphi) = (\calL \psi, \varphi)$ holds, where the inner product is defined by
$$ (\psi,\phi) = \int_{\R} \Tr \left( \psi^*(x) \phi(x) \right) \, dx . $$
\end{exercise}

\begin{exercise}
\label{ex:NLS.Gerard}
Consider the alternative Lax representation (from the original paper of Zakharov and Shabat \cite{ZS:1972})
$$
\begin{aligned}
L	&= \twomat{i\dee_x}{q}{\qbar}{-i\dee_x}	
\\[5pt]
B	&=	\twomat	{2i \dee_x^2 - i|q|^2}{q_x +2q \dee_x}
						{\overline{q_x} +2\qbar \dee_x}{-2i\dee_x^2 + i|q|^2}
\end{aligned}
$$
Show that \eqref{NLS} is equivalent to the operator identity
$$ \dot{L} = [B,L]. $$
\emph{Remark}: The operator $L$ is formally self-adjoint and $B$ is formally skew-adjoint. This structure corresponds to the Lax representation for KdV.
\end{exercise}

\begin{exercise}
\label{ex:DSII.crossdiff}
Suppose given a family of smooth solutions $\psi_1(z,k,t), \psi_2(z,k,t)$ of \eqref{DSII:Lax.z}--\eqref{DSII:Lax.t}, indexed by $k \in \C$,
so that\footnote{These conditions are motivated by what one can actually prove about the solutions $m^1(z,k,t)=e^{-ikz+ik^2 t} \psi_1(z,k,t)$ and $m^2(z,k,t)=e^{-ikz+ik^2 t} m^2(z,k,t)$!}
\begin{itemize}
\item[(i)] $\lim_{k \to \infty} e^{-ikz + ik^2t} \psi_1(z,t,k) = 1$, 
\item[(ii)] $\lim_{k \to \infty} e^{-ikz+ik^2 t} \psi_2(z,t,k) = 0$, 
\item[(iii)] for each $(t,z)$, $\psi_2(z,k,t) \neq 0$ for at least one $k$.
\end{itemize}
Cross-differentiate the first equations of \eqref{DSII:Lax.z} and \eqref{DSII:Lax.t} and equate mixed partials to show that
\begin{multline}
\label{DSII:crossdiff}
-2i \eps  \left( q \dee_z \qbar\right) \psi_1 
	+  \left(\dot{q}-i\gbar q \right) \psi_2 = \\ 
			2i \left(2\eps \qbar \dee_z q + i\eps  q \dee_z \qbar+ \frac{1}{2}\dee_{\zbar} g \right) \psi_1 
	 		+	2i \left(\dee_z^2 q + \dee_{\zbar}^2 q + \frac{1}{2}gq \right) \psi_2
\end{multline}
Conclude  that the compatibility condition \eqref{DSII} holds. To be really thorough (!), you should check that cross-differentiating the second equations of \eqref{DSII:Lax.z}--\eqref{DSII:Lax.t} gives the same relation.

\emph{Hint}: use \emph{both} equations \eqref{DSII:Lax.z} to eliminate $\zbar$-derivatives of $\psi_1$ and $z$-derivatives of $\psi_2$. Expressions involving `irreducible' derivatives such as $\dee_z \psi_1$ and $\dee_{\zbar} \psi_2$ should cancel, leading to \eqref{DSII:crossdiff}. Then use the asymptotic conditions to argue that the coefficients of $\psi_1$ and $\psi_2$ must both be zero.
\end{exercise}

\newpage
\section{The Defocussing Cubic Nonlinear Schr\"{o}dinger Equation}

This lecture largely follows the analysis of Deift-Zhou \cite[esp.\ \S 3]{DZ:2003} with a few inessential changes. We will analyze the direct and inverse scattering maps for NLS and, for completeness, give a proof of Beals-Coifman's result that the solution formula via inverse scattering generates a classical solution of the defocussing NLS equation \eqref{NLS} if $q_0 \in \scrS(\R)$. 

We will solve the NLS equation in the sense that we find a solution of the integral equation
\begin{equation}
\label{NLS:int}
q(t) = e^{it\Delta}q_0 - i \int_0^t e^{i(t-s)\Delta}\left( 2 |q(s)|^2 q(s) \right) \, ds
\end{equation}
on $H^1(\R)$, where $e^{it\Delta}$ is the solution operator \eqref{LSE:SO} for the linear Schr\"{o}dinger equation. Here
\begin{equation}
\label{NLS:H1}
H^1(\R) = \left\{ u \in L^2(\R):  u' \in L^2(\R) \right\}.
\end{equation}
Although \eqref{NLS:int} can be solved in much weaker spaces (see, for example \cite{LP:2015} or \cite[Chapter 3]{Tao:2006}), the space  $H^1(\R)$ will serve our purpose of showing that the inverse scattering method produces a continuous solution map on $H^{1,1}(\R)$. The following lemma shows that, to show that \eqref{NLS:IST-sol} solves \eqref{NLS:int}, it suffices to show that \eqref{NLS:IST-sol} produces a classical solution of \eqref{NLS} for initial data in $\scrS(\R)$.

\begin{lemma}
\label{lemma:NLS.approx}
Let $q_0 \in H^{1}(\R)$ and suppose that $\{ q_n \}$ is a sequence from $\scrS(\R)$ with $\norm[H^1]{q_n - q_0} \to 0 $ as $n \to \infty$. Suppose that $q_n(z,t)$ solves \eqref{NLS:int} with initial data $q_n$ and that $q_n(z,t) \to q(z,t)$ in the sense that $\sup_{t \in (0,T)} \norm[H^1]{q_n(\dotarg,t)-q(\dotarg,t)} \to 0$ as $n \to \infty$. Then $q(z,t)$ solves \eqref{NLS:int} with initial data $q_0$. 
\end{lemma}

We leave the proof as Exercise \ref{ex:NLS.approx}.

\subsection{The Direct Scattering Map}
\label{NLS:sec.direct}

In this subsection we'll construct the direct scattering map by studying solutions $\Psi^\pm$ of the problem
$\calL \psi = \lam \psi$.  Here $\calL$ is the ZS-AKNS operator \eqref{ZS-AKNS}, $\psi$ is $2 \times 2$ matrix-valued, 
$\lam \in \R$, and $\Psi^\pm$ satisfy the asymptotic conditions
$$ \lim_{x \to \pm \infty} \Psi^\pm(x,\lam)e^{i\lam x \sigma_3} = \I.$$
It is well-known that the Jost solutions exist and are unique for $q \in L^1(\R)$, and that $\det\Psi^\pm(x,\lam) = 1$. 

We begin with some reductions. A straightforward computation shows 
that for any $z \in \C$, the solution space of $\calL\psi = z \psi$ is invariant under the mapping
\begin{equation}
\label{NLS:AKNS-sym}
 \psi(x,z) \mapsto \sigma_1 \overline{\psi(x,\zbar)} \sigma_1^{-1}, \quad 
\sig_1 = 
\begin{pmatrix}
0	&	1	\\
1	&	0
\end{pmatrix}
\end{equation}
(Exercise \ref{ex:NLS.sigma1}).
From this symmetry and  the uniqueness of Jost solutions,  it follows that the matrix-valued Jost solutions take the form
\begin{equation}
\label{Jost.sym}
\Psi(x,\lam) = 
\begin{pmatrix}
\Psi_{11}(x,\lam)	&	\overline{\Psi_{21}(x,\lam)}	\\
\Psi_{21}(x,\lam)	&	\overline{\Psi_{11}(x,\lam)}
\end{pmatrix}
\end{equation}
and that the matrix $T(\lam)$ defined in \eqref{NLS:T} takes the form \eqref{NLS:T.sym}.
From the relation $|a(\lam)|^2 - |b(\lam)|^2 = 1$ it follows that $|a(\lam)| \geq 1$, and that
$$ r(\lam)  = -b(\lam)/\overline{a(\lam)} $$
is a well-defined function with $|r(\lam)| < 1$. We will prove:

\begin{theorem}
\label{thm:NLS.direct}
The map $q \mapsto r$ is locally Lipschitz continuous from $H^{1,1}(\R)$ to $H^{1,1}_1(\R)$. 
\end{theorem}

The approach we'll take here is inspired by the analysis of the scattering transform by Muscalu, Thiele, and Tao in \cite{MTT:2003}, which also contains an interesting discussion of the Fourier-like mapping properties of the scattering transform. In order to obtain effective formulas for the scattering data $a(\lam)$ and $b(\lam)$, we make the change of variables
\begin{equation}
\label{NLS:N}
\Psi^+(x,\lam) = e^{-i\lam x \sig_3} \bfN(x,\lam).
\end{equation}
It follows from the equation $\calL \Psi^+ = \lam \Psi^+$ that 
\begin{equation}
\label{NLS:N.de}
\left\{
\begin{aligned}
\frac{d}{dx} \bfN(x,\lam) &= 
\begin{pmatrix}
0	&	e^{2i\lam x} q(x) \\
e^{-2i\lam x} \overline{q(x)}	&	0
\end{pmatrix}
\bfN(x,\lam)\\[5pt]
\lim_{x \to +\infty} \bfN(x,\lam) = \I
\end{aligned}
\right.
%}
\end{equation}
while, by \eqref{NLS:T},
\begin{equation}
\label{NLS:ab}
\lim_{x \to -\infty} \bfN(x,\lam) = T(\lam).
\end{equation}
By the symmetry \eqref{Jost.sym}, we have 
$$ 
\bfN (x,\lam) = 
\begin{pmatrix}
N_{11}(x,\lam)	&	\overline{N_{21}(x,\lam)}\\[5pt]
N_{21}(x,\lam)	&	\overline{N_{11}(x,\lam)}
\end{pmatrix}
$$
so it suffices to construct $N_{11}$ and $N_{21}$. 
Equation \eqref{NLS:N.de} is equivalent to the integral equation
$$
\bfN(x,\lam)	=	\I - 
		\int_x^\infty 
				\begin{pmatrix}
					0	&	e^{2i\lam y} q(y) \\
					e^{-2i\lam y} \overline{q(y)}	&	0
				\end{pmatrix}
				\bfN(y,\lam) \, dy
$$
which has a convergent Volterra series solution for $q \in L^1(\R)$. Indeed, setting
$$ N_{11}(x,\lam) = a(x,\lam), \quad N_{21}(x,\lam) = b(x,\lam), $$
we have
\begin{align}
\label{NLS:asol}
a(x,\lam)	&=	1 + \sum_{n=1}^\infty A_{2n}(x,\lam), \\
\label{NLS:bsol}
b(x,\lam)	&=	-\sum_{n=0}^\infty A_{2n+1}(x,\lam).
\end{align}
Here
$$
A_n(x,\lam) = 
	\int_{x < y_1 < y_2 < \ldots < y_n} 
			Q_n(y_1,\ldots,y_n) 
			e^{2i\lam \phi_n(y)} 
		\, dy_n \, \ldots dy_1 
$$
where
$$
Q_n(y_1,\ldots,y_n) = 
\begin{cases}
\prod_{j=1}^m q(y_{2j-1}) \overline{q(y_{2j})}, &	n=2m,\\
\\
\overline{q(y_1)} \prod_{j=1}^{m} q(y_{2j}) \overline{q(y_{2j+1})}, 	&	n=2m+1
\end{cases}
$$
and (with the convention that $\phi_0 = 0$)
\begin{align*}
\phi_{2m}(y_1,\ldots,y_{2m})			&=	\sum_{j=1}^m \left(y_{2j-1}-y_{2j}\right),\\
\phi_{2m+1}(y_1,\ldots,y_{2m+1})	&=	-y_1+\phi_{2m}(y_2,\ldots,y_{2m+1})\\
												&=	-\phi_{2m}(y_1,\dots,y_{2m})- y_{2m+1}
\end{align*}
The bound
$$
\int_{y_1 < y_2 \ldots < y_n} \left| Q_n(y) \right| \, dy_n \, dy_{n-1} \, \ldots \, dy_1
\leq \frac{\norm[1]{q}^n}{n!}
$$
shows that the Volterra series converge uniformly in $x \in \R$ and $q$ in bounded subsets of $L^1(\R)$. By \eqref{NLS:ab} and dominated convergence, we obtain the following  representations of the maps $q \mapsto a$ and $q \mapsto b$:
\begin{align}
\label{NLS.a.ser}
a(\lam)	&=	1+\sum_{n=1}^\infty A_{2n}(\lam)\\
\label{NLS.b.ser}
b(\lam)	&=	-\sum_{n=0}^\infty  A_{2n+1}(\lam)
\end{align}
where
\begin{equation}
\label{NLS:An}
A_n(\lam) =\int_{ y_1 < y_2 < \ldots < y_n} Q_n(y_1,\ldots,y_n) e^{2i\lam \phi_n(y)} \, dy_n \, \ldots dy_1.
\end{equation}

From this representation we obtain an  $L^1 \to L^\infty$ mapping property of the scattering transform.

\begin{proposition}
\label{prop:NLS.direct.1}
The map $q \mapsto r$ is locally Lipschitz continuous from $L^1(\R)$ to $L^\infty(\R)$. 
\end{proposition}

\begin{proof}
It suffices to show that  $q \mapsto a$ and $q \mapsto b$ are locally Lipschitz continuous.  If so, this continuity and the  lower bound $|a(\lam)| \geq 1$ imply local Lipschitz continuity of $q \mapsto r$. 
If $M_n: (L^1(\R))^n \to L^\infty(\R)$ is a multilinear map and $$F_n(q) = M_n(q, \ldots q, \qbar, \ldots \qbar)$$ with $m$ entries of $q$ and $n-m$ entries of $\qbar$, then
\begin{align*}
F_n(q_1) - F_n(q_2) 
	&= \sum_{j=1}^m M_n(q_2, \dots, q_2, \underbrace{q_1 - q_2}_{\text{$j$th entry}}, , q_1, \ldots q_1, \overline{q_1},\dots \overline{q_1}) \\
	&\quad + \sum_{j=m+1}^n M_n(q_2,\dots, q_2, \overline{q_2}, \ldots, \overline{q_2}, \underbrace{\overline{q_1-q_2}}_{\text{$j$th entry}}, \overline{q_1}, \ldots, \overline{q_1} )
\end{align*}
so that, setting 
$$\gamma = \max	\left(
									\norm[L^1]{q_1}, \norm[L^1]{q_2}
							\right),
$$ 
we have
$$ \norm[L^\infty]{F_n(q_1) - F_n(q_2)}
	\leq \norm[(L^1)^n \to L^\infty]{M_n} 
		n \gamma^{n-1}  \norm[L^1]{q_1-q_2}
$$
Thus, referring to \eqref{NLS:An}, we have
$$ 
\norm[L^\infty]{A_n(\lam;q_1) - A_n(\lam;q_2)} 
\leq 
\frac{1}{(n-1)!} 
	\gamma^{n-1}  \norm[L^1]{q_1-q_2}.
$$
We conclude that
\begin{align*}
\norm[L^\infty]{a(\dotarg; q_1)  -a(\dotarg,q_2)} & \leq e^\gamma \norm[L^1]{q_1-q_2}\\
\norm[L^\infty]{b(\dotarg;q_1) - b(\dotarg,q_2)} & \leq e^\gamma \norm[L^1]{q_1-q_2}.
\end{align*}
\end{proof}

\begin{remark}
\label{NLS:Direct.lin}
From the above analysis, it is easy to see that the Fr\'{e}ch\'{e}t derivative of $\calR$ at $q=0$ is the ``antilinear Fourier transform''
$$ \left( \calF_a q \right)(\lam) = -\int e^{-2ix\lam} \overline{q(x)} \, dx. $$
\end{remark}

With a bit more work, we can prove:

\begin{proposition}
\label{prop:NLS.direct.2}
The map $q \mapsto r$ is locally Lipschitz continuous from $L^1(\R) \cap L^2(\R)$ into $L^2(\R)$. 
\end{proposition}

\begin{proof}
In what follows we use the fact that
$$ \norm[L^2(\R)]{f} = \sup _{\varphi \in C_0^\infty(\R)} \left| \int \varphi(\lam) f(\lam) \, d\lam \right|.$$
It follows from \eqref{NLS:An} and the trivial inequalities
\begin{align*}
\int 	\left| 
				\widehat{\varphi}(\phi_{2n-1}(y)) q(y_{2n-1}) 
		\right| \, dy_{2n-1} 
		&\leq 
			\norm[2]{\varphi} \norm[2]{q},
\\
\int  	\left| 
				\widehat{\varphi}(\phi_{2n}(y)) q(y_{2n}) 
		\right| \, dy_{2n} 
		&\leq 
			\norm[2]{\varphi} \norm[2]{q},
\end{align*}
that
$$ \norm[L^2(\R)]{A_n} \leq \frac{\norm[L^1]{q}^{n-1}}{(n-1)!}  \norm[L^2]{q}. $$
Thus the power series representations for $a-1$ and $b$ converge for $X= L^1(\R) \cap L^2(\R)$, $Y=L^2(\R)$, showing that $q \mapsto a$ and $q \mapsto b$ are locally Lipschitz continuous as maps from $L^1(\R) \cap L^2(\R)$ into $L^2(\R)$. It now follows that $q \mapsto r$ has the same continuity.
\end{proof}

As in the theory of the Fourier transform, additional smoothness of $q$ implies additional decay of $r$.

\begin{proposition}
\label{prop:NLS.direct.3}
The map $q \mapsto r$ is locally Lipschitz continuous from $H^{1,1}(\R)$ to $H^{0,1}(\R)$. 
\end{proposition}

\begin{proof}
It suffices to exhibit an $L^2$-convergent power series for $\lam b(\lam)$. We will assume for the moment that $q \in \scrS(\R)$ and begin with the formula
$$\lam b(\lam) = \sum_{n=1}^\infty \lam A_{2n-1}(\lam). $$
Using the integration by parts identity
\begin{multline*}
 \int_{y_{2n-2}}^\infty \overline{q(z)} (2i\lam) e^{2i\lam (-\phi_{2n-2}(y)+z)} \, dz
	=	\\
	\left. \left[ \overline{q(z)} e^{2i\lam(-\phi_{2n-2}(y)+z)} \right] \right|_{y_{2n-2}}^\infty
		-	\int_{y_{2n-2}}^\infty \overline{q'(z)}  e^{2i\lam( -\phi_{2n-2}(y)+z)} \, dz
\end{multline*}
we conclude that
$$ \int \lam b(\lam)  \varphi(\lam) \, d\lam 
	= 	\int \varphi(\lam) {\widehat{\overline{q}} \,}'(\lam) \, d\lam +
		\sum_{n=2}^\infty \left( I_{1,2n-1} + I_{2,2n} \right),
$$
where, for $n \geq 2$, 
\begin{multline*}
\left| I_{1,2n-1} \right|	
		\leq	\\
	\int_{y_1 < \ldots < y_{2n-2}}
						\left( \prod_{k=1}^{2n-2} |q(y_j)| \right) \, 
						\left| 
							\widehat{\varphi}(\phi_{2n-3}(y) ) 
						\right| 
						|q(y_{2n-2})|
					\, dy_{2n-2} \ldots dy_{1}
\end{multline*}
and
\begin{multline*}
\left| I_{2,2n-1} \right| 
	\leq \\
	\int_{y_1 < \ldots < y_{2n-1}}
						\left( \prod_{k=1}^{2n-2} |q(y_j)| \right) \, 
						\left| \widehat{\varphi}(\phi_{2n-1}(y)) \right|
						|q'(y_{2n-1})| \,
					\, 	dy_{2n-1} \ldots dy_1.
\end{multline*}
It follows that 
\begin{align*}
\left|I_{1,2n-1}\right|
		&\leq	\frac{1}{(2n-3)!} 
				\norm[L^1]{q}^{2n-3} 
				\norm[L^\infty]{q} 
				\norm[L^2]{q} 
				\norm[L^2]{\varphi} 	\\
\left|I_{2,2n-1}\right|
		&\leq\frac{1}{(2n-2)!} \norm[L^1]{q}^{2n-2} \norm[L^2]{q'}\norm[L^2]{\varphi}
\end{align*}
Taking suprema over $\varphi$ with $\norm[L^2]{\varphi}=1$ and recalling that
$\norm[L^\infty]{q} \lesssim \norm[H^1]{q}$ (Exercise \ref{ex:H10}), we recover the estimate
$$ \norm[L^2]{(\dotarg)b(\dotarg)}
\leq
\norm[H^{1,0}]{q} + 
	\sum_{n=2}^\infty 
			\frac{1}{(2n-3)!} 
				\norm[L^1]{q}^{2n-3} 
				\left( 
					\norm[L^1]{q}+ 
					\norm[L^2]{q} 
				\right) 
				\norm[H^{1,0}]{q}
$$
which is finite because $\norm[L^1]{q} \lesssim \norm[H^{1,1}]{q}$ and the series
$$ \sum_{n=2}^\infty \frac{x^{2n-3}}{(2n-3)!} = \sinh(x)$$
converges for all $x$.  The local Lipschitz continuity is proven by continuity of multilinear functionals as in Proposition \ref{prop:NLS.direct.1}.
\end{proof}

Finally:

\begin{proposition}
\label{prop:NLS.direct.4}
The map $q \mapsto r'$ is locally Lipschitz continuous from $H^{0,1}(\R)$ to $L^2(\R)$. 
\end{proposition}  

\begin{proof}
By the quotient rule and the lower bound on $|a|$, it suffices to show that the map $q \mapsto (a',b')$ is locally Lipschitz continuous from $H^{0,1}(\R)$ to $L^2(\R) \times L^2(\R)$
From \eqref{NLS.a.ser}--\eqref{NLS.b.ser} we have
\begin{align*}
-i \frac{\dee a}{\dee \lam} (\lam)
	&=	\sum_{n=1}^\infty \int_{y_1 < \ldots < y_{2n}} 
					2Q_{2n}(y) \phi_{2n}(y) e^{2i\lam \phi_{2n}(y)} \, dy\\
-i \frac{\dee b}{\dee \lam} (\lam)
	&=	\sum_{n=1}^\infty \int_{y_1 < \ldots < y_{2n-1}}
					2Q_{2n-1}(y) \phi_{2n-1}(y) e^{2i\lam \phi_{2n-1}(y)} \, dy
\end{align*}
so integrating against $\varphi \in C_0^\infty(\R)$ we obtain 
\begin{align}
\label{aprime}
\left|	\int \varphi \, \frac{\dee a}{\dee \lam} \, d\lam	\right|
	&\leq	\sum_{n=1}^\infty 
			\int_{y_1 < \ldots < y_{2n}}
					\left|Q_{2n}(y) \right| \, 
					\left| \phi_{2n}(y) \right| \,
					\left| \widehat{\varphi}(\phi_{2n}(y)) \right| 
			\, dy\\
\nonumber
	&\leq		
			\sum_{n=1}^\infty 
			\int_{y_1 < \ldots < y_{2n}}
					\sum_{j=1}^{2n} 
							|y_j| \, 
							\left|Q_{2n}(y) \right| \, 
							\left| \widehat{\varphi}(\phi_{2n}(y)) \right| 
			\, dy\\[10pt]
\label{bprime}
\left|	\int \varphi \, \frac{\dee b}{\dee \lam} \, d\lam	\right|	
	&\leq		
			\sum_{n=1}^\infty
			\int_{y_1 < \ldots < y_{2n-1}}
					\left| Q_{2n-1}(y) \right| \, \left| \phi_{2n-1}(y) \right| \, 
					\left| \widehat{\varphi} (\phi_{2n-1}(y)) \right| \,
			dy\\
\nonumber
	&\leq	\sum_{n=1}^\infty
			\int_{y_1 < \ldots < y_{2n-1}}
					\sum_{j=1}^{2n-1} |y_j|
					\left| Q_{2n-1}(y) \right| \,
					\left| \widehat{\varphi} (\phi_{2n-1}(y)) \right| \,
			dy
\end{align}
As before, we will bound the left-hand integrals in \eqref{aprime} and \eqref{bprime}
by norms of $q$ times $\norm[L^2]{\varphi}$ and take the supremum over $\varphi$ with $\norm[L^2]{\varphi}=1$.

To bound the right-hand side of \eqref{aprime}, first note that the integrand is symmetric under interchange of $(n-1)$ pairs of indices (the pair containing $y_j$ is excluded), so we may write
\begin{multline*}
\left| \int \varphi \frac{\dee a}{\dee \lam} \, d\lam \right|
\leq	\\
\sum_{n=1}^\infty \frac{1}{(n-1)!}  
	\sum_{j=1}^{2n}
		\int_{\R^{2n}}
				\left(\prod_{k=1}^{j-1} |q(y_k)|\right) \, 
				\left( |y_j q(y_j)| \right) \, 
				\left(\prod_{k=j+1}^{2n} |q(y_k)|\right)
			|\widehat{\varphi}(\phi_{2n}(y))| \, dy
\end{multline*}
Using Young's inequality $\norm[2]{f*g} \leq \norm[1]{f} \norm[2]{g}$ repeatedly beginning with the $y_{2n}$ integration, we have
$$ \norm[L^2(dy_j)]
{
\int_{\R^{2n-j}} 
	|\widehat{\varphi}(\phi_{2n}(y))| \, 
	\left( \prod_{k=j+1}^{2n} |q(y_k)| \right) \, dy_{2n} \, \ldots \, dy_{j+1}
} 
\leq 
\norm[L^1]{q}^{2n-j} \norm[L^2]{\varphi}
$$
so that
$$
\norm[L^2]{\frac{\dee a}{\dee \lam}}
	\leq	
	\sum_{j=1}^\infty 
			\frac{2n}{(n-1)!} 
				\norm[H^{0,1}]{q} 
				\norm[L^1]{q}^{2n-1}
$$
which is uniformly bounded for $q$ in bounded subsets of $H^{0,1}$ since
$$ \sum_{j=1}^\infty \frac{2n}{(n-1)!} x^{2n-1} $$
converges for all $x$.

To bound the right hand side of \eqref{bprime}, we first note that the $n=1$ term is trivially bounded 
by $\norm[H^{0,1}]{q} \norm[L^2]{\varphi}$. For $n \geq 2$, the integrand is symmetric under $(n-2)!$ interchanges of pairs so we may estimate the remaining terms on the right-hand side of \eqref{bprime} by
$$
\sum_{n=2}^\infty 
	\frac{1}{(n-2)!} 
		\sum_{j=1}^{2n-1} 
			\int_{\R^{2n-1}} 
				|y_j| \, |Q_{2n-1}(y)| \, |\widehat{\varphi}(\phi_{2n-1}(y))|
			\, dy
$$
Writing
\begin{multline*}
\int_{\R^{2n-1}} |y_j| \, |Q_{2n-1}(y)| \, |\widehat{\varphi}(\phi_{2n-1}(y))| \, dy=\\
\int_{\R^{2n-1}} 
	\left(\prod_{k=1}^{j-1} |q(y_k)| \right)
	\left( |y_j| |q(y_j)| \right)
	\left(\prod_{k=j+1}^{2n-1} |q(y_k)| \right)
	|\widehat{\varphi}(\phi_{2n-1}(y))|
	\, dy
\end{multline*}
we may use the estimate
$$
\norm[L^2(dy_j)]
{
\int_{\R^{2n-j-1}}
	\left(\prod_{k=j+1}^{2n-1} |q(y_k)| \right)
	|\widehat{\varphi}(\phi_{2n-1}(y))|
	\, dy
}
\leq
\norm[L^1]{q}^{2n-j-1} \norm[L^2]{\varphi}
$$
(which again follows by repeated applications of Young's inequality)
to conclude that
$$
\norm[L^2]{\frac{\dee b}{\dee \lam}}
\leq
\norm[H^{1,0}]{q} + \sum_{n=2}^\infty \frac{2n-1}{(n-2)!} \norm[L^1]{q}^{2n-2} \norm[H^{0,1}]{q}.
$$
The right-hand side is again bounded uniformly for $q$ in a bounded subset of $H^{0,1}(\R)$ since
the series
$$\sum_{n=2}^\infty \frac{2n-1}{(n-2)!} x^{2n-2}$$ 
converges for all $x$. 

We have shown that $\dee a/\dee \lam$ and $\dee b /\dee \lam$ have convergent  series representations. We can use multilinearity of the terms as in the proof of Proposition \ref{prop:NLS.direct.1} to obtain local Lipschitz continuity.
\end{proof}

\subsection{Beals-Coifman Solutions}
\label{NLS:sec.bc}

Beals and Coifman \cite{BC:1984} identified solutions of \eqref{NLS:M} which have piecewise analytic continuations to $\C \setminus \R$ and solve a Riemann-Hilbert problem determined completely by the scattering data. It follows from the definition of $\bfM$ that two nonsingular solutions $\bfM_1$ and $\bfM_2$ of \eqref{NLS:M} are related by 
\begin{equation}
\label{M1M2A}
{\bfM}_1(x,\lam) = {\bfM}_2(x,\lam) e^{-ix\lam \ad \sig_3} A 
\end{equation}
where $A$ is a constant matrix and
$$ e^{t\ad \sig_3} 
\begin{pmatrix}
a	&	b	\\
c	&	d
\end{pmatrix}
=
\begin{pmatrix}
a					&	e^{2t}b	\\
e^{-2t} c	&	d
\end{pmatrix}.
$$
(see Exercises \ref{ex:NLS.adsig} and \ref{ex:NLS.M1M2}).
We will use this fact repeatedly in what follows.

We now construct the Beals-Coifman solutions from solutions $\bfM^\pm$ of \eqref{NLS:M} corresponding to the Jost solutions $\Psi^\pm$. 
By the factorization $\psi(x,\lam) = \bfM(x,\lam) e^{-i\lam x \sigma_3}$ and the symmetry \eqref{Jost.sym},  solutions of \eqref{NLS:M} normalized by either of the two conditions $\lim_{x \to \pm \infty} \bfM(x,z) = \I$  take the form
$$\bfM(x,\lam) = 
\begin{pmatrix}
m_{11}(x,\lam) 	&	\overline{m_{21}(x,\lam)}	\\[5pt]
m_{21}(x,\lam)	&	\overline{m_{11}(x,\lam)}
\end{pmatrix}
$$
so we need only study $m_1 \coloneqq m_{11}$ and $m_2 \coloneqq m_{21}$. 
Morever, it is clear from \eqref{NLS:M} and \eqref{NLS:N} that
\begin{equation}
\label{NtoM}
\bfM(x,\lam)  =e^{-i\lam x \sig_3} \bfN(x,\lam) e^{i\lam x \sig_3}.
\end{equation} 

Let $\Psi^+(x,\lam) = \bfM^+(x,\lam) e^{-i\lam x \sig_3}$. It follows from \eqref{NtoM} that
\begin{align*}
m^+_1(x,\lam) \coloneqq m^+_{11}(x,\lam)	&=	 a(x,\lam)\\
m^+_2(x,\lam) \coloneqq m^+_{21}(x,\lam)	&=	e^{2i\lam x} \overline{b(x,\lam)}
\end{align*}
It follows from \eqref{NLS:asol}--\eqref{NLS:bsol} that
\begin{align}
\label{m1sol}
m^+_1(x,\lam)	
	&=	1	+ \sum_{n=1}^\infty \int_{x < y_1 < \ldots  < y_{2n}} Q_{2n}(y) e^{2i\lam \phi_{2n}(y_1,\ldots,y_{2n})} \, dy\\
\label{m2sol}
m^+_2(x,\lam)
	&=	- \sum_{n=1}^\infty \int_{x < y_1 < \ldots < y_{2n-1}} \overline{Q_{2n+1}(y)} e^{2i\lam(\phi_{2n}(x,y_1,\ldots,y_{2n-1})} \, dy
\end{align}
Since the phase functions
\begin{align*}
\phi_{2n}(y) 	&\coloneqq \phi_{2n}(y_1,\ldots, y_{2n})
\intertext{and} 
\phi_{2n}(x,y) 		&\coloneqq\phi_{2n}(x,y_1,\ldots,y_{2n-1})
\end{align*}
are nonpositive over their respective domains of integration, $m^+_1$ and $m^+_2$ continue to analytic functions $m^+_1(x,z)$ and $m^+_2(y,z)$
for $\imag z < 0$ obeying the bounds
$$
|m^+_1(x,z)-1| 	+ |m^+_2(x,z)| \leq	\gamma_+(x) e^{\gamma_+(x)} 
$$
where $$\gamma_+(x) = \int_x^\infty |q(y)| \, dy.$$ 
It follows that $a(\lam)$ also has a bounded analytic continuation to the lower half-plane which we denote by $a(z)$. Using \eqref{m1sol} and \eqref{m2sol} with $\lam$ replaced by $z$ we can deduce the large-$x$ asymptotics

$$
\lim_{x \to -\infty} 	\begin{pmatrix} m^+_1(x,z) \\[5pt] m^+_2(x,z) \end{pmatrix} =
							\begin{pmatrix} a(z) \\[5pt] 0 \end{pmatrix},
\quad
\lim_{x \to \infty} 		\begin{pmatrix} m^+_1(x,z) \\[5pt] m^+_2(x,z) \end{pmatrix} =
							\begin{pmatrix} 1 \\[5pt] 0 \end{pmatrix}
$$
for each fixed $z$ with $\imag z < 0$. 

We can also use the Volterra series to analyze the large-$z$ behavior of the extensions $m_1^+(x,z)$ and $m_2^+(x,z)$ for fixed $x$. Observe that
$$\left|e^{2iz\phi_{2n}(y)}\right| 
\leq e^{-2\imag(z) 
\phi_{2n}(y)}$$
and 
$$\left| e^{2iz(\phi_{2n}(x,y)} \right| \leq e^{-2\imag(z) \phi_{2n}(x,y)}.$$ 
An argument using the dominated convergence theorem together with the absolute and uniform convergence of the Volterra series for $m_1$ and $m_2$ shows that 
\begin{equation}
\label{NLS:m1m2.z.infty}
\lim_{z \to \infty}
\begin{pmatrix}
m_1^+(x,z)	\\[5pt] 	m_2^+(x,z)
\end{pmatrix}
=
\begin{pmatrix}
~1~	\\[5pt]	~0~
\end{pmatrix}
\end{equation}
and
$$
\lim_{z \to \infty} a(z) = 1 
$$
where the limit is taken as $|z| \to \infty$  in any proper subsector of the lower half-plane. 

In a similar way, if $\Psi^-(x,\lam) = \bfM^-(x,\lam)e^{-i\lam x \sig_3}$, we can use the Volterra series
\begin{align*}
m^-_1(x,\lam)	
	&=	1	+ \sum_{n=1}^\infty \int_{y_{2n}<  \ldots < y_1 < x} Q_{2n}(y) e^{2i\lam \phi_{2n}(y)} \, dy\\
m^-_2(x,\lam)
	&=	-\sum_{n=1}^\infty  \int_{y_{2n-1} < \ldots < y_1 < x} \overline{Q_{2n+1}(y)} e^{2i\lam \phi_{2n}(x,y)} \, dy
\end{align*}
and the fact that $\phi_{2n}(y)$ and $\phi_{2n}(x,y)$ are nonnegative on their respective domains of integration to show that 
$m_1^-(x,\lam)$ and $m_2^-(x,\lam)$ continue to analytic functions $m_1^-(x,z)$ and 
$m_2^-(x,z)$ for $\imag z > 0$ with
$$
|m^-_1(x,z)-1| 	+ |m^-_2(x,z)| \leq	\gamma_-(x) e^{\gamma_-(x)}
$$
where 
$$\gamma_-(x) = \int_{-\infty}^x |q(y)| \, dy. $$
We can deduce the asymptotics
\begin{equation}
\label{NLS:m1-m2-.z}
\lim_{x \to +\infty} 	\begin{pmatrix} m^-_1(x,z) \\[5pt] m^-_2(x,z) \end{pmatrix} =
							\begin{pmatrix} \overline{a(\zbar)} \\[5pt] 0 \end{pmatrix},
\quad
\lim_{x \to -\infty} 	\begin{pmatrix} m^-_1(x,z) \\[5pt] m^-_2(x,z) \end{pmatrix} =
							\begin{pmatrix} 1 \\[5pt] 0 \end{pmatrix}.
\end{equation}

It will be important to know that $a(z)$ has no zeros in $\imag z<0$. It follows from \eqref{NLS:T} that
$$ 
a(\lam) = 
\begin{vmatrix}
\Psi_{11}^+(x,\lam)	&	\Psi_{12}^-(x,\lam)\\[5pt]
\Psi_{21}^+(x,\lam)	&	\Psi_{22}^-(x,\lam)
\end{vmatrix}
$$
so the same holds true for $\lam$ replaced by $z$ with $\imag z<0$ by analytic continuation. Thus 
$a(z)=0$ if and only if the columns are linearly dependent. Since $(\Psi_{11}^+,\Psi^+_{21})$ decay exponentially as $x \to +\infty$ and $(\Psi_{12}^-,\Psi_{22}^-)$ decay exponentially as $x \to -\infty$, it is easy to show that this condition leads to a square-integrable solution of $\calL \psi = z \psi$ with imaginary eigenvalue $z$, which is forbidden by the self-adjointness of the operator $\calL$ in \eqref{ZS-AKNS} (see Exercise \ref{ex:ZS-AKNS.sa}). Hence $a(z)$ has no zeros in $\imag z<0$.

\bigskip

We can now construct piecewise analytic solutions of \eqref{NLS:M}, normalized so
that $\bfM^r(x,z) \to \I$ as $x \to +\infty$ (the ``$r$'' is for ``right-normalized''), by the formulas
\begin{align}
\label{BC.right.form}
\bfM^{r}(x,z) &= 
\begin{pmatrix}
	m_1^-(x,z)	&	\overline{m_2^+(x,\zbar)}\\[5pt]
	m_2^-(x,z)	&	\overline{m_1^+(x,\zbar)}
\end{pmatrix}
\begin{pmatrix}
1/\overline{a(\zbar)}	&	0	\\[5pt]	
0								& 	1
\end{pmatrix},
&	\imag z > 0
\intertext{
(recall the first asymptotic relation of \eqref{NLS:m1-m2-.z})
and}
\label{BC.right.extend}
\bfM^r(x,z) &= \sig_1 \overline{\bfM^r(x,\zbar)} \sig_1, 	&	\imag z < 0
\end{align}
(recall \eqref{NLS:AKNS-sym}).
The piecewise analytic function $\bfM^r(x,z)$ admits boundary values $\bfM^r_\pm(x,\lam)$ 
as $\pm \imag z \to 0$ and $\real z \to \lam$.  

Since $\bfM^r_\pm(x,\lam)$ solve \eqref{NLS:N.de}, it follows from \eqref{M1M2A} that there is a jump matrix $\bfV_r(\lam)$ with 
$$
\bfM^r_+(x,\lam) = \bfM^r_-(x,\lam) e^{-i\lam x \ad \sig_3} \bfV_r(\lam).
$$
To compute the jump matrix, first note that, from \eqref{NLS:T} and the definition of $\bfM^\pm$,
\begin{equation}
\label{NLS:T.M}
\bfM^+(x,\lam)
 = 
 \bfM^-(x,\lam) 
 e^{-ix\lam \ad \sig_3}
\begin{pmatrix}
a(\lam)				&		\overline{b(\lam)}		\\[5pt]
b(\lam)				&	\overline{a(\lam)}
\end{pmatrix} 
\end{equation}
Write 
$$f(x) \underset{x \to \pm \infty}{\sim} g(x)$$ if $$\lim_{x \to \pm \infty} |f(x)-g(x)|=0.$$ 
Since $\bfM^\pm(x,\lam) \to \I $ as $x \to \pm \infty$, it follows from \eqref{NLS:T.M} that
\begin{align*}
\begin{pmatrix}
m_1^+(x,\lam)\\[5pt]
m_2^+(x,\lam)
\end{pmatrix}
&	\underset{x\to -\infty}{\sim}
\begin{pmatrix}
a(\lam)	\\[5pt]
e^{2ix\lam}b(\lam)
\end{pmatrix},
&
\begin{pmatrix}
m_1^+(x,\lam)\\[5pt]
m_2^+(x,\lam)
\end{pmatrix}
&	\underset{x\to +\infty}{\sim}
\begin{pmatrix}
~1~	\\[5pt]
~0~
\end{pmatrix},\\
\\
\begin{pmatrix}
m_1^-(x,\lam)\\[5pt]
m_2^-(x,\lam)
\end{pmatrix}
&	\underset{x\to -\infty}{\sim}
\begin{pmatrix}
~1~	\\[5pt]
~0~
\end{pmatrix},
&
\begin{pmatrix}
m_1^-(x,\lam)\\[5pt]
m_2^-(x,\lam)
\end{pmatrix}
&	\underset{x\to +\infty}{\sim}
\begin{pmatrix}
\overline{a(\lam)}	\\[5pt]
-e^{2ix\lam} b(\lam)
\end{pmatrix}.
\end{align*}
From these asymptotic relations, \eqref{BC.right.form} (for $\bfM^r_-$), and \eqref{BC.right.extend} (for $\bfM^r_+$), we conclude that
\begin{align*}
\bfM^r_-(x,\lam) &\underset{x \to \infty}{\sim} 
e^{-i\lam x \ad \sig_3}
\begin{pmatrix}
1		&		-\overline{b(\lam)}/a(\lam)	\\[5pt]
0		&		1
\end{pmatrix}
\\[5pt]
\bfM^r_+(x,\lam) &\underset{x \to -\infty}{\sim}
e^{-i\lam x \ad \sig_3}
\begin{pmatrix}
1										&		0			\\[5pt]
-b(\lam)/\overline{a(\lam)}	&		1
\end{pmatrix},
\end{align*}
so that
$$ 
\bfV^r(\lam)		=	
\begin{pmatrix}
1	-	|r(\lam)|^2	&	-\overline{r(\lam)}	\\[5pt]
r(\lam)				&	1
\end{pmatrix}
$$
where
$$ 
r(\lam) = -b(\lam)/\overline{a(\lam)}.
$$
The Beals-Coifman solutions $\bfM^r(x,z)$ play a fundamental role in the inverse problem. To describe their large-$z$ asymptotic behavior, recall (cf.\ Remark \ref{NLS:remark.lim.z})  that $\lim_{z \to \infty} F(z) = A$ uniformly in proper subsectors of $\C \setminus \R$ if
$$ 
\lim_{R \to \infty} \sup_{\substack{|z| \geq R\\ \arg(z) \in (\alpha,\beta)}} \left| F(z) - A \right| = 0
$$
for any proper subinterval $(\alpha,\beta)$ of $(0,\pi)$ or $(\pi,2\pi)$.

\begin{theorem}[Right-normalized Beals-Coifman Solutions]
\label{thm:NLS.BC.right} 
Suppose that $q \in L^1(\R)$. For each $z \in \C \setminus \R$, there exists a unique solution to the problem 
\begin{equation}
\label{BC.right}
\left\{
\begin{aligned}
\frac{d}{dx} \bfM  =	-iz  \ad \sig_3 (\bfM)  + \bfQ_1(x) \bfM,\\
\lim_{x \to +\infty}	\bfM(x,z)	&=	\I,\\
\bfM(x,z)	&\text{ is bounded as }x \to -\infty.
\end{aligned}
\right.
%}
\end{equation}
The unique solution $\bfM^r(x,z)$ has the asymptotic behavior
$$ \lim_{z \to \infty} \bfM^r(x,z) = \I$$ as $z \to \infty$ in any proper subsector of the upper or lower half-planes. Moreover, $\bfM^r(x,z)$ has continuous boundary values $\bfM^r_\pm$ as $\pm \imag z \downarrow 0$  and $\real z \to \lam \in \R$  that 
satisfy the jump relation
\begin{equation}
\label{BC.Jump.right}
\bfM^r_+(x,\lam)  = \bfM^r_-(x,\lam) e^{-i\lam x \ad \sig_3} 
\bfV^r(\lam)
\end{equation}
where 
\begin{equation}
\label{NLS:Jump.right}
\bfV^r(\lam) = 
\begin{pmatrix}
1-|{r}(\lam)|^2 		&	-\overline{r(\lam)}	\\[5pt]
{r}(\lam)	&	1
\end{pmatrix}
\end{equation}
and
${r}(\lam) = -b(\lam)/\overline{a(\lam)}$. 
If $q \in L^1(\R) \cap L^2(\R)$, then, for each $x$, 
\begin{equation}
\label{BC.asy.right}
\bfM^r_\pm(x,\lam) - \I \in L^2(\R). 
\end{equation}
Finally, if $q \in L^1(\R) \cap C(\R)$, 
\begin{equation}
\label{NLS:right.recon}
q(x) = \lim_{z \to \infty} 2iz \left( \bfM^r \right)_{12}(x,z)
\end{equation}
where the limit is taken as $|z| \to \infty$ in any proper subsector of the upper or lower half-plane.
\end{theorem}

\begin{proof}
We have already computed the jump relation; the claimed large-$z$ asymptotic behavior follows from \eqref{NLS:m1m2.z.infty} and the analogous  statement for $m_1^-$ and $m_2^-$.
iI remains to show that the Beals-Coifman solutions are unique, to show that \eqref{BC.asy.right} holds, and to prove the reconstruction formula \eqref{NLS:right.recon}.

To prove uniqueness, suppose that, for given $z$, $\bfM$ and $\bfM^\sharp$ solve \eqref{BC.right}. We'll assume that $\imag z<0$ since the proof for $\imag z> 0$ is similar. 

Since $\bfM$ and $\bfM^\sharp$ both solve \eqref{NLS:M}, 
there is a constant matrix
$$ 
A = 
\begin{pmatrix}
a_{11}	&	a_{12}	\\[5pt]
a_{21}	&	a_{22}
\end{pmatrix}
$$
so that
$$ \bfM(x,z) = \bfM^\sharp(x,z) 
	\begin{pmatrix}
	a_{11}	&	e^{-2ixz} a_{12}	\\[5pt]
	e^{2ixz} a_{21}	&	a_{22}
	\end{pmatrix}.
$$
Since $\left| e^{-2ixz} \right| \to \infty$  as $x \to +\infty$ while both $\bfM, \, \bfM^\sharp \to \I$ as $x \to +\infty$,  it follows that $a_{12}=0$. Since $\left|e^{2ixz}\right|  \to \infty$ as $x \to -\infty$ while both $\bfM$ and $\bfM^\sharp$ are bounded, we conclude that $a_{21} = 0$. Using the normalization at $+\infty$ again we conclude that $a_{11}=a_{22}=1$ and $\bfM=\bfM^\sharp$.

The property \eqref{BC.asy.right} follows from the series representations for $m_1^\pm$ and $m_2^\pm$ and an argument similar to the proof of Proposition \ref{prop:NLS.direct.4}.

Finally, we consider the reconstruction formula \eqref{NLS:left.recon}. We give the proof for $\imag z > 0$ since the proof for $\imag z<0$ is similar. First, note that
$$
\left( \bfM^r\right)_{12}(x,z) = \overline{m_2^+(x,\zbar)},
$$ 
it suffices to show that
$$ 
q(x) = \lim_{z \to \infty} 2iz \overline{m_2^+(x,\zbar)}. $$
To see this we use
the absolutely and uniformly convergent Volterra series representation \eqref{m2sol}, the fact that 
$$ \lim_{z \to \infty} 2iz \int_x^\infty q(y) e^{2i\zbar(x-y)} \, dy = q(x),$$
and the fact that , for any $n \geq 2$,
$$ \lim_{z \to \infty} 2iz \int_{x < y_1 < \ldots < y_{2n-1}} Q_{2n-1}(y) e^{2i\zbar(x-\phi_{2n-1}(y))} \, dy= 0 $$
by dominated convergence. 
\end{proof}

We can also construct ``left'' Beals-Coifman solutions normalized at $-\infty$ as follows:
\begin{align*}
\bfM^\ell(x,z) &= 
\begin{pmatrix}
m_1^+(x,z)	&	\overline{m_2^-(x,\zbar)}	\\[5pt]
m_2^+(x,z)	&	\overline{m_1^-(x,\zbar)}
\end{pmatrix}
\begin{pmatrix}
1/a(z)	&	0	\\[5pt]
0			&	1	
\end{pmatrix},
& \imag z < 0
\intertext{and}
\bfM^\ell(x,z) &= \sig_1\overline{\bfM^\ell(x,\zbar)}\sig_1, & \imag z > 0.
\end{align*}

\begin{theorem}[Left-normalized Beals-Coifman solutions]
\label{thm:NLS.BC.left}
Suppose that $q \in L^1(\R)$. For each $z \in \C \setminus \R$, there exists a unique solution to the problem 
\begin{equation*}
\left\{
\begin{aligned}
\frac{d}{dx} \bfM  =	-iz  \ad \sig_3 (\bfM)  + \bfQ_1(x) \bfM,\\
\lim_{x \to -\infty}	\bfM(x,z)	&=	\I,\\
\bfM(x,z)	&\text{ is bounded as }x \to +\infty.
\end{aligned}
\right.
%}
\end{equation*}
The unique solution $\bfM^\ell(x,z)$ has the asymptotic behavior
$$ \lim_{z \to \infty} \bfM^\ell(x,z) = \I$$ as $z \to \infty$ in any proper subsector of the upper or lower half-plane. Moreover, $\bfM^\ell(x,z)$ has continuous boundary values $\bfM^\ell_\pm$ on $\R$ as $\pm \imag z \downarrow 0$ and $\real z \to \lam \in \R$ that 
satisfy the jump relation
\begin{equation}
\label{BC.Jump.left}
 \bfM^\ell_+(x,\lam)  = \bfM^\ell_-(x,\lam) e^{-i\lam x \ad \sig_3} 
\bfV^\ell(z)
\end{equation}
where 
\begin{equation}
\label{NLS:Jump.left}
\bfV^\ell(z)=
\begin{pmatrix}
1						&	-\overline{\breve{r}(\lam)}	\\[5pt]
\breve{r}(\lam)	&	 1-|\breve{r}(\lam)|^2 
\end{pmatrix}
\end{equation}
and 
$\breve{r}(\lam) = -b(\lam)/a(\lam)$. If $q \in L^1(\R) \cap L^2(\R)$, then, for each $x$, 
\begin{equation}
\label{BC.asy.left}
\bfM^\ell_\pm(x,\lam) - \I \in L^2(\R). 
\end{equation}
Finally, if $q \in L^1(\R) \cap C(\R)$, 
\begin{equation}
\label{NLS:left.recon}
q(x) = \lim_{z \to \infty} 2iz \left( \bfM^\ell \right)_{12}(x,z)
\end{equation}
where the limit is taken as $|z| \to \infty$ in any proper subsector of the upper or lower half-plane.
\end{theorem}

We omit the proof.

The large-$z$ asymptotics of $\bfM^r(x,z)$ (resp.\ $\bfM^\ell(x,z)$) together with the jump relation \eqref{BC.Jump.right} (resp.\ the jump relation \eqref{BC.Jump.left}) define a \emph{Riemann-Hilbert problem}. We will see that, properly formulated, these Riemann-Hilbert problems have unique solutions given the data $r$ and $\breve{r}$, offering a means of recovering $q$ from $r$ and $\breve{r}$. 

In fact, $r$ uniquely determines $\breve{r}$, $a$, and $b$, so that  the Riemann-Hilbert problem for $\bfM^r$ can be conjugated to the Riemann-Hilbert problem for $\bfM^\ell$. To see this, consider the analytic function $F$ on $\C \setminus \R$ defined by
\begin{equation}
\label{NLS:F}
F(z) = 
\begin{cases}
1/\overline{a(\zbar)}	&	\imag(z) > 0,	\\
\\
a(z)							&	\imag(z) < 0 .
\end{cases}
\end{equation}
From what has already been proved, $F(z)$ is piecewise analytic in $\C \setminus \R$, 
$F(z) \to 1$ as $|z| \to \infty$ in any proper subsector of the upper or lower half-plane, and $F(z)$ has continuous boundary values $F_\pm$ on the real axis with 
$$
F_+(\lam) = F_-(\lam) (1-|r(\lam)|^2) 
$$
as follows from the definitions of $F$ and $r$ together with the relation \eqref{NLS:T.sym}. Taking logarithms we see that
$$ \log F_+ (\lam) - \log F_-(\lam) = \log \left(1-|r(\lam)|^2 \right). $$
Recall that, if $f \in H^1(\R)$, the function
\begin{equation}
\label{RHP:scalar}
W(z) = \frac{1}{2\pi i} \int_{-\infty}^\infty \frac{f(\zeta)}{\zeta-z} \, d\zeta
\end{equation}
is the unique function on $\C \setminus \R$ with $W(z) \to 0$ as $z \to \infty$ and $W_+-W_-=f$, where $W_\pm$ are the boundary values of $W(z)$ as $\pm \imag z \to 0$. 
Motivated by this fact, we set
$$ G(z) = \exp\left(\frac{1}{2\pi i} \int_{-\infty}^\infty \frac{1}{s-z} \log\left(1-|r(s)|^2 \right) \, ds \right). $$
Note that, for $r \in L^\infty \cap L^2$ with $\norm[\infty]{r} < 1$, we have
$$
G(z) =1 + \bigO{\frac{1}{z}}
$$
as $z \to \infty$. 
The function $G$ satisfies the jump and boundary conditions and is analytic in $\C \setminus \R$. Moreover, the function
$H(z) = F(z)/G(z)$ is analytic in the same region, continuous across the real axis, and $\lim_{z \to \infty}  H(z)=1$. It follows from Liouville's theorem that $F(z)=G(z)$, which shows that $a$ is uniquely determined by $r$. Since $\breve{r}(\lam) = r(\lam) \overline{a(\lam)}/a(\lam)$ we see that $r$ determines $\breve{r}$.

In what follows, it will be important to note that the boundary values $F_\pm(\lam)$ satisfy the identity
$$
F_+(\lam) F_-(\lam) =   a(\lam) /\overline{a(\lam)}   
$$
as follows easily from the definition.

One can also conjugate the Riemann-Hilbert problem for $\bfM^r$ to that for $\bfM^\ell$ as follows. Given a function $\bfM^r(x,z)$ solving the ``right'' Riemann-Hilbert problem (i.e., the jump condition \eqref{BC.Jump.right} and the normalization condition \eqref{BC.asy.right}), the function
$$
 \bfM(x,z) = \bfM^r(x,z) 
	\begin{pmatrix}
	F(z)^{-1}	&	0	\\
	0				&	F(z)
	\end{pmatrix}
$$
is easily seen to solve the Riemann-Hilbert problem for $\bfM^\ell$ (i.e., the jump condition \eqref{BC.Jump.left} and the normalization condition \eqref{BC.asy.left})  since the additional factor doesn't change the large-$z$ asymptotic behavior of the solution,  while the jump matrices $\bfV^r$ and $\bfV^\ell$ are related by the identity $\bfV^r = F_-^{-\sig_3} \bfV^\ell F_+^{\sig_3}$.

\subsection{The Inverse Scattering Map}
\label{NLS:sec.inverse}

To reconstruct $q$ we will solve the following Riemann-Hilbert problem (compare RHP \ref{NLS:RHP.I}).

\begin{RHP}
\label{NLS:RHP.right}
Given $r \in H^{1,1}_1(\R)$ and $x \in \R$, find a function $\bfM(x,z): \C \setminus \R \to SL(2,\C)$ so that:
\begin{enumerate}
\item[(i)]		$\bfM(x,z)$ is analytic in $\C \setminus \R$ for each $x$,
\item[(ii)]	$\bfM(x,z)$ has continuous boundary values $\bfM_\pm(x,\lam)$ on $\R$, 
\item[(iii)]	$\bfM^\pm(x,\lam) - \I$ in $L^2(\R)$, and
\item[(iv)]	The jump relation
		$$ \bfM_+(x,\lam) = \bfM_-(x,\lam) e^{-i\lam x \ad \sig_3} \bfV(\lam) $$
		holds, where
		$$
		\bfV(\lam) = 
		\begin{pmatrix}
		1-|r(\lam)|^2	&	-\overline{r(\lam)}	\\[5pt]
		r(\lam)			&	1
		\end{pmatrix}.
		$$
\end{enumerate}
\end{RHP}

We will recover $q(x)$ from 
$$ q(x) = \lim_{z \to \infty} 2iz \left( \bfM \right)_{12}(x,z). $$

Riemann-Hilbert problem \ref{NLS:RHP.right} may usefully be thought of as an elliptic boundary value problem (the analyticity condition means  that $\dbar \bfM = 0$ on $\C \setminus \R$). For this reason one should be able to reformulate RHP \ref{NLS:RHP.right} as a boundary integral equation, much as the Dirichlet problem on bounded domain may be reduced to a boundary integral equation. We now describe such a formulation, due to  Beals and Coifman \cite{BC:1984}.

First, observe that the jump matrix $\bfV(\lam)$ admits a factorization of the form
$$ \bfV(\lam) = \left(I-w^-(\lam) \right)^{-1} \left(I+ w^+(\lam) \right) $$
where
$$ 
w^+(\lam) =
\begin{pmatrix}
0	&	0	\\
r(\lam)	&	0
\end{pmatrix}, 
\quad
w^-(\lam) = 
\begin{pmatrix}
0	&	-\overline{r(\lam)}	\\
0	&	0
\end{pmatrix}.
$$
so that
$$ e^{-i \lam x \ad \sig_3} \bfV(\lam) = \left(I - w_x^-(\lam) \right)^{-1} \left( I + w_x^+(\lam) \right) $$
where
$$ 
w_x^+(\lam) =
\begin{pmatrix}
0	&	0	\\
e^{2i\lam x}r(\lam)	&	0
\end{pmatrix}, 
\quad
w_x^-(\lam) = 
\begin{pmatrix}
0	&	-e^{-2i\lam x} \overline{r(\lam)}	\\
0	&	0
\end{pmatrix}.
$$

Note that, if $r \in H^{1,1}_1(\R)$, then $w_\pm \in L^\infty(\R) \cap L^2(\R)$ with $\norm[\infty]{w^\pm} < 1$.

Next, introduce the unknown matrix-valued function
$$ \mu(x,\lam)= \bfM_+(x,\lam) \left( I + w_x^+(\lam)\right)^{-1} = \bfM_-(x,\lam) \left( I  - w_x^-(\lam) \right)^{-1} $$
and observe that
$$ \bfM_+(x,\lam)  - \bfM_-(x,\lam)  = \mu(x,\lam) \left( w_x^+(\lam) + w_x^-(\lam) \right). $$
Recalling \eqref{RHP:scalar} and the asymptotic condition on $\bfM(x,z)$, we conclude that
\begin{equation}
\label{NLS:BC.pre} \bfM(x,z) = \I + \frac{1}{2\pi i} \int \frac{\mu(x,s)\left(w_x^+(\lam) + w_x^-(\lam) \right) }{s-z} \ ds. 
\end{equation}
Using this representation, we can derive a boundary integral equation for the unknown function $\mu(x,\lam)$ which, if solvable, uniquely determines $\bfM(x,z)$ from the Cauchy integral formula. For $f \in H^1(\R)$, define the Cauchy projectors $C_\pm$ by 
\begin{equation}
\label{Cpm}
 \left( C_\pm f \right) (\lam) =   \lim_{\eps \downarrow 0} \frac{1}{2 \pi i} \int \frac{f(s)}{s-(\lam \pm i\eps)} \, ds. 
\end{equation}
The projectors $C_\pm$ extend to isometries of $L^2(\R)$ with $C_+ - C_- = I$, the identity operator on $L^2(\R)$.
Moreover, $C_\pm$ act in Fourier representation as multiplication by the respective functions $\chi_{(0,\infty)}$ and $-\chi_{(-\infty,0)}$, where $\chi_A$ denotes the characteristic function of the set $A$.  Taking limits in \eqref{NLS:BC.pre} we recover
$$
\bfM_+(x,z) = \I  + C_+\left( \mu w_x^+ + \mu w_x^- \right).
$$
Since $$\bfM_+ = \mu \left(I+ w_x^+\right)$$ and $$C_+ (\mu w_x^+) = \mu w_x^+ + C_-\left(\mu w_x^+\right)$$ we conclude 
that
\begin{equation}
\label{NLS:BC}
\mu	=	\I + \calC_w (\mu)
\end{equation}
where for a matrix-valued function $h$ and $w=(w_+,w_-)$
\begin{equation}
\label{NLS:BC.op}
\calC_w(h) = C_+(h w_x^-) + C_-(h w_x^+).
\end{equation}
The integral operator $\calC_w$ is called the \emph{Beals-Coifman integral operator}, and equation \eqref{NLS:BC} is called the \emph{Beals-Coifman integral equation}. For $r \in L^\infty(\R) \cap L^2(\R)$, the operator $\calC_w$ is a bounded operator on matrix-valued $L^2(\R)$ functions; moreover, since the Beals-Coifman solutions are expected to have boundary values $\bfM_\pm$ with $\bfM_\pm(x,\dotarg)  - \I$ belonging to $L^2(\R)$ (see Theorem \ref{thm:NLS.BC.left}), it is reasonable to impose the condition $\mu(x,\dotarg)- \I \in L^2(\R)$. 

\begin{proposition}
\label{NLS:prop.BC}
Suppose that $x \in \R$ and $r \in H^{1,1}_1(\R)$. There exists a unique solution $\mu$ of the Beals-Coifman integral equation \eqref{NLS:BC} with $\mu(x,\dotarg) - \I \in L^2(\R)$. Moreover, $\mu(x,\dotarg)-\I \in H^1(\R)$ with
\begin{equation}
\label{NLS:BC.lam.apriori}
\norm[L^2]{\frac{\dee \mu}{\dee \lam}(x,\dotarg)} \lesssim \frac{\norm[H^{1,0}]{r}}{1- \norm[L^\infty]{r}} \norm[L^2]{\mu-1}. 
\end{equation}
where the implied constant depends only on $x$. 
\end{proposition}

\begin{proof}
Norm the matrix-valued functions on $L^2(\R)$ by
$$ \norm[L^2]{F}^2 = \int \left(|F_{11}(\lam)|^2 + |F_{12}(\lam)|^2 + |F_{21}(\lam)|^2 + |F_{22}(\lam)|^2  \right) \, d\lam $$
i.e., $\norm[L^2]{F}^2 = \int |F(\lam)|^2 \, d\lam$ where $| A |$ is the Frobenius norm on $2 \times 2$ matrices. Since
$\norm[L^2 \to L^2]{C_\pm} = 1$ and $\norm[L^\infty]{w_x^\pm} = \norm[L^\infty]{r}$, it follows that 
$$\norm[L^2 \to L^2]{\calC_w} = \norm[L^\infty]{r} <  1.$$ 
Hence $(I-\calC_w)^{-1}$ exists as a bounded operator on $L^2(\R)$. Setting $\mu^\sharp = \mu-\I$, \eqref{NLS:BC} becomes 
\begin{equation}
\label{NLS:BC.sharp}
\mu^\sharp = \calC_w(\I) + \calC_w(\mu^\sharp) 
\end{equation}
where  $\calC_w(\I) \in L^2$ since $r \in L^2$. Hence
\begin{equation}
\label{NLS:BC.sharp.sol}
\mu^\sharp = \left(I - \calC_w\right)^{-1} \calC_w \I 
\end{equation}
and $\mu = \I + \mu^\sharp$. Any two solutions $\mu_1$ and $\mu_2$ with $\mu_1  - \I, \,  \mu_2-\I \in L^2$ satisfy $(\mu_1-\mu_2) = \calC_w (\mu_1 - \mu_2)$ so that $\mu_1 = \mu_2$.

Next, we show that, for each $x$, $\mu(x,\dotarg) - \I \in H^1(\R)$, following closely the argument in \cite[\S 3]{DZ:2003}. First suppose that $r \in C_0^\infty(\R)$. In \eqref{NLS:BC.sharp}, the first right-hand term is actually a smooth function since $C_\pm$ preserve Sobolev spaces. An argument with difference quotients shows that the derivative of $\mu^\sharp$ with respect to $\lam$ exists as a vector in $L^2$ and
\begin{equation}
\label{NLS:BC.sharp.lam}
\frac{\dee \mu^\sharp}{\dee \lam} = \calC_{dw/d\lam} \I + \calC_{dw/d\lam} \mu^\sharp + \calC_w \left( \frac{\dee \mu^\sharp}{\dee \lam} \right)
\end{equation}
where
$$ 
\calC_{dw/d\lam} h = C_+\left( h \frac{\dee w_x^+}{\dee \lam} \right) + C_-\left( h \frac{\dee w_x^-}{\dee \lam} \right). 
$$
To obtain an effective bound on $\dee \mu^\sharp / \dee \lam$ we first note that
$$ \norm[L^2]{\frac{\dee w_x^\pm}{\dee \lam}} \leq c \norm[H^{1,0}]{r} $$
where $c$ depends linearly on $x$. Next, we recall that for any $\eps>0$,
$$ \norm[L^\infty]{f} \lesssim \eps \norm[L^2]{f'} + \eps^{-1} \norm[L^2]{f}. $$
It now follows from \eqref{NLS:BC.sharp.lam} that
\begin{align*}
\norm[L^2]{(\mu^\sharp)'}
	&	\leq  
					c \norm[L^\infty]{\mu^\sharp} \norm[H^{1,0}]{r} + 
					\norm[L^\infty]{r} \norm[L^2]{\frac{\dee \mu^\sharp}{\dee \lam}}\\
	&	\leq		c \eps \norm[L^2]{(\mu^\sharp)'} \norm[H^{1,0}]{r} + 
					c \eps^{-1} \norm[L^2]{\mu^\sharp} \norm[H^{1,0}]{r} + 
					\norm[L^\infty]{r} \norm[L^2]{\frac{\dee \mu^\sharp}{\dee \lam}}
\end{align*}
For $\eps$ with $c\eps + \norm[L^\infty]{r} < 1$ we conclude that \eqref{NLS:BC.lam.apriori} holds if 
$r \in C_0^\infty(\R)$. 

To complete the argument, suppose $r \in H^{1,0}(\R)$ and $\{r_n\}$ is a sequence from $C_0^\infty(\R)$ with
$r_n \to r$ in $H^{1,0}(\R)$.  Let $\mu_n^\sharp$ correspond to $r_n$. Using the second resolvent formula, it is easy to see that $(I-\calC_{w_n})^{-1} \to (I-\calC_w)^{-1}$ as operators on $L^2$ so that $\mu^\sharp_n \to \mu^\sharp$ as $n \to \infty$. It is now easy to see that $\mu^\sharp$ has a bounded weak $L^2$ derivative obeying \eqref{NLS:BC.lam.apriori}.
\end{proof}

For subsequent use, we note a simple but very important consequence of the proof of Proposition \ref{NLS:prop.BC}.

\begin{proposition}[Vanishing Theorem for RHP \ref{NLS:RHP.right}]
\label{NLS:prop.null}
Suppose that $r \in H^{1,0}(\R)$ with $\norm[L^\infty]{r} < 1$ and $x \in \R$. Suppose that $\bfn(x,z): \C \setminus \R \to SL(2,\C)$ so that 
\begin{enumerate}
\item[(i)]		$\bfn(x,z)$ is analytic in $\C \setminus \R$,
\item[(ii)]	$\bfn(x,z)$ has continuous boundary values $\bfn_\pm(x,\lam)$ on $\R$,
\item[(iii)]	$\bfn_\pm(x,\dotarg) \in L^2(\R)$,
\item[(iv)]	The jump relation 
				$$ \bfn_+(x,\lam) = \bfn_-(x,\lam) e^{-i\lam x \ad \sig_3} \bfV(\lam) $$
				holds, with $\bfV(\lam)$ as in RHP \ref{NLS:RHP.right}.
\end{enumerate}
Then $\bfn(x,z) \equiv \mathbf{0}$.
\end{proposition}

\begin{proof}
Given such a function $\bfn(x,z)$, let 
$$\nu(x,\lam) = \bfn_+(x,\lam) \left(I + w_x^+(\lam)\right)^{-1} = \bfn_-(x,\lam) \left(I- w_x^-(\lam) \right)^{-1}.$$
Mimicking the arguments that lead to the Beals-Coifman integral equation we conclude that
$$ \nu = \calC_w \nu $$
which shows that $\nu \equiv 0$ since $\norm[L^2 \to L^2]{\calC_w} < 1$. It now follows from \eqref{NLS:BC.pre} with 
$\mu$ replaced by $\nu$ that $\bfn(x,z) \equiv 0$.
\end{proof}

A piecewise analytic function $\bfn(x,z)$ satisfying (i)--(iv) above is called a \emph{null vector} for RHP \ref{NLS:RHP.right}. Proposition \ref{NLS:prop.null} asserts that RHP \ref{NLS:RHP.right} has no nontrivial null vectors. 

Since the solution of RHP \ref{NLS:RHP.right} is unique, any transformation of $\bfM$ that leaves the solution space invariant is a symmetry of the solution. Since
$$ \sigma_1 \overline{\bfV(\lam)} \sigma_1 = \bfV(\lam)^{-1}, $$
it follows that the map
$$ \bfM(x,z) \mapsto \sig_1 \overline{\bfM(x,\zbar)} \sig_1 $$
preserves the solution space. This symmetry implies that
\begin{equation}
\label{NLS:mu.sym}
\mu(x,\lam) = \sig_1 \overline{\mu(x,\lam)} \sig_1.
\end{equation}

We now define 
$$\bfM_+(x,\lam) = \mu(x,\lam) \left(I+w_x^+(\lam)\right), \quad \bfM_-(x,\lam) = \mu(x,\lam) \left(I - w_x^-(\lam) \right).$$
The next proposition shows that $\bfM(x,z)$ are Beals-Coifman solutions for a potential $q$ determined by the asymptotics of $\bfM(x,z)$. 

\begin{proposition}
\label{NLS:prop.BC.Q}
Suppose that $r \in H^{1,1}_1(\R)$, denote by $\bfM(x,z)$ the unique solution of RHP \ref{NLS:RHP.right}, and by 
$\bfM_\pm(x,\lam)$ the boundary values of $\bfM(x,z)$. Then 
$$ \frac{d}{dx} \bfM(x,z) =- iz \ad \sig_3 \left( \bfM \right) + \bfQ_1 (x) \bfM(x,z) $$
where
\begin{equation}
\label{NLS:Q1.recon} 
\bfQ_1(x) = \frac{1}{2\pi} \ad \sig_3 \left( \int \mu(x,s) \left( w_x^+(s) + w_x^-(s) \right) \right)  \, ds
\end{equation}
takes the form
$$ \bfQ_1(x) = 
\begin{pmatrix}
0	&	q(x)	\\
\overline{q(x)}	&	0
\end{pmatrix}.
$$
\end{proposition}

\begin{proof}
First, by differentiating the solution formula $\mu - \I = (I - \calC_w)^{-1}\mu$ and using the fact that $r \in H^{0,1}(\R)$, 
it is easy to see that $(d\mu/dx)(x,\dotarg) \in L^2(\R)$. It follows from the representation
$$
\bfM_\pm(x,\lam) - \I = C_\pm\left( \mu(x,\dotarg) \left(w_x^+(\dotarg) + w_x^-(\dotarg)\right) \right) (\lam) 
$$
the same is true for $\bfM^\pm(x,\lam)-\I$. 
We will differentiate the jump relation for $\bfM$ and use Proposition \ref{NLS:prop.null}. From the jump relation for 
$\bfM_\pm$, we have
\begin{equation}
\label{NLS:RHP.x}
\left( \frac{d\bfM_+}{dx} + i\lam \ad \sig_3  (\bfM_+)\right) = 
\left( \frac{d\bfM_-}{dx}  +i\lam \ad \sig_3 (\bfM_-)  \right) e^{-ix\lam \ad \sig_3} \bfV(\lam)
\end{equation}
where we used $\bfV = \bfM_-^{-1} \bfM^+$ and the Leibniz rule for the derivation $A \mapsto \ad \sig_3 (A)$. 
Using the identity
\begin{equation}
\label{RHP.Cpm.1}
i\lam (C_\pm f)(\lam) = -\frac{1}{2\pi} \int f(s) \, ds  + C_\pm\left((\dotarg) f(\dotarg)\right)(\lam) 
\end{equation}
and the fact that $r \in H^{1,1}(\R)$, we see that
$$ i\lam \ad \sig_3 (\bfM_\pm) +\bfQ_1(x) \in L^2(\R)$$
where $\bfQ_1$ is the bounded continuous function of $x$ given by \eqref{NLS:Q1.recon}. We conclude that
$$ 
\bfn(x,z) \coloneqq \frac{d}{dx}\bfM(x,z) + i\lam \ad \sig_3 (\bfM(x,z)) - \bfQ_1(x) \bfM(x,z) 
$$
is a null vector for RHP \eqref{NLS:RHP.right} for each $x$, hence identically zero by Proposition \ref{NLS:prop.null}.
Finally, the diagonal components of $\bfQ_1$ are zero since $\bfQ_1$ lies in the range of $\ad \sig_3(\dotarg)$, and 
$\left( \bfQ_1 \right)_{21} = \overline{ \left( \bfQ_1 \right)_{12} }$ owing to the symmetry \eqref{NLS:mu.sym}.
\end{proof}

Tracing through the definitions we obtain the reconstruction formula
\begin{equation}
\label{NLS:q.recon.int}
q(x) = -\frac{1}{\pi} \int \overline{r(s}) e^{-2ixs} \mu_{11}(x,s) \, ds
\end{equation}
which together with RHP \ref{NLS:RHP.right} defines the inverse scattering map $\calI: r \to q$.

\begin{remark}
\label{NLS:Inverse.lin}
The Fr\'{e}ch\'{e}t derivative of the map $\calI$ at $r=0$ is clearly the map
$$ q(x) = - \frac{1}{\pi} \int e^{-2ix\lam} \overline{r(\lam)} \, d\lam .$$
This map is the inverse map for the Fr\'{e}ch\'{e}t derivative of $\calR$ (see Remark 
\ref{NLS:Direct.lin}).
\end{remark}

In the sequel, it will be important to know that
\begin{equation}
\label{NLS:mu.de}
\frac{d}{dx} \mu =- i\lam \ad \sig_3 (\mu) + \bfQ_1(x) \mu.
\end{equation}
This is a simple consequence of Proposition \ref{NLS:prop.BC.Q} and the Leibniz rule for $\ad \sig_3$ (see Exercise \ref{ex:NLS.mu.bc}). 

We'll first show that $q \in H^{1,1}(\R)$, and then show that the map $r \mapsto q$ is a locally Lipschitz continuous  map from $H^{1,1}_1(\R)$ to $H^{1,1}(\R)$.  To aid the analysis, note that \eqref{NLS:BC} has $11$ and $12$ components
\begin{align}
\label{NLS:mu11}
\mu_{11}(x,\lam)	&=	1	+ 	C_-\left( \mu_{12}(x,\dotarg) e^{2i(\dotarg)x} r(\dotarg) \right)(\lam)\\
\label{NLS:mu12}
\mu_{12}(x,\lam)	&=	-C_+\left( \mu_{11}(x,\dotarg) e^{-2i(\dotarg)x} \overline{r(\dotarg)} \right)
\end{align}
so that $\mu_{11}-1 \in \Ran C_-$.  The following lemma \cite[Lemma 3.4]{DZ:2003} will play a critical role.

\begin{lemma}
\label{NLS:lemma.phase-weight}
Suppose that $r \in H^{1,0}(\R)$. For any $x > 0$, the estimates
\begin{equation}
\label{NLS:phase-weight}
\begin{aligned}
\norm[L^2]{C_+ e^{-2ix(\dotarg)} \overline{r(\dotarg)} }
	& \lesssim \norm[H^{1,0}]{r} (1+x^2)^{-1/2}, 
\\
\norm[L^2]{C_- e^{{2ix}(\dotarg)}{r(\dotarg)} } 
	&\lesssim \norm[H^{1,0}]{r} (1+x^2)^{-1/2}
\end{aligned}
\end{equation}
hold.
\end{lemma}

\begin{proof}
By Plancherel's theorem and  the fact that $C_+$ acts in Fourier transform representation as multiplication by $\chi_+(\xi) \coloneqq \chi_{(0,\infty)}(\xi)$, we may estimate
\begin{align*}
\norm[L^2]{C_+ e^{-2ix(\dotarg)} \overline{r(\dotarg)}}
	&=			\norm[L^2]{\chi_+ \widehat{\overline{r}}(\dotarg+x)}\\
	&\lesssim	(1+|x|^2)^{-1/2}\norm[H^{0,1}]{\widehat{r}}
\end{align*}
where in the last step we used $(1+|x|^2)^{1/2}(1+ |\xi+x|^2)^{-1/2} \leq 1$ for $x>0$ and $\xi>0$. The other proof is similar.
\end{proof}

Using Lemma \ref{NLS:lemma.phase-weight}, we can obtain ``one-sided'' control over the inverse scattering map.

\begin{proposition}
\label{NLS:prop.inverse.left}
Suppose that $r \in H^{1,1}_1(\R)$. Then $q$ as defined by \eqref{NLS:q.recon.int} belongs to $H^{1,1}(\R^+)$, and the map $r \mapsto q$ is locally Lipschitz continuous from $H^{1,1}_1(\R)$ to $H^{1,1}(\R^+)$.
\end{proposition}

\begin{proof}
We write \eqref{NLS:q.recon.int} as $q(x) = q_0(x) + q_1(x)$ where 
$$q_0(x)= -\frac{1}{\pi} \int \overline{r(s)} e^{-2isx} \, ds$$
and
$$ q_1(x) =- \frac{1}{\pi} \int \overline{r(s)} e^{-2ixs} \left(\mu_{11}(x,s)-1 \right) \, ds. $$
Clearly, $q_0 \in H^{1,1}(\R)$ with the correct continuity so it suffices to study $q_1(x)$. From \eqref{NLS:mu11} we may write
$$ q_1(x) = -\frac{1}{\pi} \int C_+\left( \overline{r(\dotarg)} e^{-2ix(\dotarg)} \right)(s) \left(\mu_{11}(x,s)  -1\right) \, ds $$
for $x< 0$, where we used the facts that $\mu_{11}-1 \in \Ran C_-$, that 
$$\int (C_- f)(s) (C_-f(s)) \, ds = 0$$ and that $C_+ - C_-=I$. From the solution formula \eqref{NLS:BC.sharp.sol} we have the  estimate
$$ \norm[L^2]{m_{11}(x,\dotarg)-1} 
	\leq \frac{\norm[L^2]{\, \calC_w \I \, }}{1-\norm[\infty]{r}} 
	\leq	(1+x^2)^{-1/2} \frac{\norm[H^{1,0}]{r}}{1-\norm[\infty]{r}}
$$
where in the last step we used \eqref{NLS:phase-weight}.
By this estimate, Lemma \ref{NLS:lemma.phase-weight}, and the Schwartz inequality, we conclude that for $x>0$,
$$ \left| q_1(x) \right| \lesssim \frac{1}{(1+x^2)} \frac{\norm[H^{1,0}]{r}}{1-\norm[\infty]{r}} $$
so that in particular $q_1 \in H^{0,1}(\R^+)$. 

To show that $q_1' \in L^2$, we differentiate and use \eqref{NLS:mu.de} to conclude that
\begin{equation}
\label{NLS:q1x}
q_1'(x) = -q(x) \, 
		\left( 
				\frac{1}{\pi} \int \overline{r(s)} e^{-2isx}  \mu_{21}(x,s) \, ds 
		\right).
\end{equation}
Since $r \in L^2$, $\mu_{21}(x,\lam) = \overline{\mu_{12}(x,\lam)}$,  and $\mu_{12} (x,\dotarg) \in L^2$ with bounds uniform in $x$, we can bound the integral uniformly in $x$ by the Schwartz inequality and conclude that $q_1' \in L^2(\R^+)$ as required.

To obtain the local Lipschitz continuity, first note that $r \mapsto q_0$ has the required mapping properties, so it suffices to consider the map $r \mapsto q_1$. To show that $r \mapsto q_1$ is locally Lipschitz continuous into $H^{0,1}(\R^+)$, it suffices, by estimates already given, to show that $r \mapsto (1+|x|^2)^{1/2}\left(\mu_{11}(x,\dotarg) -1\right) $ is locally Lipschitz continuous. It follows from \eqref{NLS:mu11}--\eqref{NLS:mu12} that
$$ \mu_{11} = 
	1 - A_r \mu_{11}
$$
where
$$ 
\left(A_{r} h \right) (\lam) = 
C_-
			\left(
				C_+
					\left( 
						h(\diamond) e^{2i(\diamond)x} \overline{r(\diamond)} 
					\right)(\dotarg) 
				e^{-2i(\dotarg)x} r(\dotarg) 
			\right) (\lam)
$$
Since $r \in H^{1,1}_1(\R)$, $A_r 1 \in L^2$ and the operator $A_r$ is bounded from $L^2$ to itself with norm $\norm[\infty]{r}^2 < 1$ so that 
$\mu_{11}$ is given by the $L^2$ -convergent Neumann series
$$ \mu_{11}- 1 = \sum_{n=1}^\infty A_r^n (1)$$
The map $r \mapsto A_r^n(1)$ takes the form $F_n(r,\ldots,r,\rbar,\ldots,\rbar)$ where
$F_n: \left(H^{1,1}_1(\R)\right)^{2n} \to L^2(\R)$ is a multilinear function obeying the bound
$$ \norm[L^2(\R)]{F_n(r_1,\ldots,r_{2n})} \leq (1+|x|^2)^{-1/2}
\left( \prod_{i=1}^{2n-1} \norm[L^\infty]{r_i}\right) \norm[H^{1,0}]{r_{2n}}.$$
The required local Lipschitz continuity for $\mu_{11} -1$ now follows as in the proof of Proposition \ref{prop:NLS.direct.1}.

To show that $r \mapsto q_1$ is locally Lipschitz from $H^{1,1}_1(\R)$ to $H^{0,1}(\R^+)$, it suffices by \eqref{NLS:q1x}  to show that $r \mapsto \mu_{21}(x,\dotarg)$ is locally Lipschitz from $H^{1,1}_1(\R)$ to $L^2(\R)$ with bounds uniform in $x \in \R^+$.  Since $\mu_{21}  = \overline{\mu_{12}}$, we can use the continuity result for $\mu_{11}$ and \eqref{NLS:mu12} to obtain the necessary result.
\end{proof}

The results obtained so far show that the map $r \mapsto q$ is locally Lipschitz from $H^{1,1}_1(\R)$ to $H^{1,1}(\R^+)$ and so gives ``half'' of the desired result. To obtain the full local Lipschitz continuity result, first note that, by trivial modifications of the proofs, we can show that $r \mapsto q$ is locally Lipschitz continuity from $H^{1,1}_1(\R)$ to $H^{1,1}(c,\infty)$ for any $c \in \R$. To finish the analysis, we consider the Riemann-Hilbert problem satisfied by the ``left'' Beals-Coifman solutions from Theorem \ref{thm:NLS.BC.left}.

\begin{RHP}
\label{NLS:RHP.left}
Given $r \in H^{1,1}_1(\R)$ and $x \in \R$, find a function $\bfM(x,z): \C \setminus \R \to SL(2,\C)$ so that:
\begin{enumerate}
\item[(i)]		$\bfM(x,z)$ is analytic in $\C \setminus \R$ for each $x$,
\item[(ii)]	$\bfM(x,z)$ has continuous boundary values $\bfM_\pm(x,\lam)$ on $\R$, 
\item[(iii)]	$\bfM^\pm(x,\lam) - \I$ in $L^2(\R)$, and
\item[(iv)]	The jump relation
		$$ \bfM_+(x,\lam) = \bfM_-(x,\lam) e^{-i\lam x \ad \sig_3} \bfV(\lam) $$
		holds, where
		$$
		\bfV(\lam)=
		\begin{pmatrix}
		1								&		-\overline{\br(\lam)}\\[5pt]
		\br(\lam)		&		1-|\br(\lam)|^2	
		\end{pmatrix}.
		$$
\end{enumerate}
\end{RHP}

The associated reconstruction formula is:
\begin{equation}
\bq(x) = \lim_{z \to \infty} 2iz \left(\bfM\right)_{12} (x,z).
\end{equation}

We can analyze RHP \ref{NLS:RHP.left} in much the same way as RHP \ref{NLS:RHP.right} and prove:

\begin{proposition}
\label{NLS:prop.inverse.right}
Suppose that $r \in H^{1,1}_1(\R)$. Then the map $\br \mapsto \bq$ is locally Lipschitz continuous from $H^{1,1}_1(\R)$ to $H^{1,1}(\R^-)$.
\end{proposition} 

Indeed, the same result holds true of $H^{1,1}(\R^-)$ is replaced by $H^{1,1}((-\infty,c))$. Since the map $r \mapsto \br$ is locally Lipschitz continuous, it remains only to prove that $q = \bq$. To do so we recall that the respective solutions $\bfM^r(x,z)$ and $\bfM^\ell(x,z)$ of RHP's \ref{NLS:RHP.left} and \ref{NLS:RHP.left} are related by
$$ \bfM^\ell(x,z) = \bfM^r(x,z) 
	\begin{pmatrix}
	F(z)		&	0		\\
	0			&	F(z)^{-1}
	\end{pmatrix}
$$
where $F(z)$ was defined in \eqref{NLS:F} and show to satisfy $F(z) = 1+ \bigO{1/z}$ as $z \to \infty$. It follows that 
$$ 
\lim_{z \to \infty} 2iz \left(\bfM^\ell\right)_{12}(x,z) =
\lim_{z \to \infty} 2iz \left( \bfM^r\right)_{12}(x,z)
$$
so that
$$ q(x)= \bq(x). $$

Proposition \ref{NLS:prop.inverse.left}, Proposition \ref{NLS:prop.inverse.right}, and these observations prove:

\begin{proposition}
The map $r \mapsto q$ defined by RHP \ref{NLS:RHP.left} and the reconstruction formula \eqref{NLS:q.recon.int} defines a locally Lipschitz continuous map from $H^{1,1}_1(\R)$ to $H^{1,1}(\R)$.
\end{proposition}

To finish the proof of Theorem \ref{thm:NLS.maps}, it remains to show that the maps $\calR$ and $\calI$ are one-to-one and mutual inverses.  

Let $r \in H^{1,1}(\R)$. By solving RHP \ref{NLS:RHP.right} we construct the unique Beals-Coifman solutions for the potential $q = \calI(r)$. From the Riemann-Hilbert problem satisfied by the solutions, we read off that $q$ has scattering transform $\calR(q) = r$, showing that $\calR \circ \calI$ is the identity map on $H^{1,1}_1(\R)$. 

Next, we claim that $\calR$ is one-to-one. Suppose that $q_1$, $q_2 \in H^{1,1}(\R)$ and $\calR(q_1) = \calR(q_2) = r$. If $\bfM^{(1)}(x,z)$ and $\bfM^{(2)}(x,z)$ are the respective Beals-Coifman solutions for $q_1$ and $q_2$, each satisfies RHP \ref{NLS:RHP.right} and so the difference satisfies a homogeneous RHP as in Proposition \ref{NLS:prop.null}. It now follows from Proposition \ref{NLS:prop.null} that $\bfM^{(1)}(x,z)=\bfM^{(2)}(x,z)$. Since $q$ can be recovered from large-$z$ asymptotics of $\bfM(x,z)$, it now follows that $q_1 = q_2$.

\subsection{Solving NLS for Schwartz Class Initial Data}
\label{NLS:sec.sol}

In this subsection we prove Theorem \ref{thm:NLS.BC}. We will use the complete integrability of NLS in the following form: a smooth function $q(x,t)$ solves NLS if and only the overdetermined system \eqref{NLS:Lax}
admits a $2\times 2$ matrix-valued fundamental solution $\Psi(x,t,\lam)$. Recall
that a joint solution $\Psi(x,t,\lam)$ is a fundamental solution if $\det \Psi(x,t,\lam) > 0$ for all $(x,t)$.
Given such a fundamental solution, one can cross-differentiate the system \eqref{NLS:Lax} and equate coefficients of $\Psi_{xt}$ and $\Psi_{tx}$ to obtain \eqref{NLS}.

We can also give a heuristic derivation of the evolution equations for the scattering data $a$ and $b$ from \eqref{NLS:Lax}, assuming that $q(x,t) \in \scrS(\R)$ as a function of $x$.  Let $\Psi^+(x,t,\lam)$ denote the Jost solution for $q(x,t)$. For each $t$, 
$$\Psi^+(x,t,\lam) \underset{x \to \infty}{\sim} e^{-i\lam x \sig_3}$$ and $\Psi_t(x,t,\lam) \to 0$ as $x \to +\infty$. On the other hand,
$$\Psi^+(x,t,\lam) \underset{x \to -\infty}{\sim} e^{-i\lam x \sig_3} T(\lam,t)  ,$$
where $T(\lam)$ is given by \eqref{NLS:T.sym} with $a=a(\lam,t)$ and $b=b(\lam,t)$. A joint solution of \eqref{NLS:Lax} must take the form $\Psi(x,t,\lam) = \Psi^+(x,t,\lam)C(t)$ for a matrix-valued function $C(t)$. From the second equation of \eqref{NLS:Lax} we obtain
\begin{equation}
\label{NLS:heuristic}
(\Psi^+)_t C(t) + \Psi^+ C'(t)= -2i\lam^2 \sig_3 \Psi^+ C + o(1) 
\end{equation}
where $o(1)$ denotes terms that vanish as $x \to \pm \infty$ for each fixed $t$ 
owing to the decay of $q$ and its derivatives. Taking $x \to +\infty$ in 
\eqref{NLS:heuristic}, we obtain $C'(t) = -2i\lam^2 \sig_3 C(t)$ so that, normalizing 
to $C(0)=\I$, we have $C(t) = e^{-2i\lam^2 \sig_3 t}$. Taking $x \to -\infty$ in 
\eqref{NLS:heuristic}, we obtain
$$
T'(\lam) e^{-i\lam x \sig_3} C(t) + T(\lam)  e^{-i\lam x \sig_3} (-2i\lam^2 \sig_3) C(t)
 = -2i\lam^2 \sig_3 e^{i\lam x \sig_3} T(\lam) C(t) 
$$
or
$$ T'(\lam) = -2i\lam^2\ad \sig_3 T(\lam) $$
which implies that
$$ \dot{a}(\lam,t) = 0, \quad \dot{b}(\lam,t) = 4i\lam^2 b(\lam). $$

We consider the solution $\bfM(z;x,t)$ of RHP \ref{NLS:RHP} and the recovered potential
\begin{equation}
\label{NLS:q.recon.bis}
q(x,t) = -\frac{1}{\pi} \int \overline{r_0(s)} e^{-2it\theta} \mu_{11}(s;x,t) \, ds 
\end{equation}
where $r_0$ is the scattering transform  of the initial data $q_0$ and $\theta$ is the phase function \eqref{NLS:RHP.phase}.  We denote by $\bfM_\pm(z;x,t)$ the boundary values of the solution to RHP \ref{NLS:RHP}. Note that, by construction, $\det \bfM_\pm(\lam;x,t) = 1$ for all $(\lam,x,t)$. 
To prove that \eqref{NLS:q.recon.bis} solves the NLS equation, we will show that the functions
$$ \Psi_\pm(\lam;x,t) = \bfM_\pm(\lam;x,t) e^{-i(\lam x + 2\lam^2 t)\sig_3}, $$
which again have determinant one, solve the overdetermined system \eqref{NLS:Lax}. 

To do this, it suffices to show that $\bfM_\pm$ solve
\begin{equation}
\label{NLS:Lax.M}
\left\{
\begin{aligned}
\bfM_x	&= 	\left(- i\lam \ad \sig_3 	+ \bfQ_1\right)  \bfM\\
\bfM_t	&=	\left( -2i\lam^2 \ad \sig_3  + 2\lam \bfQ_1 + \bfQ_2 \right) \bfM
\end{aligned}
\right.
%}
\end{equation}
We will prove:

\begin{theorem}
\label{NLS:thm.Lax}
Suppose that $q_0 \in \scrS(\R)$, let $r = \calR(q)$, let $\bfM_\pm(\lam;x,t)$ be the boundary values of the solution to RHP \ref{NLS:RHP}, and let $q$ be given by \eqref{NLS:q.recon.bis}. Then
$q$ is a classical solution of the defocussing NLS equation \eqref{NLS} with $q(x,0)=q_0$. 
\end{theorem}

\begin{proof}
We have already shown that $\bfM_\pm$ solves the first of equations \eqref{NLS:Lax.M} in Proposition \ref{NLS:prop.BC.Q} by differentiating RHP \ref{NLS:RHP.left} with respect to the parameter $x$ and using Proposition \ref{NLS:prop.null}, the vanishing theorem for RHP \ref{NLS:RHP.left}. We will show that the second equation in \eqref{NLS:Lax.M} holds by differentiating the time-dependent RHP \ref{NLS:RHP} with respect to $t$ and using an analogous vanishing theorem. 

The jump matrix in RHP \ref{NLS:RHP} may be written
$$ 
\bfV(\lam;x,t) = 
	e^{-it\theta \ad \sig_3} \bfV(\lam), \quad
\bfV(\lam) = 
\begin{pmatrix}
1-|r_0(\lam)|^2	&	-\overline{r_0(\lam)}	\\[5pt]
r_0(\lam)			&		1	
\end{pmatrix}
$$
Differentiating the jump relation for $\bfM_\pm$ and using the Leibniz rule for $\ad \sig_3(\dotarg)$, we obtain 
$$
\left( \frac{\dee }{\dee t} + 2i\lam^2 \ad \sig_3 \right) \bfM_+(\lam;x,t)  =
\left( \frac{\dee }{\dee t} + 2i \lam^2 \ad \sig_3  \right) \bfM_-(\lam;x,t) \bfV(\lam;x,t)
$$
We will show that 
$\dee \bfM_\pm/\dee t$
and
$2i\lam^2 \ad \sig_3 (\bfM_\pm) -2 \lam \bfQ_1 \bfM_\pm - \bfQ_2 \bfM_\pm$
are $L^2$ boundary values of functions analytic in $\C \setminus \R$, so that
$$ 
\bfn_\pm(\lam;x,t) 
	\coloneqq  \left( 
							\frac{\dee}{\dee t} +2 i\lam^2 \ad \sig_3 - 
							2\lam \bfQ_1 - \bfQ_2 
					\right) \bfM_\pm$$
satisfy the hypothesis of Proposition \ref{NLS:prop.null} for each $t$. It will then follow that
the functions $\bfM_\pm(\lam;x,t)$ satisfy the second of equations \eqref{NLS:Lax.M}, showing that $q(x,t)$ is a classical solution of NLS. It follows from Theorem \ref{thm:NLS.maps} that $q(x,0)=q_0(x)$, so that $q(x,t)$ satisfies the initial value problem.

It remains to show that  $\dee \bfM_\pm/ \dee t$ and $2i\lam^2 \ad \sig_3 (\bfM_\pm) - 2\lam \bfQ_1 \bfM_\pm - \bfQ_2 \bfM_\pm$ have the required properties. This is accomplished in Lemmas \ref{lemma:NLS.RHP.t1} and \ref{lemma:NLS.RHP.t2} below.

\end{proof}

In what follows we will write $h_\pm \in \dee C(L^2)$ for a pair of $L^2$ functions $(h_-,h_+)$ if $h_\pm$ are the boundary values of a function $h$ analytic in $\C \setminus \R$.  In this language, conditions (i)--(iii) of Proposition \ref{NLS:prop.null} state that $\bfn_\pm \in \dee C(L^2)$.

\begin{lemma}
\label{lemma:NLS.RHP.t1}
Suppose that $r \in \scrS_1(\R)$ and let $\bfM_\pm(x,t,\lam)$ be boundary values of the  unique solution of the RHP \ref{NLS:RHP}. Then 
$\dee \bfM_\pm/\dee t \in \dee C(L^2). $
\end{lemma}

\begin{proof}
First we study $\dee \mu / \dee t$ where $\mu$ solves the Beals-Coifman integral equation
\begin{equation}
\label{NLS:BC.int.t}
\mu = \I + \calC_{w_{x,t}} \mu
\end{equation} 
where
$$ w_{x,t}^\pm(\lam) = e^{-it\theta \ad \sig_3}w^\pm(\lam)$$
and 
$$
\calC_{w_{x,t}} h = C_-\left( h w_{x,t}^+ \right) + C_+\left( h w_{x,t}^- \right).
$$
Differentiating \eqref{NLS:BC.int.t} we see that
$$
\frac{\dee \mu}{\dee t} = \calC_{\dee w_{x,t}/\dee t}(\mu) + \left( \calC_{w_{x,t}} \frac{\dee \mu}{\dee t} \right),
$$
Since $(I-\calC_{w_{x,t}})$ is invertible, this equation can be solved to show that $\dee \mu / \dee t \in L^2(\R)$ provided the
inhomogeneous term
$$ 
\calC_{\dee w_{x,t}/\dee t} \mu = 
	C_+\left( \mu \frac{\dee w_{x,t}^-}{\dee t} \right) +
	C_-\left( \mu \frac{\dee w_{x,t}^+}{\dee t} \right)
$$
belongs to $L^2$ as a function of $\lam$. Since $\mu-I \in L^2$ it suffices to show that
$\dee w_{x,t}^\pm /\dee t \in L^\infty \cap L^2$.  Since $\dee w_{x,t}^+ / \dee t = i\theta r e^{it\theta}$ and $\dee w_{x,t}^- /\dee t = -i\theta \overline{r} e^{-it\theta}$ and $\theta$ is a quadratic polynomial in $\lam$, this follows from the fact that $r \in \scrS(\R)$. 

Since $\bfM_\pm - I = C_\pm \left( \mu (w_{x,t}^+ + w_{x,t}^- \right)$ we have
$$ \frac{\dee \bfM_\pm}{\dee t} = 
	C_\pm 	\left[
							\frac{\dee \mu}{\dee t} \left(w_{x,t}^- + w_{x,t}^+\right) +
							\mu \left( \frac{\dee w_{x,t}^-}{\dee t} + \frac{\dee w_{x,t}^+}{\dee t} \right)
				\right]
$$
It follows from the facts that $\dee \mu/\dee t \in L^2$ and $r \in \scrS_1(\R)$ that the expression in square brackets is an $L^2$ function. This shows that $\dee \bfM_\pm/ \dee t \in \dee C(L^2)$. 

\end{proof}

In the proof of the next lemma, we will make use of the following large-$z$ asymptotic expansion for the right-normalized Beals-Coifman solution for $r \in \scrS_1(\R)$. Since the Beals-Coifman solution solves the Riemann-Hilbert problem, 
we have (compare \eqref{NLS:BC.pre})
\begin{equation}
\label{NLS:BC.post.pre}
\bfM(z;x)	=	\I + \frac{1}{2\pi i} \int \frac{1}{s-z} f(s;x) \, ds 
\end{equation}
where
\begin{equation}
\label{NLS:BC.post.pre.f}
f(s;x) = \mu(s;x) \left( w_x^-(s) + w_x^+(s) \right). 
\end{equation}
If $r \in \scrS_1(\R)$ then, since $\mu(\dotarg;x) - \I \in L^2(\R)$ for each $x$, 
the asymptotic expansion
\begin{equation}
\label{NLS:BC.exp}
\bfM(z;x,t)	\sim \I + \sum_{j \geq 0} \frac{m_j(x)}{z^{j+1}}
\end{equation}
holds. Substituting \eqref{NLS:BC.exp}
into the differential equation \eqref{BC.right} we obtain the relations
\begin{align*}
i\ad \sig_3 (m_0(x))	&=	\bfQ_1(x)	\\
m_j'(x) 					&=	-i \ad \sig_3 (m_{j+1}) + \bfQ m_j(x),	&	j\geq 0
\end{align*}
One can compute the coefficients $m_i(x)$ by deriving using these relations together with  the boundary condition $$\lim_{x \to +\infty} m_i(x) = 0.$$
Given all coefficients up to $m_{j-1}$, one first computes $\ad \sig_3 (m_j)$ and then uses $\ad \sig_3 (m_j)$ to find the diagonal of $m_j$. 
We will only need the following identities:
\begin{align}
\label{NLS:BC.left.m0}
m_0(x)	&=
	\begin{pmatrix}
	\dfrac{i}{2}\dint_{\hspace{-1mm} +\infty}^{\hspace{0.5mm}x} |q(s)|^2 \, ds	
			&		-\dfrac{i}{2} q(x)	\\[10pt]
	\dfrac{i}{2}\overline{q(x)}						
			&		-\dfrac{i}{2} \dint_{\hspace{-1mm} +\infty}^{\hspace{0.5mm}x} |q(s)|^2 \, ds
	\end{pmatrix},	\\[10pt]
\label{NLS:BC.left.ad.m1}
-i\ad \sig_3 (m_1(x))	&=	
	\begin{pmatrix}
		0			
		&		\dfrac{i}{2} q(x) \dint_{\hspace{-1mm}+\infty}^{\hspace{0.5mm} x} |q(s)|^2 \, ds  \\[5pt]
		-\dfrac{i}{2} \overline{q(x)} \dint_{\hspace{-1mm}+\infty}^{\hspace{0.5mm} x} |q(s)|^2 \, ds 
		&	0
	\end{pmatrix}\\[10pt]
\nonumber
		&	\quad	+
	\begin{pmatrix}
		0													&	- \dfrac{i}{2}q_x(x)			\\[5pt]
		 \dfrac{i}{2}\overline{q}_x	(x)				&	0
	\end{pmatrix}
\end{align}
We can identify 
$$ m_j(x) = -\frac{1}{2\pi i} \int s^{j-1} f(s;x) \, ds, $$
where $f$ is given by \eqref{NLS:BC.post.pre.f},
using equation \eqref{NLS:BC.post.pre}. In the application, $\mu$, $w_\pm$,  $f$, $m_j$ and $f_j$ also depend parametrically on $t$.

\begin{lemma}
\label{lemma:NLS.RHP.t2}
Fix $r \in \scrS_1(\R)$, and let $\bfM_\pm(x,t,\lam)$ be boundary values of the unique solution to RHP \ref{NLS:RHP}. Then 
$$ 2i\lam^2 \ad \sig_3 (\bfM_\pm) - 2 \lam \bfQ_1 \bfM_\pm - \bfQ_2 \bfM_\pm \in \dee C(L^2). $$
\end{lemma}

\begin{proof}
In what follows we write $f_\pm \doteq g_\pm$ if $f_\pm - g_\pm \in \dee C(L^2)$. In this notation, we seek to prove that
$$ 
2i\lam^2 \ad \sig_3 (\bfM_\pm)  
	\doteq 2\lam \bfQ_1 \bfM_\pm + \bfQ_2 \bfM_\pm . 
$$

We compute
\begin{align*}
2i \lam^2 \ad \sig_3 (\bfM_\pm)
	&=	2i \lam^2 \ad \sig_3 \left(\bfM_\pm - \I \right)\\
	&=	\ad \sig_3 \left( 2i \lam^2 C_\pm f \right)
\end{align*}
where $f$ is given by \eqref{NLS:BC.post.pre.f} (but now $\mu$ and $w_x^\pm$ also depend on $t$). Using the identity
\begin{equation}
\label{RHP.Cpm.2}
\lam^2 \left(C_\pm h\right)(\lam)=
C_{\pm} \left( (\dotarg)^2 h(\dotarg) \right)(\lam)
- \frac{\lam}{2\pi i} \int h(s) \, ds - \frac{1}{2\pi i} \int sh(s) \, ds
\end{equation}
and identifying $m_j(x)$ with the moments of $f$, we conclude that
$$
\lam^2 C_\pm f \doteq \lam m_0(x,t) + m_1(x,t)
$$
so that
\begin{align}
\label{NLS:BC.t2.pre1}
2i\lam^2 \ad \sig_3 (\bfM_\pm)
	& \doteq	2\lam Q_1 + 2i \ad \sig_3 (m^{(1)})	\\
\nonumber
	& \doteq	2\lam Q_1  \bfM_\pm + 2\lam Q_1 (\I - \bfM_\pm) + 2i \ad \sig_3(m_1) \bfM_\pm
\end{align}
where we used the facts that $\I - \bfM_\pm \doteq 0$ and that $Q_1$ is a bounded function of $x$. We compute the second right-hand term in \eqref{NLS:BC.t2.pre1}:
\begin{align}
\label{NLS:BC.t2.pre2}
2\lam Q_1 \left( \I -\bfM_\pm \right) 
	&	\doteq	-2Q_1 \lam C_\pm f	\\
\nonumber
	&	\doteq	-2Q_1 m_0 \\
\nonumber
	&=	\begin{pmatrix}
				-i|q|^2	&	iq \dint_{\hspace{-1mm} \infty}^x |q|^2	\\[5pt]
				-i\qbar \dint_{\hspace{-1mm} \infty}^x |q|^2		&	i|q|^2
			\end{pmatrix} \\[10pt]
\nonumber
	&\doteq
			\begin{pmatrix}
				-i|q|^2		&	iq \dint_{\hspace{-1mm} \infty}^x |q|^2	\\[5pt]
				-i\qbar \dint_{\hspace{-1mm} \infty}^x |q|^2	&	i|q|^2
			\end{pmatrix}
			\bfM_\pm
\end{align}
Combining \eqref{NLS:BC.left.ad.m1}, \eqref{NLS:BC.t2.pre1}, and \eqref{NLS:BC.t2.pre2}, we conclude that 
$$ 
2i\lam^2 \ad \sig_3 (\bfM_\pm)  
	\doteq 2\lam \bfQ_1 \bfM_\pm + \bfQ_2 \bfM_\pm 
$$
as claimed.

\end{proof}

\subsection*{Exercises for Lecture 2}

\begin{exercise}
\label{ex:H01}
Show that if $f \in H^{0,1}(\R)$, then $f \in L^p(\R)$ for $1 \leq p \leq 2$ with
$$ \norm[p]{f} \leq \left( \int \, (1+x^2)^{-p/(2-p)} \, dx  \right)^{(2-p)/2p} \norm[H^{0,1}]{f}. $$
\end{exercise}

\begin{exercise}
\label{ex:H10}
Recall the space $H^1(\R)$ defined in \eqref{NLS:H1}.
Show that, if $f \in H^1(\R)$, then $f$ is bounded and H\"{o}lder continuous with $\norm[\infty]{f} \leq c \norm[H^{1}]{f}$
and $|f(x)-f(y)| \leq \norm[H^{1}]{f} |x-y|^{1/2}$. Show also that $H^{1}(\R)$ is an algebra, i.e., if $f, g \in H^1(\R)$, then $fg \in H^1(\R)$. 
\end{exercise}

\begin{exercise}
\label{ex:NLS.Cauchy.com}
Prove the identities \eqref{RHP.Cpm.1} and \eqref{RHP.Cpm.2}. You can either use the definition of 
$C_\pm$ as a limit of Cauchy integrals or use their definition as Fourier multipliers.
\end{exercise}

Exercises \ref{ex:NLS.Phi.map} -- \ref{ex:NLS.int} outline a proof of  local well-posedness for NLS viewed as the integral equation \eqref{NLS:int}.

\begin{exercise}
\label{ex:NLS.Phi.map}
Let $X = C((0,T);H^1(\R))$, the space of continuous $H^1(\R)$-valued functions on $(0,T)$. Fix $q_0 \in H^1(\R)$ and define a mapping $\Phi: X \to X$ by
$$ \Phi(q) = e^{it\Delta}q_0 - i \int_0^t e^{i(t-s)\Delta} \left( 2 |q(s)|^2 q(s) \right) \, ds. $$
Using the result of Exercise \ref{ex:H10}, show that the estimates
\begin{align*}
\norm[X]{\Phi(q)} 	
		&\leq 	\norm[H^1]{q_0} + 2c^2 T \norm[X]{q}^3	\\
\norm[X]{\Phi(q_1) - \Phi(q_2)} 
		&	\leq 2c^2 T \left( \norm[X]{q_1} + \norm[X]{q_2} \right)^2 \norm[X]{q_1-q_2}
\end{align*}
hold, where $c$ is the constant in the inequality of Exercise \ref{ex:H10}.
\end{exercise}

\begin{exercise}
\label{ex:NLS.int}
The solution of \eqref{NLS:int} is a fixed point for the map $\Phi(q)$. For $\alpha>0$, denote
by $B_\alpha$ the ball of radius $\alpha$ in $X$. 
\begin{enumerate}
\item[(i)] Show that for 
$\norm[H^1]{q_0} < \alpha/2$ and $2c^2 T < 1/(8\alpha^2)$ (i.e., $T$ sufficiently small depending on $\norm[H^1]{q_0}$), 
$\Phi$ maps $B_\alpha$ into itself.
\item[(ii)] Show that, under the same conditions, $\Phi$ is a contraction on $B_\alpha$.
\end{enumerate}
Conclude that, for $T$ sufficiently small,  $\Phi$ is a contraction on the ball of radius $\alpha$
and so has a unique fixed point.
\end{exercise}

\begin{exercise}
\label{ex:NLS.approx}
Prove Lemma \ref{lemma:NLS.approx}. \emph{Hints}: Note that $e^{it\Delta}$ is an isometry of $H^1(\R)$.  Use the fact that $H^1(\R)$ is an algebra (see Exercise \ref{ex:H10}) to conclude that $|q_n(s)|^2 q_n(s) \to |q(s)|^2 q(s)$ in $H^1$ uniformly in $s \in [0,T]$, and take limits in \eqref{NLS:int}.
\end{exercise}

\begin{exercise}
\label{ex:NLS.adsig}
Show that 
$$
\ad \sig_3 \twomat{a}{b}{c}{d} = \twomat{0}{2b}{-2c}{0}
$$
and conclude that $A \mapsto \ad \sig_3(A)$ is a linear map on $2 \times 2$ matrices with 
eigenvalues $2$, $0$, and $-2$.  Find the eigenvectors and show that
$$
\exp(t \ad \sig_3) \twomat{a}{b}{c}{d} = \twomat{a}{e^{2t} b}{e^{-2t}c}{d}. 
$$
Check that
\begin{equation}
\label{NLS:exp.ad}
\exp(t\ad \sig_3) (A) = e^{t\sig_3} A e^{-t\sig_3}.
\end{equation}
\end{exercise}

\begin{exercise}
\label{ex:NLS.ad.Leibniz}
Show that  $\ad \sig_3 (\dotarg)$ obeys the Leibniz rule
$$ \ad \sig_3 (AB) = \ad \sig_3(A) B + A \ad \sig_3 (B)$$
and use this to verify \eqref{NLS:RHP.x}.
\end{exercise}

\begin{exercise}
\label{ex:NLS:Jacobi}
Prove \emph{Jacobi's formula for differentiation of determinants}:
$$ \frac{d}{dx} \det \Phi(x) = 
	\sum_{i=1}^n \det
		\begin{pmatrix}
			\Phi_{1,1}(x)	&	\Phi_{1,2}(x)	&	\ldots		&	\Phi_{1,n}(x) \\[3ex]
			\Phi_{2,1}(x)	&	\Phi_{2,2}(x)	&	\ldots		&	\Phi_{2,n}(x) \\
			\vdots			&	\vdots			&				&	\vdots		\\
			\Phi_{i,1}'(x)	&	\Phi_{i,2}'(x)	&	\ldots		&	\Phi_{i,n}'(x)	\\
			\vdots			&	\vdots			&				&	\vdots		\\
			\Phi_{n,1}(x)	&	\Phi_{n,2}(x)	&	\ldots		&	\Phi_{n,n}(x)
		\end{pmatrix}		 .
$$ 
Show that if we define the \emph{adjugate matrix} of a nonsingular matrix $A$ by 
$$ A (\adj A)  = \det(A) I $$
(where $I$ is the $n \times n$ identity matrix),
then Jacobi's formula may be written
$$ \frac{d}{dx} \det \Phi(x) = \tr \left( \adj(\Phi(x)) \frac{d\Phi}{dx}(x) \right). $$
\end{exercise}

\begin{exercise}
\label{ex:NLS.det1}
Using Jacobi's formula, show that if $\Psi(t)$ is a differentiable, $N\times N$ matrix-valued function and $\Psi'(t) = B(t) \Psi(t)$ for a traceless matrix $B(t)$, then $\det \Psi(t)$ is independent of $t$.
\emph{Hint}: recall that $\Tr(AB) = \Tr(BA)$ for any $n\times n$ matrices $A$, $B$. 
\end{exercise}

\begin{exercise}
\label{ex:NLS.matmult}
Show that, if $\Psi_1$ and $\Psi_2$ are $2\times 2$ nonsingular matrix-valued solutions of $\calL \psi = z \psi$, then $\Psi_2^{-1} \Psi_1$ is independent of $x$.  
\end{exercise}

\begin{exercise}
\label{ex:NLS.M1M2}
Using the result of Exercise \ref{ex:NLS.matmult}, show that \eqref{M1M2A} holds for any two nonsingular solutions $\bfM_1$ and $\bfM_2$ of \eqref{NLS:M} (see \eqref{NLS:exp.ad}).
\end{exercise}

\begin{exercise}
\label{ex:NLS.sigma1}
Show that the map $\Psi \mapsto \sigma_1 \overline{\Psi(x,\zbar)} \sigma_1^{-1}$ preserves the solution space of $\calL \psi = z \psi$.
\end{exercise}

\begin{exercise}
\label{ex:NLS.mu.bc}
Using the fact that $M_\pm$ satisfy \eqref{NLS:M} for $z=\lam$, show that the same is true of $\mu$. You will need to use the Leibniz rule from Exercise \ref{ex:NLS.ad.Leibniz} together with the fact that $(d/dx) w_x^+ = -i\lam \ad \sig_3 (w_x^+)$. 
\end{exercise}

\newpage
\section{The Defocussing DS II Equation}
\label{sec:lec3}

In this lecture we will solve the defocussing Davey-Stewartson equation by inverse scattering method. The original lecture in August 2017 was  based on Perry's \cite{Perry:2016} earlier work, which solved the DS II equation for initial data in $H^{1,1}(\R^2)$. Subsequently, Nachman, Regev and Tataru \cite{NRT:2017} used the inverse scattering method to prove global well-posedness in $L^2(\R)$. In this lecture we will ``compromise'' by solving DS II in the space $H^{1,1}(\R^2)$ but use some of the tools introduced in \cite{NRT:2017} to simplify the proof. In particular, we will avoid entirely the resolvent expansions and multilinear estimates which make the proof in \cite{Perry:2016} somewhat complicated.

The DS II equation 
is
the nonlinear dispersive equation\footnote{We have rescaled $q$ to agree with the conventions of \cite{NRT:2017}.}
\begin{equation}
\label{DSII:bis}
\left\{
\begin{aligned}
i \dee_t q + 2\left( \dee_z^2 + \dee_{\zbar}^2 \right) q +  (g+ \gbar) q &= 0,\\
\dee_{\zbar} g + 4\eps \dee_z \left( |q|^2 \right) 	&=	0,\\
q(z,0)	&=	q_0(z)
\end{aligned}
\right.
%}
\end{equation}
where $\eps=+1$ for the defocussing equation, and $\eps = -1$ for the focusing equation, and
\begin{equation}
\label{DSII:dee-dbar}
 \dee_{\zbar} = \frac{1}{2} \left( \dee_{x_1} + i \dee_{x_2} \right) \quad
     \dee_{z} = \frac{1}{2} \left( \dee_{x_1} - i \dee_{x_2} \right).
\end{equation}
We will describe the formal inverse scattering theory for either sign of $\eps$, but only solve the defocussing case ($\eps = +1$)
for initial data in $H^{1,1}(\R^2)$. 
The DSII equation is the compatibility condition for the following system of equations:
\begin{align}
\label{DSII:Lax.z}
&
\left\{
\begin{aligned}
\dee_{\zbar} \psi_1 &=	q \psi_2\\
\dee_z \psi_2  &=	\eps \qbar \psi_1
\end{aligned}
\right.
%}
\\[5pt]
\label{DSII:Lax.t}
&
\left\{
\begin{aligned}
\dee_t \psi_1 &= \hphantom{-}2i\dee_z^2 \psi_1 
					+ 	2i\left(\dee_{\zbar} q\right)
					-	2iq \dee_{\zbar} \psi_2 + i g \psi_1\\
\dee_t \psi_2 &= -2i\dee_{\zbar}^2 \psi_2 
					-	2i\eps \left(\dee_z \qbar\right)\psi_1 
					+	2i\eps \qbar \dee_z \psi_1 - i \gbar \psi_2
\end{aligned}
\right.
%}
\end{align}

Motivated by the Lax representation \eqref{DSII:Lax.z}--\eqref{DSII:Lax.t} for the defocussing ($\eps=1$) DS II equation and the formal inverse scattering theory of section \ref{sec:DSII.integrability}, we
will establish the existence of a scattering transform $\calS: H^{1,1}(\R^2) \to H^{1,1}(\R^2)$ associated to the linear system \eqref{DSII:Lax.z} which linearizes the defocussing DS II equation. Using \eqref{DSII:Lax.t}, we will see that if $\bfs(t) = \calS q(t)$ for a solution $q(t)$ of the defocussing DS II equation, then $\bfs(t)$ obeys the  linear evolution equation 
$$\dot{\bfs}(k,t) = 2i(k^2+\kbar^2) \bfs(k,t).$$ 
We will show that the scattering transform $\calS$ satisfies $\calS^{-1} = \calS$ so that a putative solution to the defocussing DS II equation is given by
\begin{equation}
\label{DSII:qinv}
q_{\inv}(z,t) = 
		\calS 
			\left( 
					e^{	
							\left(
								2it\left( (\dotarg)^2 + \overline{(\dotarg)}^2\right)
							\right)
						}
					\left(\calS q_0\right)(\dotarg) 
			\right)(z)
\end{equation}
The mapping properties of $\calS$ established in section \ref{DSII:subsec.scatt} imply that $q_{\inv}(z,0) = q_0$ and that $(t,q_0) \mapsto q_{\inv}(\dotarg,t;q_0)$ is a continuous map from $(-T,T) \times H^{1,1}(\R^2)$ to $H^{1,1}(\R^2)$ for any $T>0$, Lipschitz continuous in $q_0$. We will then show that $q_{\inv}$ solves the DS II equation for initial data $q_0 \in \scrS(\R^2)$  by constructing solutions of the system \eqref{DSII:Lax.z}--\eqref{DSII:Lax.t}, where $q = q_\inv$, with prescribed asymptotic behavior. It will follow from Exercise \ref{ex:DSII.crossdiff} that $q_\inv$ solves the DS II equation for $q_0 \in \scrS(\R^2)$. The Lipschitz continuity of $\calS$ and local well-posedness theory for the DS II equation then imply that $q_\inv$ solves the integral equation form \eqref{DSII:int} of DS II for initial data $q_0 \in H^{1,1}(\R^2)$. 

To keep the exposition of reasonable length, we will take as given the results of Beals-Coifman \cite{BC:1985a,BC:1986,BC:1989} and Sung \cite{Sung:1994a,Sung:1994b,Sung:1994c} that the scattering transform $\calS$ maps $\scrS(\R^2)$ into itself. Our emphasis is on the estimates that extend the map $\calS$ to $H^{1,1}(\R^2)$ which enable us to apply the formula \eqref{DSII:qinv} to initial data in this space. One can use the techniques developed in these lectures to give a simpler proof Sung's results, but we will not carry this out here. 

\subsection{Preliminaries}
\label{DSII:subsec.prelim}

As already outlined in the first lecture, both the direct and inverse scattering transforms are defined via a system of $\dbar$ equations. In this subsection we collect some useful estimates on the solid Cauchy transform (see \eqref{DSII:solid-Cauchy}), the Beurling transform (see \eqref{Beurling}), and other useful integral operators.

The \emph{Hardy-Littlewood-Sobolev inequality} plays a fundamental role in the analysis of $\dbar$ problems and also in the proof of dispersive estimates in the local well-posedness theory for the DS II equation. For a proof, see for example \cite[Section 2.2]{LP:2015}. A sharp constant for the Hardy-Littlewood-Sobolev inequality together with an explicit maximizer is given in \cite{Lieb:1983}; see \cite{FL:2012} for a simplified proof of the optimal inequality.

\begin{theorem}[Hardy-Littlewood-Sobolev Inequality]
Suppose that $0< \alpha < n$, $1 < p < q < \infty$, and 
$$ \frac{1}{q} = \frac{1}{p} - \frac{\alpha}{n}. $$ 
If $f \in L^p(\R^n)$, the integral
$$ \left(I_\alpha f\right)(x) = \int \frac{f(y)}{|x-y|^{n-\alpha}} \, dy $$
converges absolutely for a.e.\ $x$, and the estimate
\begin{equation}
\label{HLS}
\norm[L^q]{I_\alpha(f)} \lesssim_{\, n, p, \alpha} \norm[L^p]{f}
\end{equation}
holds.
\end{theorem}

The \emph{solid Cauchy transform} is the integral operator
\begin{equation}
\label{DSII:solid-Cauchy}
\left(\dee_{\zbar}^{-1} f\right)(z) = \frac{1}{\pi} \int \frac{1}{z-w} f(w) \, dw 
\end{equation}
initially defined on $C_0^\infty(\R^2)$ and extended by density to $L^p(\R^2)$ for $p \in (1,2)$ by \eqref{DBAR:HLS}.
Proofs of the following fundamental estimates may be found, for instance, in \cite[Chapter I.6]{Vekua:1962} or \cite[section 4.3]{AIM:2009}. Some are exercises at the end of this section. We leave the formulation of similar results for the conjugate solid Cauchy transform
\begin{equation}
\label{DSII:conjugate-solid-Cauchy}
\left( \dee_z^{-1} f \right)(z) = \frac{1}{\pi} \int \frac{1}{\zbar - \wbar} f(w) \, dw
\end{equation}
to the reader.

(1) Fractional integration and $L^\infty$ estimates. Let $p \in (1,2)$ and let $p^*$ be the Sobolev conjugate exponent $({p^*})^{-1} = p^{-1} - 1/2$ for $n=2$. Then, as a consequence of the Hardy-Littlewood-Sobolev inequality \eqref{HLS},
\begin{equation}
\label{DBAR:HLS}
\norm[L^{p^*}]{\dee_{\zbar}^{-1} f} \lesssim_{\, p} \norm[L^p]{f}.
\end{equation}
On the other hand, an easy argument with H\"{o}lder's inequality (Exercise \ref{ex:DBAR.Vekua}) shows that for $1 < p < 2 < r < \infty$,
\begin{equation}
\label{DBAR:Vekua}
\norm[L^\infty]{\dee_{\zbar}^{-1} f} \lesssim_{\, q, r} \norm[L^p(\R^2) \cap L^r(\R^2)]{f}.
\end{equation}

(2) H\"{o}lder continuity and asymptotic behavior.  If $p \in (2,\infty)$, if $p'$ is the H\"{o}lder conjugate of $p$ and if $f \in L^p \cap L^{p'}$, 
then $\dee_{\zbar}^{-1} f$ is continuous and
\begin{equation}
\label{DBAR:Holder}
\left| \left(\dee_{\zbar}^{-1} f\right)(z) - \left(\dee_{\zbar}^{-1} f\right)(z') \right|
\lesssim_{\, p} \norm[L^p ]{f} \left| z-z' \right|^{1-2/p}.
\end{equation}
Again assuming $f \in L^p \cap L^{p'}$, 
$$
\lim_{|z| \rarr \infty} \left( \dee_{\zbar}^{-1} f\right)(z) = 0. 
$$

Next, we consider the model operator 
\begin{equation}
\label{DBAR:op}
S : f \rarr \dee_{\zbar}^{-1} \left( q f \right)
\end{equation}
which occurs in the analysis of the scattering transform.
An important consequence of \eqref{DBAR:HLS} is that for any $q \in L^2$ and any $p>2$, the 
operator $S$ is a bounded operator from $L^p$ to itself with operator bound
\begin{equation}
\label{S.est}
\norm[L^p \rarr L^p]{S} \lesssim_{\, p} \norm[L^2]{q},
\end{equation}
so that  $\ker_{L^p}(I-S)$ is trivial for $\norm[L^2]{q}$ sufficiently small. 

The  operator $S$ is also a compact operator.
Recall that a subset $V$ of a metric space is called \emph{precompact} if the closure of $V$ is compact, and  that a bounded operator $A$ on a Banach space $X$ is \emph{compact} if $A$ maps bounded subsets of $X$ into precompact subsets of $X$.  To prove that $S$ is compact, we first discuss the Kolmogorov-Riesz theorem that characterizes compact subsets of $L^p(\R^n)$. Our discussion draws on \cite{HH:2010,HH:2016} which provides a very readable exposition of the history and proof of this theorem. 

Recall that a metric space $(X,d)$ is said to be \emph{totally bounded} if, for any $\eps>0$, $X$ admits a finite cover by $\eps$-balls. 
A metric space is compact if and only if it is complete and totally bounded, and a subset of a metric space is precompact if and only if it is totally bounded.

\begin{theorem}[Kolmogorov-Riesz]
\label{thm:Kolmogorov-Riesz}
A subset $F$ of $L^p(\R^n)$ is totally bounded if, and only if:
\begin{itemize}
\item[(i)]		$F$ is bounded,
\item[(ii)]	(uniform decay) For every $\eps>0$ there is an $R>0$ so that 
$$\int_{|x|\geq R} |f(x)|^p \, dx < \eps^p \text{ for all } f \in F,$$ and
\item[(iii)]	($L^p$-equicontinuity) For every $\eps>0$ there is  a $\delta>0$ so that for every $f \in F$ and every $h \in \R^n$ with $|h| < \delta$, 
$$ \int_{\R^n} \left| f(x+h) - f(x) \right|^p \, dx < \eps^p. $$
\end{itemize}
\end{theorem}

\begin{lemma}
\label{lemma:S.compact}
The operator $S:L^p(\R^2) \rarr L^p(\R^2)$ is compact for any $q \in L^2(\R^2)$ and any $p>2$.
\end{lemma}

\begin{proof}
We need to show that, for any $p>2$ and any bounded subset $B$ of $L^p(\R^2)$, the set $$\{ Sf: f \in B\}$$ is totally bounded. Since $S$ is a bounded operator by \eqref{S.est}, (i) of Theorem \ref{thm:Kolmogorov-Riesz} is obvious. To prove (ii), let $\chi_R$ denote the characteristic function of the set $\{ x: |x| \leq R\}$. Then
\begin{multline}
\label{S.Holder}
(1-\chi_{4R})(x) (Sf) (x) = \\ (1-\chi_{4R}) \frac{1}{\pi} \int \frac{1}{x-y}
\left( \chi_R(y) + (1-\chi_R(y))\right) q(y) f(y) \, dy
\end{multline}
The first right-hand term of \eqref{S.Holder} is bounded by a constant times $R^{2/p-1} \norm[L^p]{f}$ where we used H\"{o}lder's inequality and the estimate 
$$(1-\chi_{4R}) |x-y|^{-1} \chi_R(y) \lesssim R^{-1}.$$ 
The second right-hand term is bounded by a constant times $\norm[L^2]{(1-\chi_R) q} \norm[L^p]{f}$. This shows that (ii) holds.

Finally, to show (iii), let $\eps>0$ be given. By Exercise \ref{ex:DSII.q.nice} we may write $q=q_n + q_s$ where 
$q_n$ is a smooth function of compact support and $\norm[L^2]{q_s} < \eps$. We may write
$$
(Sf)(x)	=	\frac{1}{\pi}\int \frac{1}{x-y} \left(q_n(y) + q_s(y)\right) f(y) \, dy = (S_n f)(x) + (S_s f)(x) 
$$
and estimate $\norm[L^p]{S_s f} \lesssim_{\, p} \eps \norm[L^p]{f}$ by \eqref{S.est}. On the other hand,  we may compute
$$
(S_n f)(x+h) - (S_n f)(x) 	=	h \int \frac{1}{x+h-y} \frac{1}{x-y} q_n(y) f(y) \, dy.
$$
It follows from Young's inequality that 
\begin{align*}
\norm[L^p]{(S_n f)(\dotarg + h) - (S_n f)(\dotarg)} 
	&	\leq	|h| \norm[L^{p'}]{(x+h)^{-1} x^{-1}} \norm[L^p]{q_n} \norm[L^p]{f} \\
	&	\lesssim	|h|^{1-2/p} \norm[L^p]{q_n} \norm[L^p]{f}.
\end{align*}
Hence
$$ 
\norm[L^p]{(Sf)(\dotarg+h) - (Sf)(\dotarg)} 
\lesssim_{\, p}
\left(2\eps+ |h|^{1-2/p} \norm[L^p]{q_n} \right)  \norm[L^p]{f}
$$
which implies the required bound. 
\end{proof}

Next, we will discus an estimate on fractional integrals due to Nachman, Regev, and Tataru. Our Theorem 
\ref{thm:DBAR.est} is a special case of \cite[Theorem 2.3]{NRT:2017}; as we will see, this estimate plays a critical 
role in the analysis of the scattering transform. We will give a simple direct proof of Theorem \ref{thm:DBAR.est} 
suggested by Adrian Nachman; in Exercise \ref{ex:NRT.frac}, we outline a complete proof of  \cite[Theorem 2.3]{NRT:2017} by the same method.

To state the estimate and introduce some key ingredients of Nachman's proof, we first recall that the Hardy-Littlewood maximal function for a locally integrable function $f$ on $\R^n$ is given by
$$
\calM f(x) = \sup_{r>0} \left( \frac{1}{\left| B(x,r) \right|} \int_{B(x,r)} |f(y)| \, dy \right) 
$$
where $B(x,r)$ is the ball of radius $r$ about $x \in \R^n$, and $\left| \dotarg \right|$ denotes Lebesgue measure. The maximal function is a bounded sublinear operator from $L^p(\R^n)$ to itself for $p \in (1,\infty]$ so that
\begin{equation}
\label{HLMax.p}
\norm[L^p]{\calM f} \lesssim_{\, p} \norm[L^p]{f}, \quad p \in (1,\infty].
\end{equation}
In particular, if $f \in L^p$ for $p \in (1,\infty]$, then $(\calM f)(x)$ is finite for almost every $x$.

Next, recall that an \emph{approximate identity} is a family of nonnegative functions $K_t \in L^1(\R^n)$, 
indexed by $t \in (0,\infty)$, with
\begin{itemize}
\item[(i)]		$\dint K_t(x) \, dx = 1$,
\smallskip
\item[(ii)]	$\left|K_t(x) \right| \leq  t^{-n}$, and
\smallskip
\item[(iii)]	$\left|K_t(x) \right| \leq At |x|^{-(n+1)}$.
\end{itemize}
It is not difficult to see that the estimate
$$ \left| \left(  K_t * f \right) (x)  \right| \lesssim \calM f(x) $$
holds, where the implied constant is independent of $t$. One example of an approximate identity is the
Poisson kernel
$$ P_t(x) = c_n \frac{t}{\left(|x|^2 + t^2\right)^{(n+1)/2}} $$
where $c_n$ is chosen to normalize the integral of $K_t$ to $1$. 

The action of the Poisson kernel by convolution may be viewed as the action of a Fourier multiplier with symbol
$e^{-t|\xi|}$. That is, denoting by $\calF$ the transform
$$ \left( \calF f \right)(\xi) = \int e^{-ix\cdot \xi} f(x) \, dx, $$ 
we have
\begin{equation}
\label{P.Fourier}
\calF \left[\left(P_t * f \right) \right] (\xi) = e^{-t|\xi|} \widehat{f}(\xi).
\end{equation}
Denote by $|D|^{-1}$ the Fourier multiplier with symbol $|\xi|^{-1}$. By the identity
$$ |\xi|^{-1} = \int_0^\infty e^{-t|\xi|} \, dt $$
it follows that 
$$ |D|^{-1} f = \int_0^\infty \left(P_t * f \right) (x) \, dt. $$

We can now state and prove:

\begin{theorem}\cite{NRT:2017}
\label{thm:DBAR.est}
Suppose that $p \in (1,2]$ and $f \in L^p(\R^2)$. The estimate
\begin{equation}
\label{DBAR.est}
\left| \left(\dee_{\zbar}^{-1} f \right)(x) \right| \lesssim \left( \calM f(x) \right)^{1/2} \left( \calM\widehat{f}(0) \right)^{1/2}
\end{equation}
holds, where $\widehat{f}$ denotes the Fourier transform of $f$.
\end{theorem}

\begin{proof}[Proof  \emph{(suggested by Adrian Nachman)}]
The Poisson kernel is an approximate identity so by standard theory
$$ \left| \left(P_t * f \right)(x) \right| \lesssim  \calM f(x) $$
with the implied constant independent of $t>0$. We now write
\begin{align*}
|D|^{-1} f (x) 	&=	\int_0^R \left(P_t * f \right)(x) \, dt + 
								\frac{1}{(2\pi) }\int_R^\infty e^{i\xi \cdot x} e^{-t|\xi|} \widehat{f}(\xi) \, d\xi\\
					&=	I_1(x) + I_2(x)
\end{align*}
where in the second term we used \eqref{P.Fourier}. We may estimate
\begin{align*}
|I_1(x)|			&\lesssim	R \calM f(x) \\
|I_2(x)|			&\lesssim	\int \frac{1}{|\xi|} e^{-R|\xi|} \left|\widehat{f}(\xi)\right| \, d\xi	\\
					&\lesssim	\sum_{j=-\infty}^\infty 2^j e^{-R 2^j} 
								2^{-2j} \int_{2^{j-1} < |\xi| < 2^j} \left| \widehat{f}(\xi) \right| \, d\xi\\
					&\lesssim \left( \sum_{j=-\infty}^\infty 2^j e^{-R 2^j} \right) \calM\widehat{f}(0)\\
					&\lesssim R^{-1} \calM\widehat{f}(0).
\end{align*}
where we introduced a dyadic decomposition in the $\xi$ variable. Thus
$$ \left| |D|^{-1} f (x) \right| \lesssim R \calM f(x) + R^{-1} \calM \widehat{f}(0). $$
Optimizing in $R$, we obtain the desired bound \eqref{DBAR.est}.
\end{proof}

We will usually use this estimate in the form
\begin{equation}
\label{DBAR.est.fourier}
\left| \left(\dee_{\zbar}^{-1} e_k f \right)(x) \right| \lesssim \left( \calM f(x) \right)^{1/2} \left( \calM\widehat{f}(k) \right)^{1/2}
\end{equation}
where
$$\widehat{f}(k) = \frac{1}{\pi} \int e_k(z) f(z) \, dz $$
is the natural Fourier transform in this setting. 

This estimate is of particular importance because it captures, in a quantitatively precise way, the effect of the oscillatory factor $e_k$ on the behavior of the fractional integral (In this context, see in particular Lemma \ref{lemma:mxk.est} and the subsequent analysis of the scattering transform in Section \ref{DSII:subsec.scatt}; in \cite{NRT:2017}, see particularly section 4).  It replaces less precise estimates, based on integration by parts, that were used in \cite{Perry:2016} to capture the behavior of solutions as a function of $k$.

The \emph{Beurling operator} is defined on $C_0^\infty(\R^2)$ as the principal value integral
\begin{equation}
\label{Beurling}
(\bfS f)(z) =  -\frac{1}{\pi} \lim_{\eps \darr 0} \int_{|z-w|>\eps} \frac{f(w)}{(z-w)^2} \, dw \end{equation}
and extends to bounded operator on $L^p(\R^2)$ for all $p \in (1,\infty)$. It is an isometry on $L^2$. We define
\begin{equation}
\label{Conjugate-Beurling}
\barS f = \overline{\bfS \overline{f}}. 
\end{equation}
The Beurling operator has the property that 
\begin{equation}
\label{Beurling-ddbar}
\bfS (\dee_{\zbar} f )= \dee_z f
\end{equation}
for functions $f \in C_0^\infty(\R^2)$. By density this extends property to functions $f \in W^{1,p}(\R^2)$ for $p \in (1,\infty)$. 
For a full discussion, see for example, \cite[Chapter 4]{AIM:2009}.

\subsection{Local Well-Posedness}
\label{sec:DSII.lwp}

Next, we review the local well-posedness theory for the DS II equation due to Ghidaglia and Saut \cite{GS:1990}.  The results in this subsection hold for either sign of $\eps$. We first recast \eqref{DSII:bis} as an integral equation using the solution operator $V(t)$ for the linear problem
\begin{equation}
\label{DSII:lin.bis}
i \dee_t v + 2 \left( \dee_z^2 + \dee_{\zbar}^2 \right) v = 0
\end{equation}
which is a linear dispersive equation. To formulate the integral equation, observe that \eqref{DSII:bis} may be reformulated as 
a nonlinear Schr\"{o}dinger-type equation with nonlocal nonlinearity:
\begin{equation}
\label{DSII:non-loc}
\left\{
\begin{aligned}
i \dee_t q + 2\left( \dee_z^2 + \dee_{\zbar}^2 \right) q + 4\eps \left( \bfS \left(|q|^2\right)  + \barS\left( |q|^2 \right)\right) q &= 0,\\
q(z,0)	&=	q_0(z)
\end{aligned}
\right.
%}
\end{equation}  
where $\bfS$ is the \emph{Beurling operator} \eqref{Beurling} and $\barS$ is the conjugate Beurling operator \eqref{Conjugate-Beurling}. 

We will say that a function $q \in C([0,T],L^2_z(\R^2)) \cap L^4(\R^2_z \times [0,T])$ solves the Cauchy problem \eqref{DSII:bis} if $q$ solves the integral equation
\begin{equation}
\label{DSII:int}
q(t) = V(t) q_0 + 4i\eps  \int_0^t V(t-s) \left[ q(s) \left(\bfS(|q|^2)(s) + \barS(|q|^2)(s) \right)  \right] \, ds
\end{equation}
as an integral equation in the space
\begin{equation}
\label{DSII:X}
X=C((0,T), L^2(\R^2)) \cap L^4(\R^2 \times (0,T)).
\end{equation}
This integral equation is motivated by Duhamel's formula (see Exercise \ref{ex:DSII.Duhamel}) and  makes sense in this space because of the Strichartz estimates discussed below. Ghidaglia and Saut \cite[Theorem 2.1]{GS:1990} prove:

\begin{theorem}
\label{thm:DSII.GS}
For any $q_0 \in L^2(\R^2)$, there is a $T^*>0$ and a unique solution $q(t)$ to \eqref{DSII:int} belonging to 
$C((0,T^*),L^2(\R^2)) \cap L^4(\R^2 \times (0,T^*))$ with $q(t)=q(0)$ and $\norm[L^2]{q(t)} = \norm[L^2]{q_0}$.
\end{theorem}

Note that the proof of Theorem \ref{thm:DSII.GS} is insensitive to the sign of $\eps$, but does not guarantee global existence. This is to be expected since there are solutions of the focussing ($\eps=-1$) DS II equation whose $L^2$-mass concentrates to a point in finite time \cite{Ozawa:1992}.

The idea of the proof is to show that the mapping
\begin{equation}
\label{DSII:Phi}
\Phi(u) = V(t) q_0 +4i\eps \int_0^t V(t-s) \left[ q(s) \left(\bfS(|q|^2)(s) + \barS(|q|^2)(s) \right)  \right] \, ds 
\end{equation}
is a contraction on the space $X$
for some $T>0$ depending on the initial data $q_0$. One can reconstruct a complete proof by tracing through standard arguments used to show that the $L^2$-critical nonlinear Schr\"{o}dinger equation 
$$ iu_t +\Delta u - |u|^2 u = 0 $$
in two space dimensions is locally well-posed (see for example the text of Ponce and Linares \cite[Section 5.1]{LP:2015} or the original paper of Cazenave and Weissler \cite{CW:1989}); dispersive estimates for $V(t)$ are essentially the same as those for the unitary group $\exp(it\Delta)$, while the nonlinear term in \eqref{DSII:non-loc} is ``morally cubic'' owing to the fact that $\barS$ preserves $L^p(\R^2)$ for any $p \in (1,\infty)$.  We will give an outline based on \cite[Section 5.1]{LP:2015}.

To carry out the proof of Theorem \ref{thm:DSII.GS}, we will need 
the following Strichartz estimates on $V(t)$. 

\begin{proposition}
Let $V(t)$ be the solution operator for the linear equation \eqref{DSII:lin.bis}. The following estimates hold.
\begin{align}
\label{DSII:S1}
\norm[L^4_{z,t}]{V(t)f} &\lesssim \norm[L^2]{f}	\\
\label{DSII:S2}
\norm[L^4_{z,t}]{\int_{-\infty}^\infty V(t-s) g(s) \, ds} &\lesssim \norm[L^{4/3}_{z,t}]{g} \\
\label{DSII:S3}
\norm[L^2_z]{\int_{-\infty}^\infty V(t) g(t) \, dt} &\lesssim \norm[L^{4/3}_{z,t}]{g}
\end{align}
\end{proposition}

These estimates are consequences of the basic dispersive estimate
\begin{equation}
\label{DSII:U.disp}
\norm[L^\infty]{V(t)f} \lesssim t^{-1} \norm[L^1]{f} 
\end{equation}
which follows from the representation of $V(t)f$ as a Fourier integral (Exercise \ref{ex:DSII.U}). 
One can prove \eqref{DSII:S1}--\eqref{DSII:S3} for $V(t)$ by mimicking the proof of  the analogous estimates for $V(t)$ replaced by $e^{it\Delta}$, the solution semigroup for the Schr\"{o}dinger equation in two space dimensions, given in  \cite[Section 4.2]{LP:2015}. The proofs are essentially identical since $\exp(it\Delta)$ and $V(t)$ both obey the basic dispersive estimate \eqref{DSII:U.disp}. The reader is asked to prove \eqref{DSII:S2} in Exercise \ref{ex:DSII.S2}.

The first step in the proof of Theorem \ref{thm:DSII.GS} is to show that the mapping \eqref{DSII:Phi} preserves a ball in $X$. Suppose that $\norm[X]{q} < \alpha$.
Using the Strichartz estimate \eqref{DSII:S3} on the second right-hand term of \eqref{DSII:Phi} and the fact that $V(t)$ is unitary on $L^2$ on the first right-hand term, we may estimate
\begin{align*}
\sup_{t \in (0,T)}\norm[L^2]{\Phi(q(t))} 	
	&	\lesssim \norm[L^2]{q_0} + T^{1/4} \norm[L^4(\R^2 \times (0,T))]{q}^3\\
	&	\lesssim \norm[L^2]{q_0} + T^{1/4} \alpha^3
\end{align*}
where in the first step we used \eqref{DSII:S3} and then used H\"{o}lder's inequality in the integration over $t$. Similarly, using \eqref{DSII:S1} and  \eqref{DSII:S2} respectively on the first and second right-hand terms of \eqref{DSII:Phi}, we obtain an estimate of the same form for
$\norm[L^4(\R^2 \times (0,T)]{\Phi(q)}$. Hence, for any $q$ with $\norm[X]{q} < \alpha$, 
$$
\norm[X]{\Phi(q)} \lesssim \norm[L^2]{q_0} + T^{1/4} \alpha^3.
$$
Choosing $\alpha \gtrsim \norm[L^2]{q_0}$ and $T^{1/4} \lesssim \alpha^{-2}$, we obtain that $\norm[X]{\Phi(q)} \leq \alpha/2$. Note that the `guaranteed' time of existence decreases with the $L^2$ norm of the initial data.

The next step is to show that $\Phi$ is a contraction in the sense that 
$$ \norm[X]{\Phi(q_1) - \Phi(q_2)} \leq \frac{1}{2} \norm[X]{q_1 - q_2} $$
for $T$ sufficiently small. Using the Strichartz estimates and the multilinearity of the map 
$$ (q_1, q_2, q_3 ) \mapsto q_1 \barS \left( q_2 q_3 \right) $$
we have
$$ \norm[X]{\Phi(q_1) - \Phi(q_2)} \lesssim  T^{1/4} \alpha^2 \norm[X]{q_1- q_2}$$
for any $q_1, q_2$ with $\norm[X]{q_1}, \norm[X]{q_2} < \alpha$. By shrinking $T$ if necessary we 
can assure that $\Phi$ is a contraction, and hence \eqref{DSII:int} has a unique solution.

\subsection{Complete Integrability}
\label{sec:DSII.integrability}

In this subsection we will sketch the formal inverse scattering theory for the DS II equations, 
tacitly assuming that $q(\dotarg,t) \in \scrS(\R^2)$ and that the scattering transform $\bfs \in \scrS(\R^2)$, so that various asymptotic expansions make sense. Sung \cite{Sung:1994a,Sung:1994b,Sung:1994c} proved rigorously that the scattering transform $q \mapsto \bfs$ maps $\scrS(\R^2)$ to itself.  It follows from these mapping properties that the putative solution $q_{\inv} $ defined by \eqref{DSII:qinv} belongs to $ C((-T,T);\scrS(\R^2))$ for any $T>0$. 
These facts imply that the functions $m^1(z,t,k)$ and $m^2(z,t,k)$ which we will construct below are bounded smooth functions with asymptotic expansions separately in $z$ for fixed $k$ or in $k$ for fixed $z$. 

Equation \eqref{DSII:bis} is the compatibility condition for the following system of equations for unknowns $\psi_1(z,t,k)$ and $\psi_2(z,t,k)$:
\begin{subequations}
\begin{align}
\label{DSII:x1}
\dee_{\zbar}	\psi_1	&=	q \psi_{2}\\
\label{DSII:x2}
\dee_z			\psi_2	&=	\eps \qbar \psi_{1}
\end{align}
\end{subequations}
\begin{subequations}
\begin{align}
\label{DSII:t1}
\dee_t			\psi_1	&=	2i \dee_z^2 \psi_{1}  + 2i (\dee_{\zbar} q) \psi_2 
										 -2 iq \dee_{\zbar} \psi_2  + ig \psi_1 \\
\label{DSII:t2}
\dee_t			\psi_2	&=	-2i \dee_{\zbar}^2 \psi_2  -2i\eps(\dee_z \qbar) \psi_1
										+ 2i\eps \qbar \dee_z \psi_1 -i\gbar \psi_2 
\end{align}
\end{subequations}
%\smallskip

\noindent
Cross-differentiating \eqref{DSII:x1} and \eqref{DSII:t1}, assuming $(\psi_1,\psi_2)$ is a joint solution and that $\psi_1$ and $\psi_2$ are linearly 
independent, one finds that the DSII equation 
\begin{equation}
\label{DSII:bis2}
\left\{
\begin{aligned}
i q_t + 2\left( \dee_z^2 + \dee_{\zbar}^2 \right) q + (g + \gbar) q	&=	0\\
\dee_{\zbar} g + 4\eps \dee_z \left(|q|^2\right)	&=0	
\end{aligned}
\right.
%}
\end{equation}
emerges as a compatibility condition
(see Exercise \ref{ex:DSII.crossdiff}).

To define and implement the scattering transform, we'll consider solutions of \eqref{DSII:x1}--\eqref{DSII:x2} with asymptotics specified by a complex parameter $k$: we seek solutions of the form\footnote{We follow the conventions of \cite{NRT:2017} and denote the renormalized forms of $\psi_1$ and $\psi_2$ respectively by $m^1$ and $m^2$; the superscripts are not exponents!}
$$ \psi_1 = C_1(k,t) e^{ikz} m^1, \quad \psi_2 = C_2(k,t) e^{ikz} m^2 $$
where for each fixed $t$ and $k$, 
\begin{equation}
\label{DSII:m.x.asy}
\left( m^1(z,k,t), m^2(z,k,t) \right) \to (1,0) \text{ as } |z| \to \infty. 
\end{equation}
A calculation similar to the one carried out in section \ref{NLS:sec.sol}
shows that
$$ C_1(k,t) = C_2(k,t) = e^{-2ik^2 t}. $$ 
We outline the computation in Exercise \ref{ex:DSII.time}.
In the new variables, we find
\begin{subequations}
\begin{align}
\label{DSII:m.x1}
\dee_{\zbar} m^1					&=	\ q m^2	\\
\label{DSII:m.x2}
\left(\dee_z + ik \right)	m^2   	&=	\eps \qbar m^1\\[10pt]
\label{DSII:m.t1}
\dee_t m^1			&=	2i\left(\dee_z^2 + 2ik\dee_z \right) m^1
									+ 2i \left(\dee_{\zbar} q\right) m^2 
									- 2iq \dee_{\zbar} m^2 +	ig m^1  \\[5pt]
\label{DSII:m.t2}
\dee_t m^2			&=	-2i \dee_{\zbar}^2 m^2 + 2ik^2 m^2  
									- 2i\eps \left( \dee_z \qbar\right) m^1 +2i\eps\qbar (\dee_z + ik) m^1 \\
						&\qquad 
									- i\gbar m^2 
\nonumber
\end{align}
\end{subequations}
As we will show (see Lemma \ref{lemma:DSII.m.k}), for each fixed time $t$ and position $z$, the solutions of \eqref{DSII:m.x1}--\eqref{DSII:m.x2} obeying the asymptotic condition \eqref{DSII:m.x.asy} also obey the dual equations
\begin{equation}
\label{DSII:m.k}
\left\{
\begin{aligned}
\dee_{\kbar} m^1	&=	e_{-k} \bfs \overline{m^2}\\
\dee_{\kbar} m^2	&=	e_{-k} \bfs \overline{m^1}\\
\left(m^1(z,k,t), m^2(z,k,t) \right) &\to (1,0) \text{ as } |k| \to \infty.
\end{aligned}
\right.
%}
\end{equation}
where 
$$
e_k(z) = e^{i\left(kz + \kbar \zbar\right)}
$$
and the scattering transform $\bfs(k,t)$ of $q(z,t)$ is defined by
\begin{equation}
\label{DSII:m2.asy.z}
m^2(z,k,t) = e_{-k} \frac{i}{z} \bfs(k,t) + \bigO{|z|^{-2}}. 
\end{equation}
Assuming that $\bfs(\dotarg,t) \in \scrS(\R^2)$ and that $m^1$ and $m^2$ are bounded, it follows from \eqref{DSII:m.k} that $m^1$ and $m^2$ have large-$k$ asymptotic expansions of the form (see Exercise \ref{ex:DBAR.asy})
$$
\left\{
\begin{aligned}
m^1(z,k,t)	&\sim	1	+ \sum_{j \geq 1} \frac{\alpha_j(z,t)}{k^j}\\
m^2(z,k,t)	&\sim	\sum_{j \geq 1} \frac{\beta_j(z,t)}{k^j}
\end{aligned}
\right.
%}
$$
for each fixed $z$, $t$. 
Substituting these expansions into \eqref{DSII:m.x1}--\eqref{DSII:m.x2} shows that 
\begin{equation}
\label{DSII:q.recon.pre}
q(z,t)	=	-i\eps \overline{\beta_1(z,t)}
\end{equation}
(see Exercise \ref{ex:DSII.m.asy.k}). 

Thus, to recover $q(z,t)$, 
we need (i) an equation of motion for the scattering transform $\bfs(k,t)$ and (ii) a way of reconstructing $m^1(z,k,t)$ and $m^2(z,k,t)$ from the scattering transform $\bfs(k,t)$.

We can derive an equation of motion for $\bfs$ formally as follows. If we assume that $q(\cdot,t) \in \calS(\R^2)$, 
we expect $m^1$ and $m^2$ to have large-$z$ (differentiable) asymptotic expansions of the 
form
\begin{equation}
\label{DSII:m.asy.z}
\left\{
\begin{aligned}
m^1(z,k,t)	&\sim	1	+ 	\frac{a(k,t)}{z} + \bigO{|z|^{-2}}\\[5pt]
m^2(z,k,t)	&\sim			e_{-k}(z) \frac{b(k,t)}{\zbar} + \bigO{|z|^{-2}}
\end{aligned}
\right.
%}
\end{equation}
Note that, comparing the second equation of \eqref{DSII:m.asy.z} with \eqref{DSII:m2.asy.z}, we have $b(k,t) = i\bfs(k,t)$.
Substituting these expansions into \eqref{DSII:m.t1}--\eqref{DSII:m.t2} and taking $|z| \to \infty$, we 
see that
$$ \dot{a}(k,t) = 0, \quad \dot{b}(k,t) = 2i\left(k^2 + \kbar^2\right) b(k,t). $$
Thus, formally, the map $q \mapsto (a,b)$ gives action-angle variables for the flow \eqref{DSII:bis2}. 
In particular, if $q(z,t)$ solves the DSII equation, then the scattering data obeys the linear evolution
$$ \bfs(k,t) = e^{2it\left( k^2 + \kbar^2\right)} \bfs(k,0). $$

It remains to show how $q(z,t)$, the solution of the DSII equation, may be recovered from $\bfs(k,t)$. Here
we use the fact that $m^1$ and $m^2$, now regarded also as functions of time, obey the equations
\begin{equation}
\label{DSII:m.inv.t}
\left\{
\begin{aligned}
\dee_{\kbar} m^1		&=	e^{it\varphi} \bfs \overline{m^2}\\
\dee_{\kbar} m^2	&=	e^{it\varphi} \bfs \overline{m^1}\\
m^1(z,k) - 1, m^2(z,k) &\to 0 \text{ as } |k| \to \infty
\end{aligned}
\right.
%}
\end{equation}
where $\bfs(k)$ is the scattering transform of the initial data $q(z,0)$, and 
\begin{equation}
\label{DSII:phase}
\varphi(z,k,t) = 2\left(k^2 + \kbar^2\right) - \frac{kz + \kbar \zbar}{t}
\end{equation}
is a phase function formed from $e_{-k}$ and the evolution for $\bfs$. We can then reconstruct
$q(z,t)$ from the asymptotics of $m^2(z,k,t)$ using \eqref{DSII:q.recon.pre}.

The proof that $q(z,t)$ so defined in fact solves \eqref{DSII:bis} uses the Lax representation \eqref{DSII:m.x1}--\eqref{DSII:m.t2}. In the case $\eps=1$, we will show that
$m^1(z,k,t)$ and $m^2(z,k,t)$ defined by \eqref{DSII:m.inv.t}
generate a solution of the Lax equations \eqref{DSII:x1}--\eqref{DSII:t2} where $q(\dotarg,t)$ is the scattering transform of $\bfs(\dotarg,t)$. It will then follow that $q(z,t)$, defined as $\calS\left( \bfs(\dotarg,t) \right)$, solving the DSII equation. 

In what follows, we will study the scattering transform $\calS$ in depth to obtain the Lipschitz mapping property (section \ref{DSII:subsec.scatt}).  In order to prove the solution formula \eqref{DSII:IST-sol}, it suffices to check initial data $q_0 \in \scrS(\R^2)$. In section \ref{DSII:subsec.time}, we use the Lax representation \eqref{DSII:Lax.z}--\eqref{DSII:Lax.t} to show that \eqref{DSII:IST-sol} does indeed generate a solution to \eqref{DSII}.

\subsection{The Scattering Map}
\label{DSII:subsec.scatt}

We now define the scattering transform $\calS: q \rarr \bfs$ more precisely. Given $q \in H^{1,1}(\R^2)$ and $k \in \C$, one first solves the linear system
\begin{equation}
\label{DSII:m}
\left\{
\begin{aligned}
\dee_{\zbar} m^1(z,k)				&=	q(z) m^2(z,k)\\
\left( \dee_z + ik \right) m^2(z,k)	&=	\overline{q(z)} m^1(z,k)\\
m^1(\dotarg,k) - 1, \,\,\, m^2(\dotarg,k) &\in L^4(\R^2).
\end{aligned}
\right.
%}
\end{equation}
One then computes the scattering transform from the integral representation
\begin{equation}
\label{DSII:scatt}
\bfs(k) = \left(\calS q \right)(k) \coloneqq -\frac{i}{\pi} \int_{\R^2} e_k(z) \overline{q(z)} m^1(z,k) \, dz.
\end{equation}
This definition accords with the definition \eqref{DSII:m2.asy.z} given by asymptotic expansion of $m_2(z,k,t)$
if $q \in \scrS(\R^2)$ because one can compute the first term in the large-$z$ asymptotic expansion for $m^2(z,k,t)$ explicitly (Exercise \ref{ex:DSII.two-s}; in keeping with the emphasis of this section, $t$-dependence is suppressed).
Note that, in this normalization, $\calS$ is an antilinear map. Its linearization at $q=0$ is an ``antilinear Fourier transform''
\begin{equation}
\label{DSII:scatt.lin}
\widehat{f}(k) = \left( \calF_a f \right)(k) \coloneqq -\frac{i}{\pi} \int_{\R^2} e_k(z) \overline{f(z)} \, dz.
\end{equation}
It is easy to check, using standard Fourier theory, that $\calF_a = \calF_a^{-1}$ defines an isometry from $L^2(\R^2)$ onto itself and a Lipschitz continuous map from $H^{1,1}(\R^2)$ onto itself (Exercise \ref{ex:DSII.F}). Thus, 
\begin{equation}
\label{S.split}
\left( \calS q \right)(k) = \left(\calF_a q\right)(k) - \frac{i}{\pi} \int e_k(z) \overline{q(z)}	 \left(m^1(z,k) - 1 \right) \, dz.
\end{equation}

Equation \eqref{S.split} provides a useful way to understand the scattering transform: it is a perturbation of the linear Fourier transform in which the integral transform also depends on $q$. In \cite{NRT:2017}, the authors exploit the fact that the second term may be viewed as a pseudodifferential operator whose mapping properties can be controlled by estimates on the `symbol' $a(x,\xi) = m^1(\xi,x)-1$ (the reversal of arguments $(x,\xi)$ in $m^1$ is deliberate!).

We will prove:

\begin{theorem}
\label{thm:DS.lip}
The scattering transform $\calS$ is a locally Lipschitz continuous map from $H^{1,1}(\R^2)$ onto itself. Moreover, 
$\calS = \calS^{-1}$ 
\end{theorem}

\begin{proof}
We begin with a reduction. Suppose we can prove that $\norm[H^{1,1}]{\calS q_1 - \calS q_2} \lesssim \norm[H^{1,1}]{q_1-q_2}$ for $q_1, q_2 \in \scrS(\R^2)$ with constants uniform in $q_1, q_2 \in \scrS(\R^2)$ having $H^{1,1}(\R^2)$ norm bounded by a fixed constant, We can then extend the map $\calS$ by density to a nonlinear mapping on $H^{1,1}(\R^2)$ with the same continuity properties. Similarly, if $\calS = \calS^{-1}$ on $\scrS(\R^2)$, this identity extends by density to $H^{1,1}(\R^2)$. 

The claimed mapping properties of $\calS$ for $q_1, q_2 \in \scrS(\R^2)$ are proved in Propositions \ref{prop:DSII.L2}, \ref{prop:DSII.H1}, and \ref{prop:DSII.L21} of what follows. The property $\calS = \calS^{-1}$ on $\scrS(\R^2)$ is proved in Proposition \ref{prop:DSII.Sinv}. 
\end{proof}

The proofs of Propositions \ref{prop:DSII.L2}, \ref{prop:DSII.H1},  \ref{prop:DSII.L21}, and \ref{prop:DSII.Sinv} rest on a careful analysis of the solutions to \eqref{DSII:m}. Some of the results along the way are proved for $q \in L^2(\R^2)$ or $q \in H^{1,1}(\R^2)$. Although we follow the outline of \cite{Perry:2016}, we use ideas of \cite{NRT:2017} at a number of points to simplify the proofs. 

\subsubsection{Existence and Uniqueness of Solutions}

First, we will show that \eqref{DSII:m} has a unique solution for each $q \in H^{1,1}(\R^2)$ and $k \in \C$. The following ``vanishing theorem'' for $\dbar$-problems is originally due to Vekua \cite{Vekua:1962}, was used by Beals-Coifman \cite{BC:1985a}, and was improved to the form stated here by Brown and Uhlmann \cite{BU:1997}. The short and elegant proof we give here is taken from the paper of Nachman, Regev, and Tataru \cite[proof of Lemma 3.2]{NRT:2017}. 

\begin{theorem}
\label{thm:Liouville}
Suppose that $a \in L^2(\R^2)$, $u \in L^p(\R^2)$ for some $p>2$, and $\dee_{\zbar} u = a \ubar$ in distribution sense. Then $u=0$. 
\end{theorem}

\begin{proof} \cite{NRT:2017}
Define
$$
a_n(x)	=
\begin{cases}
a(x),	&	|x| < n \text{ and } |a(x)| < n\\
\\
0,		&	\text{otherwise}
\end{cases}
$$ 
and $a_s = a - a_n$. We then have
$a=a_n + a_s$ where $a_n \in L^{p_1} \cap L^{p_2}$ for some $1 < p_1 < 2 < p_2$ and $\norm[L^2]{a_s}$ is small for $n$ large. Let 
$$ 
\nu(z) = 	\begin{cases}
				\exp\left( -\dee_{\zbar}^{-1} a_n \dfrac{\ubar}{u} \right)(z), 	&	u(z) \neq 0\\
				\\
				1,																			&	u(z) = 0

			\end{cases}
$$
The function $u \nu$ obeys
$$ \dee_{\zbar} (u \nu) = a_s (u \nu) $$
so that, choosing $n$ large enough, we may conclude that $u \nu = 0$, by the remarks following \eqref{DBAR:op}. On the other hand, $\nu$ and $\nu^{-1}$ belong to $L^\infty$ by \eqref{DBAR:Vekua}. Hence, $u=0$.
\end{proof}

A short computation using the operator identity 
\begin{equation}
\label{dee-id}
(\dee_z + ik) = e_{-k} \dee_z e_k
\end{equation} 
shows that the functions
\begin{equation}
\label{DBAR:mpm}
m_\pm(z,k)	=	m^1(z,k) \pm e_{-k} \overline{m^2(z,k)}
\end{equation}
solve the system
\begin{equation}
\label{DSII:mpm}
\left\{
\begin{aligned}
\dee_{\zbar} m^\pm(z,k)	&=	\pm e_{-k} q(z) \overline{m^\pm(z,k)},\\
m_\pm - 1						&\in   L^4(\R^2)
\end{aligned}
\right.
%}
\end{equation}

\begin{proposition}
\label{prop:DBAR.mpm}
There exists a unique solution of \eqref{DSII:m} for any $k \in \C$ and $q \in L^2(\R^n)$. 
\end{proposition}

\begin{proof}
We prove uniqueness first. Suppose that $(m^1,m^2)$ and $(n^1,n^2)$ solve \eqref{DSII:m} for $q \in L^2$. We claim that $m^1=n^1$ and $m^2=n^2$.
Setting $$w^1 = m^1-n^1, \quad w^2 = m^2 - n^2,$$ we obtain a solution $(w^1,w^2)$ of \eqref{DSII:m} with $w^1(\dotarg,k)$ and $w^2(\dotarg,k)$ in $L^4(\R^2)$, so that the same is true of $w_\pm$ under the change of variable \eqref{DBAR:mpm}. By Theorem \ref{thm:Liouville}, $w_\pm = 0$, so $(w^1,w^2)=0$.

In order to prove existence of solutions to \eqref{DSII:m}, it suffices to solve \eqref{DSII:mpm}. To this end, consider the equation
\begin{equation}
\label{DSII:m.model}
\dee_{\zbar} w +  e_{-k}u \wbar = - e_{-k} u
\end{equation}
where one should think of $w$ as $m^\pm -1$ and $u$ as $\mp q$. This equation is equivalent to the integral equation
$$
w - T w = \dee_{\zbar}^{-1} \left( e_{-k} u \right)
$$
where 
$$
T f = \dee_{\zbar}^{-1}\left( e_{-k} u \fbar \right).
$$
The operator $T$ is the composition of the operator $S$ with complex conjugation (the factor $e_{-k}$ can be absorbed into the definition of $q$ in \eqref{DBAR:op}). Hence, by Lemma \ref{lemma:S.compact}, $T$ is a compact operator.  
Since $T$ is compact, $(I-T)$ is a Fredholm operator.

We claim that,  by Theorem \ref{thm:Liouville}, $\ker_{L^4}(I-T)$ is trivial. If so, it follows from the Fredholm alternative that $(I-T)^{-1}$ is a bounded operator on $L^4$ and that
\begin{equation}
\label{DSII:m.model.sol}
w =(I-T)^{-1} \left(\dee_{\zbar}^{-1}(e_{-k} u)\right)
\end{equation}
solves \eqref{DSII:m.model}. Suppose then that $f \in \ker_{L^4}(I-T)$, i.e., $Tf=f$. Then $f$ is a weak solution of the equation $\dee_{\zbar} f = e_{-k} q \overline{f}$ and hence, by Theorem \ref{thm:Liouville}, $f=0$. This finishes the proof.
\end{proof}

We end this subsection with a resolvent estimate on $(I-T)^{-1}$. This is one of the key points where we use the smoothness of $q \in H^{1,1}(\R^2)$. Very different techniques are used in \cite[\S 3]{NRT:2017} to control the resolvent \ assuming only that $q \in L^2(\R^2)$.

We will exploit the integration by parts formula
\begin{equation}
\label{Bukhgeim}
\frac{1}{\pi} \int \frac{1}{z-w} e_{-k} f(w) \, dw =
		-\frac{e_{-k}(z) f(z)}{i\kbar} + 
		\frac{1}{i\kbar} \int \frac{1}{z-w} e_{-k}(w) (\dee_{\zbar} f)(w) \, dw 
\end{equation}
(see Exercise \ref{ex:DSII.Bukhgeim}).

From this identity it follows that
\begin{equation}
\label{T.ip}
(Tf)(z) = -\frac{1}{i\kbar} e_{-k}(z) q(z) \overline{f(z)} 
	+ \frac{1}{i\kbar} \int \frac{1}{z-w} e_{-k}(w) \dee_{\zbar} (q\fbar )(w) \, dw.
\end{equation}
Using the estimate \eqref{DBAR:HLS} and \eqref{T.ip}, we see that
$$
\norm[L^4]{Tf} \lesssim \frac{1}{|k|} \left( \norm[L^4]{q \fbar} + \norm[L^{4/3}]{\dee_{\zbar}\left(q\fbar \right)}\right)
$$
so that 
\begin{align*}
\norm[L^4]{T^2 f} 
	&	\lesssim \frac{1}{|k|} 
			\left( \norm[L^4]{qTf} + \norm[L^{4/3}]{\dee_{\zbar}\left(qTf\right)} \right)	\\
	&	\lesssim	 \frac{1}{|k|} 
			\left( 
				\norm[L^8]{q} \norm[L^8]{Tf} + \norm[L^2]{\dee_{\zbar} q} \norm[L^4]{Tf} + \norm[L^4]{q}^2 \norm[L^4]{f} 
			\right) \\
	&	\lesssim \frac{1}{|k|}
			\left(
				\norm[L^8]{q} \norm[L^{8/3}]{q} + \norm[L^2]{\dee_{\zbar} q} \norm[L^2]{q} + \norm[L^4]{q}^2 
			\right) \norm[L^4]{f}
\end{align*}
Since $H^1(\R^2)$ is continuously embedded in $L^p(\R^2)$ for all $p\geq 2$ (see Exercise \ref{ex:DSII.H2}), it follows that
\begin{equation}
\label{T2.est}
\norm[L^4 \rarr L^2]{T^2} \lesssim \frac{1}{|k|} \norm[H^1]{q}^2
\end{equation}

From \eqref{T2.est} and the identity $(I-T) = (I-T^2)^{-1}(I+T)$, we immediately obtain the following large-$k$ resolvent estimate.

\begin{lemma}
\label{lemma:resT.est1}
Fix $R>0$.  There is an $N=N(R)$ so that for all $k \in C$ with $|k| \geq N$ and all $q \in H^1(\R^2)$ with 
$\norm[H^1]{q} \leq R$, the estimate
$\norm[L^4 \to L^4]{(I-T)^{-1}} \leq 2$ holds.
\end{lemma}

To obtain uniform resolvent estimates (i.e., estimates valid for all $k \in \C$ and $q$ in a bounded subset of $H^{1,1}(\R^2)$),  we now follow the ideas of \cite{Perry:2016}. Using a different approach, Nachman, Regev, and Tataru obtain similarly uniform estimates for $q$ in a bounded subset of $L^2$ (see \cite[Section 3]{NRT:2017}). 

In our case, Lemma \ref{lemma:resT.est1} gives uniform control for $q$ in a bounded subset of $H^{1,1}(\R^2)$ and sufficiently large $|k|$. It remains to control the resolvent for $(k,q)$ with $k$ in a bounded subset of $\C$ and $q$ in a bounded subset of $H^{1,1}(\R^2)$.  

\begin{lemma}
\label{lemma:resT.est2}
Let $B$ be a bounded subset of $H^{1,1}(\R^2) \times \C$. Then
$$ \sup_{(q,k) \in B} \norm[L^4 \to L^4]{(I-T)^{-1}} < \infty. $$
\end{lemma}

\begin{proof}
Write $T$ as $T(q,k)$ to show the dependence of the operator on $q \in L^2(\R^2)$ and $k \in \C$. 
We prove the required estimate in two steps. First, we show that  
the mapping 
\begin{equation}
\label{DSII:map.res}
L^2(\R^2) \times \C \ni (q,k) \mapsto (I-T(q,k))^{-1} \in \calB(L^4) 
\end{equation}
is continuous. Second, we show that  if $B$ is a bounded subset of $H^{1,1}(\R^2) \times \C$, then $B$ is a pre-compact subset in
$L^2(\R^2) \times \C$. Thus the resolvents $\{ (I-T)^{-1}: (q,k) \in B\}$, as the image of a pre-compact  set under a continuous map, form a bounded subset of $\calB(L^4)$. 

First we consider continuity of the map \eqref{DSII:map.res}.
By the second resolvent formula, it suffices to show that the map 
$ (q,k) \mapsto T(k,q)$ is continuous from $L^2(\R^2) \times \C$ to $\calB(L^4)$. But
\begin{align}
\label{restT.est2}
\norm[L^4 \to L^4]{T(k,q) - T(k',q')}
	&\leq			\norm[L^4 \to L^4]{T(k,q) - T(k',q)} \\
\nonumber
	&\quad		+ \norm[L^4 \to L^4]{T(k',q) - T(k',q')}\\
\nonumber
	&\lesssim	\norm[L^2]{(e_k - e_{k'})q} + \norm[L^2]{q-q'}
\end{align}
where in the second step we used \eqref{S.est} (where $q$ now includes the factor $e_k$) and the linearity of $S$ in $q$. The continuity is immediate.

Pre-compactness of $B$ as a subset of $L^2(\R^2) \times \C$  follows from the {Kolmogorov}{-}{Riesz} Theorem and is  left as Exercise \ref{ex:DSII.H11.compact}.
\end{proof}

We can also prove Lipschitz continuity of the resolvent as a function of $q$. 

\begin{lemma}
\label{lemma:resT.est3}
Fix $R>0$ and $q_1, q_2 \in H^{1,1}(\R^2)$ with $\norm[H^{1,1}]{q_i} \leq R$, $i=1,2$. Then
$$
\sup_{k \in \C} \norm[L^4 \to L^4]{(I-T(q_1,k))^{-1} - (I-T(q_2,k))^{-1}} \lesssim_{\, R} \norm[L^2]{q_1-q_2}.
$$
\end{lemma}

\begin{proof}
This is a consequence of Lemma \ref{lemma:resT.est2}, the estimate \eqref{restT.est2}, and the second resolvent formula.
\end{proof}

\subsubsection{Estimates on the scattering transform}

In order to analyze the scattering transform, we need estimates on the regularity in $k$ and large-$k$ behavior of the scattering solutions $m^1(z,k)$ and $m^2(z,k)$. In essence, this entails understanding the joint $(z,k)$ behavior of solutions to the model equation \eqref{DSII:m.model}. 

In order to do this, we need (i) estimates on the resolvent $(I-T)^{-1}$ uniform in $k$ and (ii) estimates on the joint $(z,k)$ behavior of the function $\dee_{\zbar}^{-1} (e_{-k} q)$. In \cite{Perry:2016} both of these steps were accomplished using the smoothness and decay of $q$ (i.e., using $q \in H^{1,1}(\R^2)$). In \cite{NRT:2017}, the authors need only assume that $q \in L^2$: they use ideas of concentration compactness \cite{Gerard:1998} to obtain the required control of the resolvent, and use the fractional integral estimates from Theorem \ref{thm:DBAR.est} to control $\dee_{\zbar}^{-1}(e_{-k} q)$.  

In these notes, we will take an intermediate route and borrow insights from \cite{NRT:2017} to provide a cleaner and more concise proof of the main results in \cite{Perry:2016}. In particular, by exploiting Theorem \ref{thm:DBAR.est}, we will avoid the multilinear estimates and resolvent expansions used in \cite{Perry:2016}. A number of calculations below also exploit the ideas behind \cite[Theorem 2.3]{NRT:2017}, a sharp $L^2$ boundedness theorem for non-smooth pseudodifferential operators. 

We begin with a mixed-$L^p$ estimate which actually holds for $q \in L^2$ (see \cite[Lemma 4.1]{NRT:2017}). The technique of proof is borrowed from \cite[Lemma 4.1]{NRT:2017}, with our weaker resolvent estimate from Lemma \ref{lemma:resT.est2} used instead of their stronger $L^2$ estimate \cite[Theorem 1.1]{NRT:2017}.

\begin{lemma}
\label{lemma:mxk.est}
Suppose that $q \in H^{1,1}(\R^2)$ and that $(\calM\widehat{q})(k)$ is finite. Let $m^1$ and $m^2$ be the unique solutions of \eqref{DSII:m}. Then
\begin{equation}
\label{mxk.est}
\norm[L^4]{m^1(\dotarg,k)-1} + \norm[L^4]{m^2(\dotarg,k)} \leq C\left(\norm[H^1]{q}\right) (\calM\widehat{q}\, )(k)^{1/2}
\end{equation}
Moreover, the maps $q \mapsto m^1$ and $q \mapsto m^2$ are locally Lipschitz continuous as maps from $H^{1,1}(\R^2)$ to $L^4(\R^2_z \times \R^2_k)$. 
\end{lemma}

\begin{proof}
By the definition \eqref{DBAR:mpm} of $m_\pm$ and the equation \eqref{DSII:mpm} obeyed by $m_\pm$, it suffices to prove the estimate 
$$ \norm[L^4]{w} \leq C\left( \norm[H^{1,1}]{q} \right)\left( \calM \widehat{q}\right) (k)^{1/2}  $$
for solutions $w$ of the model equation \eqref{DSII:m.model}.

From the solution formula \eqref{DSII:m.model.sol}, we estimate
\begin{align*}
\norm[L^4]{w} 	
	&\lesssim	
		\norm[L^4 \to L^4]{(I-T)^{-1}} 
		\norm[L^4]{\dee_{\zbar} \left( e_{-k} u \right)}\\
	&\lesssim
		C\left( \norm[H^{1,1}]{q} \right)
		\norm[L^2]{u}^{1/2} \left( \calM \widehat{u} \right)(k)^{1/2}
\end{align*}
where we used Lemma \ref{lemma:resT.est2} and the fractional integral estimate \eqref{DBAR.est.fourier}. This estimate now implies \eqref{mxk.est}.

An immediate consequence of \eqref{mxk.est} is the estimate
\begin{equation}
\label{mxk.est2}
\norm[L^4(R^2_z \times R^2_k)]{m^1 - 1} + \norm[L^4(\R^2_z \times \R^2_k)]{m^2}
\leq 
C\left( \norm[H^{1,1}]{q} \right) 
\norm[L^2]{q}^{1/2}.
\end{equation}
The Lipschitz continuity follows from Lemma \ref{lemma:resT.est3},   the solution formula \eqref{DSII:m.model.sol}, and \eqref{HLMax.p}.
\end{proof}

We can now prove:
\begin{proposition}
\label{prop:DSII.L2}
The scattering transform $\calS$ is bounded and Lipschitz continuous from $H^{1,1}(\R^2)$ to $L^2(\R^2)$. 
\end{proposition}

\begin{proof}
As already discussed, it suffices to prove the Lipschitz continuity estimates for $q \in \scrS(\R^2)$. We use the fact that $q \in \scrS(\R^2)$ in the computations leading to  \eqref{DSII:L2.2}.

By equation \eqref{S.split}, it suffices to show that the integral
\begin{equation}
\label{DSII:L2.1}
I(k) = -\frac{i}{\pi} \int e_{k}(z) \overline{q(z)} \left( m^1(z,k) - 1 \right) \, dz 
\end{equation}
defines an $L^2$ function of $k$, locally Lipschitz as a function of  $q$. From \eqref{DSII:m}, we may write
$ m^1(z,k) - 1 = \dee_{\zbar}^{-1} \left(q(\dotarg) m^2(\dotarg,k)\right)$ and change orders of integration to obtain
\begin{equation}
\label{DSII:L2.2}
I(k)  = - \frac{i}{\pi} \int  \left[\dee_{\zbar}^{-1}\left( e_k \overline{q} \right)(z)\right]  q(z) m^2(z,k) \, dz
\end{equation}
and conclude from the estimate \eqref{DBAR.est} and Lemma \ref{lemma:mxk.est} that
\begin{align*}
|I(k)| 
	&	\lesssim C\left(\norm[H^1]{q}\right) (\calM\widehat{\qbar}(k))^{1/2} \int (\calM\qbar(z))^{1/2}\,  |q(z)| \, |m^2(z,k)| \, dz\\
	&	\leq		 C\left(\norm[H^1]{q}\right) (\calM\widehat{\qbar}(k))^{1/2} \norm[L^2]{q}^{3/2} \norm[L^4]{m^2(\dotarg,k)}\\
\end{align*}
where in the second line we used \eqref{HLMax.p}.
Using \eqref{mxk.est2} and H\"{o}lder's inequality, we conclude that $I \in L^2$ with
$$ \norm[L^2]{I} \lesssim C\left(\norm[H^1]{q}\right) \norm[L^2]{q}^2.$$

To show Lipschitz continuity, we note that, by \eqref{DSII:L2.2},
\begin{align}
\label{DSII:I.lip}
I(k;q_1) & - I(k;q_2)  =\\
\nonumber
	&	\hphantom{+}  \frac{i}{\pi} \int 
				\left[q_1(z) \dee_{\zbar}^{-1}(e_k \overline{q_1})(z) - q_2(z) \dee_{\zbar}^{-1}(e_k \overline{q_2})(z)\right] 
				m^2(z,k;q_1)
			\, dz \\
\nonumber
	& + \frac{i}{\pi} \int 
				\left[ q_2(z) \dee_{\zbar}^{-1}(e_k \overline{q_2})(z)\right] 
				\left( m^2(z,k;q^1) - m^2(z,k;q_2)\right) 
			\, dz
\end{align}
The map $q \mapsto q \, \dee_{\zbar}^{-1}(e_k \qbar)$ is Lipschitz continuous from $L^2(\R^2)$ into $L^4(\R^2_k;  L^{4/3}(\R^2_z)))$ by multilinearity, \eqref{HLMax.p}, and \eqref{DBAR.est.fourier} (see Exercise \ref{ex:DSII.L2}). The map $q \mapsto m^2(z,k;q)$ is Lipschitz continuous from $H^{1,1}(\R^2)$ into $L^4(\R^2_z \times \R^2_k)$ by Lemma \ref{lemma:mxk.est}.
\end{proof}

\begin{remark}
\label{rem:DSII.nonlinear-ft}
The ``integration by parts'' that transforms \eqref{DSII:L2.1} to \eqref{DSII:L2.2} is also one of the key ideas behind the proof of the $L^2$ boundedness theorem for pseudodifferential operators with non-smooth symbols, Theorem 2.3, in \cite{NRT:2017}. Tracing through the argument used to estimate $I(k)$, it is easy to see that the same argument proves that
$$I(f,k) = -\frac{i}{\pi} \int e_{k}(z) f(z) \left( m^1(z,k)-1 \right) \, dz  $$
satisfies the estimate
$$ |I(k,f)| \lesssim C\left( \norm[H^1]{q} \right) \left( (\calM\widehat{f}(k) \right)^{1/2} \norm[L^2]{f} \norm[L^4]{m^2(\dotarg,k)}$$
so that 
\begin{equation}
\label{DSII:L2-bd}
 \norm[L^2]{I(\dotarg,f)} \lesssim C\left( \norm[H^1]{q} \right)  \norm[L^2]{q} \norm[L^2]{f}. 
\end{equation}
\end{remark}

To prove Theorem \ref{thm:DS.lip}, it remains to show that, for $q \in H^{1,1}(\R^2)$, $\calS q \in H^{1}(\R^2)$ and $\calS q \in L^{2,1}(\R^2)$, and that the corresponding maps are locally Lipschitz continuous. As a first step, we show that, if $q \in H^{1,1}(\R^2)$, then
$\calS q \in L^p(\R^2)$ for all $p \in [2,\infty)$.

\begin{proposition}
\label{prop:DSII.Lp}
For any $p \in [2,\infty)$, the scattering transform $\calS$ is locally Lipschitz continuous from $H^{1,1}(\R^2)$ to $L^p(\R^2)$. 
\end{proposition}

\begin{proof}
The Fourier tranform has this mapping property by the Hausdorff-Young inequality and the fact that $H^{1,1} \hookrightarrow L^q$ for $q \in (1,2]$ (see Exercise \ref{ex:DSII.L21}).  
Hence, owing to \eqref{S.split}, it suffices to prove that the map $q \mapsto I(k)$ defined by \eqref{DSII:L2.1} has the required continuity. 

Using \eqref{DSII:L2.2}, the fractional integral estimate \eqref{DBAR.est.fourier}, and the a priori estimate \eqref{mxk.est} on $m^2$,  we may estimate
\begin{align*}
|I(k)| 	&\lesssim	\int \left| \dee_{\zbar}^{-1}(e_k \qbar)(z) \right| \, |q(z)| \, \left| m^2(z,k) \right| \, dz\\
		&\lesssim	C\left(\norm[H^{1,1}]{q}\right) \left[(\calM\widehat{\qbar})(k)^{1/2} (Mq)(z)^{1/2}\right]^2 |q(z)|  \, dz\\
		&\lesssim	C\left(\norm[H^{1,1}]{q}\right) \norm[L^2]{q}^2 (\calM\widehat{\qbar})(k)
\end{align*}
which shows that $I \in L^p(\R^2)$ for any $p>2$ by the Hausdorff-Young inequality again. The proof of Lipschitz continuity uses \eqref{DSII:I.lip} and analogous estimates.
\end{proof}

Proposition \ref{prop:DSII.Lp} allows us to prove:

\begin{proposition}
\label{prop:DSII.H1}
The map $\calS$ is locally Lipschitz continuous from $H^{1,1}(\R^2)$ to $H^1(\R^2)$. 
\end{proposition}

\begin{proof}
It suffices to prove the Lipschitz estimate for $q \in \scrS(\R^2)$. In view of Proposition \ref{prop:DSII.L2}, property \eqref{Beurling-ddbar} of the Beurling transform,  and the boundedness of the Beurling transform on $L^p$, it suffices to show
that the map $q \mapsto \dee_{\zbar} I$ (where the differentiation is with respect to $k$) is locally Lipschitz continuous from $H^{1,1}(\R^2)$ into $L^2(\R^2)$. In Lemma \ref{lemma:DSII.m.k}, we will show that, for $q \in \scrS(\R^2)$, $m^1$ and $m^2$ also solve the $\dbar_k$-problem \eqref{DSII.m.k}. Thus,
for $q \in \scrS(\R^2)$ we may compute
\begin{align*}
%\label{ds}
\dee_{\kbar} \bfs(k) 	&= 	\frac{1}{\pi} \int e_k(z) \zbar \overline{q(z)} m^1(x,k) \, dz			
									 	-\frac{i}{\pi}  \bfs(k) \int  \overline{q(z)} \overline{m^2(z,k)} \, dz	\\
\nonumber
								&=	I_1 + I_2
\end{align*}
Tracing through the proof of  Proposition \ref{prop:DSII.L2} with $\qbar$ replaced by $\zbar \qbar$, we conclude that $I_1$ defines an $L^2$ function of $k$, Lipschitz continuous in $q$. It remains to estimate $I_2$.

%%%%%%%%%%%%%%%%%%%%%%%%%%%%%%%%%%%%
%
%		New proof - for old proof see auxiliary TeX file lec3-old-proof-of....
%
%%%%%%%%%%%%%%%%%%%%%%%%%%%%%%%%%%%%
By Proposition \ref{prop:DSII.Lp}, $\bfs \in L^4(\R^2)$, so it suffices to show that the integral defines a Lipschitz map from $q \in H^{1,1}(\R^2)$ to $L^4(\R^2_k)$. Since $m^2 \in L^4(\R^2_z \times \R^2_k)$ and $q \in H^{1,1}(\R^2) \subset L^{4/3}(\R^2_z)$, this is a consequence of H\"{o}lder's inequality and Lemma \ref{lemma:mxk.est}.
\end{proof}

To complete the proof of Theorem \ref{thm:DS.lip}, we show that $\calS q \in L^{2,1}(\R^2)$ with appropriate Lipschitz continuity.

\begin{proposition}
\label{prop:DSII.L21}
The map $\calS$ is locally Lipschitz continuous from $H^{1,1}(\R^2)$ to $L^{2,1}(\R^2)$. 
\end{proposition}

\begin{proof}
We need only show that $q \mapsto I(k)$ has the above property, where $I(k)$ is defined by \eqref{DSII:L2.1}.
We begin with a computation for $q \in \scrS(\R^2)$, using the trivial identity $\dee_{\zbar} e_k = i\kbar e_k$ and integration by parts:
\begin{align*}
\kbar I(k)	&=	-\frac{1}{\pi} 
								\int e_k(z) \dee_{\zbar} 
										\left(
											\overline{q(z)} \left(m^1(z,k) - 1 \right)
										\right) \, dz\\
					&=	I_1 + I_2
\intertext{where}
I_1				&=	-\frac{1}{\pi} \int e_k(z) \left( \dee_{\zbar} \overline{q} \right)(z) \, m^1(z,k) \, dz,	\\
I_2				&=	-\frac{1}{\pi} \int e_k(z) |q(z)|^2 m^2(z,k) \, dz,
\end{align*}
and we used the first equation in \eqref{DSII:m} to simplify $I_2$.  

The integral $I_1$ defines an $L^2$ function of $k$ since $\dee_{\zbar} \overline{q} \in L^2(\R^2)$ by an argument similar to the proof of Proposition \ref{prop:DSII.L2}.  To analyze $I_2$, we use the second equation in \eqref{DSII:m} to write
\begin{align*}
I_2		&=		\int  |q(z)|^2 \dee_z^{-1} 
								\left( e_k (\dotarg) q(\dotarg)  m^1(\dotarg,k) \right)(z) 
						\, dz		\\
			&=		I_{21}+ I_{22} 
\end{align*}
where
\begin{align*}
I_{21}	&=	\int |q(z)|^2  
							\dee_z^{-1} \left( e_k q \right) (z) 
					\, dz,\\
I_{22}	&=	\int |q(z)|^2
							 \dee_z^{-1} 
							 		\left( 
							 				e_k q(\dotarg) \left( m^1(\dotarg,k) - 1 \right) 
							 		\right) (z)
							 \, dz.
\end{align*}
By ``integration by parts'' we have
\begin{align*}
 I_{21} 	&= 	-\int e_k(z) q(z) \dee_z^{-1} \left(|q|^2\right)(z) \, dz
\intertext{which exhibits $I_{21}$ as the Fourier transform of an $L^2$ function since $|q|^2$ in $L^{4/3}(\R^2)$. On the other hand}
I_{22}	&=	-\int e_k(z) q(z) 
						\left[
							\dee_z^{-1}
								\left( |q(\dotarg)|^2 \right) 
						\right](z)
						\left(m^1(z,k) - 1 \right) 
						\, dz
\end{align*}
which exhibits $I_{22}$ in the form $I(k,f)$ (see Remark \ref{rem:DSII.nonlinear-ft}) where $f$ is  the $L^2$ function $q \dee^{-1} |q|^2$. The needed $L^2$ bound is a direct consequence of \eqref{DSII:L2-bd}.

As usual, the proof of Lipschitz continuity rests on the multilinearity of explicit expressions involving $q$ and the Lipschitz continuity of $m^1$ and $m^2$ viewed as functions of $q$. To prove that $I_1$ is locally Lipschitz continuous, one mimics the proof that $I$ is Lipschitz beginning with \eqref{DSII:I.lip} in the proof of Proposition \ref{prop:DSII.L2}. To show that $I_2$ is Lipschitz continuous, one notes that $I_{21}$ is an explicit multilinear function of $q$ , while $I_{22}$ can be controlled by the same method used to prove Lipschitz continuity of $I$ on the proof of Proposition \ref{prop:DSII.L2}. 
\end{proof}

It now follows that $\calS$, initially defined on $\scrS(\R^2)$, extends to a locally Lipschitz continuous map from $H^{1,1}(\R^2)$ to itself. It remains to prove that $\calS^{-1} = \calS$. 

By the Lipschitz continuity of $\calS$ on $H^{1,1}(\R^2)$, it suffices to prove that $\calS = \calS^{-1}$ on the dense subset $\scrS(\R^2)$. The idea of the proof is to use uniqueness of solutions to the system \eqref{DSII:m} together with the fact that, for $q \in \scrS(\R^2)$, the functions $(m^1, m^2)$ satisfy \emph{both} the system \eqref{DSII:m} and the following system of $\dbar_k$-equations. 

\begin{lemma}
\label{lemma:DSII.m.k}
Suppose that $q \in \scrS(\R^2)$ and let $(m^1,m^2)$ be the unique solutions to \eqref{DSII:m}. Then, for each $z \in \C$, 
\begin{equation}
\label{DSII.m.k}
\left\{
\begin{aligned}
\dee_{\kbar} m^1(z,k) 		&=	e_{-k} \bfs(k) \overline{m^2(z,k)}\\
\dee_{\kbar} m^2(z,k)		&=	e_{-k} \bfs(k) \overline{m^1(z,k)}\\
m^1(z,k) - 1, \, m^2(z,k)	&=	\bigO{|k|^{-1}}
\end{aligned}
\right.
%}
\end{equation}
where $\bfs(k)$ is given by \eqref{DSII:scatt}
\end{lemma}

\begin{proof}
For $q \in \scrS(\R^2)$, the solutions $(m^1,m^2)$ of \eqref{DSII:m} have the large-$z$ asymptotic (differentiable) expansions 
\begin{align*}
m^1(z,k) 	&=	1	+	\bigO{|z|^{-1}}\\
m^2(z,k)	&=	 e_{-k}(z) \frac{i\bfs(k)}{\zbar} + \bigO{|z|^{-2}}
\end{align*}
where $\bfs$ is given by \eqref{DSII:scatt} (see Exercise \ref{ex:DSII.two-s} and the comments after \eqref{DSII:scatt}). If $v^1 = \dee_{\kbar} m^1$ and $v^2 = \dee_{\kbar} m^2$ then, differentiating 
\eqref{DSII:m} with respect to $\kbar$ we recover
\begin{align*}
\dee_{\zbar} v^1	&=	q	v^2\\
\left(\dee_z + ik \right) v^2	&=	\qbar v^1
\end{align*}
It follows from the asymptotic expansions for $m^1$ and $m^2$ above that  $v^1 = \bigO{|z|^{-1}}$ but $v^2 = e_{-k} \bfs(k) + \bigO{|z|^{-1}}$. 
Hence, in order to use the uniqueness theorem for solutions of \eqref{DSII:m}, we need to make a subtraction to remove the constant term in $v^2$. 
Setting 
$$
w^1(z,k) 	= \dee_{\kbar} m^1 - e_{-k} \bfs \overline{m^2}, \quad
w^2 			= \dee_{\kbar} m^2 - e_{-k} \bfs \overline{m^1},
$$ 
and using \eqref{DSII:m}, we conclude that
\begin{align*}
\dee_{\zbar} w^1	&=	q	w^2\\
\left(\dee_z + ik \right) w^2	&=	\qbar w^1
\end{align*}
where $w_1$ and $w_2$ are $\bigO{|z|^{-1}}$ as $|z| \to \infty$. 
Hence by Proposition \ref{prop:DBAR.mpm}, $w^1 = w^2 = 0$. Since $w^1$ and $w^2$ are smooth functions of $z$ and $k$, it follows that the first two of equations
\eqref{DSII.m.k} hold for each $z$. 

It remains to show that, for each fixed $z$, $m_1(z,k) -1$ and $m^2(z,k)$ are $\bigO{|k|^{-1}}$ as $|k| \to \infty$.   
For $q \in \scrS(\R^2)$, the functions  $m^1$ and $m^2$ are smooth  functions of $z$ and $k$ with
bounded derivatives (see Sung \cite[section 2]{Sung:1994a}). From the integral formulas
\begin{align}
\label{DSII:m1.int}
m^1(z,k)				&=	1	+ 	
										\frac{1}{\pi} 
											\int 
												\frac{1}{z-\zeta} q(\zeta)  m^2(\zeta,k) 
											\, d\zeta\\
\label{DSII:m2.int}
e_{k}(z) m^2(z,k)	&=	\frac{1}{\pi} 
										\int 
											\frac{1}{\zbar - \zetabar} 
											e_k(\zeta) \overline{q(\zeta)} m^1(\zeta,k) 
										\, d\zeta
\end{align}
we first note that it is enough to prove that $m^2(z,k) = \bigO{|k|^{-1}}$ uniformly in $z$ since it will then follow from \eqref{DSII:m1.int} that $m^1(z,k)-1 = \bigO{|k|^{-1}}$. We can integrate by parts in \eqref{DSII:m2.int} to see that
\begin{align*}
e_k(z) m^2(z,k)	&=	\frac{1}{ik} e_k(z) \overline{q(z)} m^1(z,k) - 
									\frac{1}{\pi i k} 
											\int \frac{1}{\zbar - \zetabar} 
											e_k(\zeta) 
											\dee_\zeta 
												\left( 
													\overline{q(\zeta)} m^1(\zeta,k) 
												\right) 
											\, d\zeta
\end{align*}
which shows that $m^2(z,k) = \bigO{|k|^{-1}}$. 
\end{proof}

Given $\bfs(k) \in \scrS(\R^2)$, the inverse scattering transform $\calS \bfs $ is computed by solving the system
\begin{equation}
\label{DSII:minv.k}
\left\{
\begin{aligned}
\dee_{\kbar} n^1(k,z)	&=	\bfs(k) n^2(k,z)	\\
\left(\dee_k + iz \right) n^2(k,z)		&=	\overline{\bfs(k)} n^1(k,z) 	\\
n^1(k,z) -1 , n^2(k,z)		&=	\bigO{|k|^{-1}}
\end{aligned}
\right.
%}
\end{equation}
and extracting $\calS \bfs$ from the asymptotic expansion
$$
n^2(k,z) = e_{-z}(k) \frac{i}{\kbar} \left(\calS \bfs\right)(z) + \bigO{|k|^{-2}}.
$$
The system \eqref{DSII:minv.k} is uniquely solvable by Proposition \ref{prop:DBAR.mpm}. On the other hand, if $m^1$ and $m^2$ solve \eqref{DSII:m} for given $q \in \scrS(\R^2)$ and $\bfs = \calS q$, these functions also solve \eqref{DSII:m.k}. A short computation shows that
$ n^1(k,z) = m^1(z,k), \, n^2(k,z) = e_{-z}(k) \overline{m^2(z,k)}$ solve the system \eqref{DSII:minv.k}. Since this solution is unique, we may compute $\calS \bfs$ using the large-$k$ expansion of $m^2(z,k)$ (see Exercise \ref{ex:DSII.m.asy.k}):
\begin{align*}
n^2(k,z)	&=	e_{-z}(k) \overline{m^2(z,k)}	\\
			&=	e_{-z}(k) \left(  \frac{iq(z)}{\kbar} + \bigO{|k|^{-2}} \right)\\
			&=	e_{-z}(k)  \frac{iq(z)}{\kbar} +  \bigO{|k|^{-2}}
\end{align*}
to conclude that $\calS \bfs = q$.  We have proved:

\begin{proposition}
\label{prop:DSII.Sinv}
Suppose that $q \in \scrS(\R^2)$. Then $\calS (\calS (q)) = q$. 
\end{proposition}

\
\subsection{Solving the DSII Equation}
\label{DSII:subsec.time}

In this subsection we use the scattering transform to solve \eqref{DSII:bis2} with initial data $q_0 \in H^{1,1}(\R^2)$.  The putative solution $q_{\inv}$ is given by \eqref{DSII:qinv}; note that $q_{\inv}(z,0;q_0)=q_0$ by Proposition \ref{prop:DSII.Sinv}. 
We will prove:

\begin{theorem}
\label{thm:DSII.qinv.sol}
The function \eqref{DSII:qinv} is the unique global solution of \eqref{DSII:bis} for any $q_0 \in H^{1,1}(\R^2)$.
\end{theorem}

We begin with an important reduction.

\begin{proposition}
\label{prop:DSII.dense}
Suppose that, for each $q_0 \in \scrS(\R^2)$, $q_{\inv}(z,t;q_0)$ solves the integral equation \eqref{DSII:int}. Then $q_{\inv}(z,t;q_0)$ solves \eqref{DSII:int} for any $q_0 \in H^{1,1}(\R^2)$.
\end{proposition}

\begin{proof}
Observe that the map $q_0 \mapsto q_{\inv}(\dotarg,\dotarg;q_0)$ is a continuous map from 
$H^{1,1}(\R^2)$ to $C((0,T); H^{1,1}(\R^2))$ for any $T>0$, and recall from Exercise \ref{ex:DSII.H2}
that $H^{1,1}(\R^2)$ is continuously embedded in $L^4(\R^2)$. It follows that $r \mapsto q_{\inv}(\dotarg,t;r)$ is a continuous map from $H^{1,1}(\R^2)$ into the space $X$ (see \eqref{DSII:X}) for any $T>0$. 

Let $q_0 \in H^{1,1}(\R^2)$ and let $\{ q_{0,n} \}$ be a sequence from $\scrS(\R^2)$ with $q_{0,n} \to q_0 \in H^{1,1}(\R^2)$.  Then $q_{\inv}(\dotarg,\dotarg;q_n) \to q_{\inv}(\dotarg,\dotarg;q)$ in $X$ as $n \to \infty$. The result now follows from the fact that \eqref{DSII:int}
takes the form $q = \Phi(q)$ where $\Phi$ is continuous on $X$.
\end{proof}

Given this reduction, it suffices to prove that $q_{\inv}(z,t;q_0)$ solves \eqref{DSII:int} for any 
$q_0 \in \scrS(\R^2)$. Recall that, by Sung's work \cite{Sung:1994a,Sung:1994b,Sung:1994c}, the map $\calS$ restricts to a continuous map from $\scrS(\R^2)$ to itself,  $\bfs=\calS q_0$ is also a Schwartz class function,  and the function $t \mapsto q_{\inv}(z,t;q_0)$ is continuously differentiable as a map from $\R$ to $\scrS(\R^2)$.
It then suffices to show that $q_{\inv}(z,t;q_0)$ is a classical solution to \eqref{DSII:bis2}. In the remainder of this section, we will use the complete integrability of \eqref{DSII:bis2} to prove this fact by showing that the solution $(m^1(z,t,q_0), m^2(z,t,q_0)$ of the $\dbar_k$-problem \eqref{DSII:m.inv.t} generates a joint classical solution of 
the equations \eqref{DSII:m.x1}--\eqref{DSII:m.t2}. We will then show that, as a consequence, $q$ is a classical solution of \eqref{DSII:bis2}.

Consider the $\dbar$-problem
\begin{equation}
\label{DSII:m.kt}
\left\{
\begin{aligned}
\left( \dee_{\kbar}	m^1	\right)(z,k,t)	
		&=	e_{-k}(z) \bfs(k,t)	\overline{m^2(z,k,t)},\\
\left( \dee_{\kbar}	m^2	\right)(z,k,t)	
		&=	e_{-k}(z) \bfs(k,t)	\overline{m^1(z,k,t)	},	\\[5pt]
m^1(z,\dotarg,t)-1, \, m^2(z,\dotarg,t)	&= \bigO{|k|^{-1}}
\end{aligned}
\right.
%}
\end{equation}
where
\begin{equation}
\label{s.ev}
\bfs(k,t)	=	e^{2i(k^2+\kbar^2)t} \bfs(k).
\end{equation}
Note that the $L^4$ condition on $(m^1, m^2)$ is replaced by an asymptotic condition since, for $\bfs \in \scrS(\R^2)$, the solutions are bounded smooth functions and have complete asymptotic expansions in $k$ (see Exercise \ref{ex:DBAR.green}).

We will show that $(m^1,m^2)$ is a joint solution of the equations \eqref{DSII:m.x1}--\eqref{DSII:m.t2} where $q(z,t) = q_{\inv}(z,t;q_0)$
and that, moreover, 
$$ \lim_{|k| \rarr \infty} m^1(z,k,t) = 1 $$
for all $(z,t)$ and  $m^2(z,k,t) \neq 0$ for all $(z,t)$ and some $k \in \C$. 
These facts, together with the identity \eqref{DSII:crossdiff} from Exercise \ref{ex:DSII.crossdiff} can then be used to  show that $q(z,t)$ so defined solves the DS II equation.

In analogy to the Riemann-Hilbert problem for defocussing NLS, we will base our proof that the solutions of \eqref{DSII:m.kt} furnish solutions of the Lax equations \eqref{DSII:m.x1}--\eqref{DSII:m.t2} on a vanishing lemma, this time for the $\dbar$-system. We state it in greater generality than is needed here.

\begin{lemma}
\label{lemma:DBAR.vanish}
Suppose that $w_1$, $w_2$ are solutions of the system
\begin{equation*}
\left\{
\begin{aligned}
\left(\dee_{\kbar} w_1\right)(z,k)	&=	e_{-k} \bfs(k) \overline{w_2(z,k)},\\
\left(\dee_{\kbar} w_2\right)(z,k)	&=	e_{-k} \bfs(k) \overline{w_1(z,k)}
\end{aligned}
\right.
%}
\end{equation*}
for $\bfs \in L^2$ and $w_1$, $w_2 \in L^4_k(\R^2)$. Then $w_1=w_2=0$.
\end{lemma}

This lemma is an easy consequence of Theorem \ref{thm:Liouville} if one considers the functions $w_\pm = w_1 \pm w_2$. 

First, we'll show that a solution of \eqref{DSII:m.kt} also solves \eqref{DSII:m.x1}--\eqref{DSII:m.x2} with 
$$m^1(\dotarg,k)-1, \, m^2(\dotarg,k) \in L^4_z(\R^2)$$ for each $k$. For notational convenience we suppress dependence on $t$. 

\begin{proposition}
\label{prop:DSII.Lax1}
Suppose that $m^1(z,k), m^2(z,k)$ solve \eqref{DSII:m.k} for each $z \in \C$. Then $m^1(\dotarg,k) - 1, \, m^2(\dotarg,k)\in L^4_z(\R^2)$ for each $k \in \C$, and $(m_1,m_2)$ solve \eqref{DSII:m.x1}--\eqref{DSII:m.x2} for each $z$, where $q(z)$ is \emph{defined} by
\begin{equation}
\label{q.recon.bis}
q(z) = -\frac{i}{\pi} \int e_k(z) \overline{\bfs(k)} m^1(z,k) \, dk. 
\end{equation}
\end{proposition}

\begin{proof}
Differentiating \eqref{DSII:m.k} we compute (see Exercise \ref{ex:DSII.dbar.zdiff})
\begin{equation}
\label{DSII:Lax1.pre1}
\left\{ 
\begin{aligned}
\dee_{\kbar} \left( \dee_{\zbar} m^1 \right)	
		&=	e_{-k} \bfs \, \overline{(\dee_z+ik)m^2}\\
\dee_{\kbar} \left( \left(\dee_z + ik \right) m^2\right)
		&=	e_{-k} \bfs \, \overline{ \dee_{\zbar} m^1 }
\end{aligned}
\right.
%}
\end{equation}
(the pointwise differentiation makes sense because, for $\bfs \in \scrS(\R^2)$, the functions $m^1$ and $m^2$ are smooth functions of both variables).
From the large-$k$ asymptotics of $m^1$, 
we see that $\dee_z m^1 = \bigO{|k|^{-1}}$, so that $\dee_z m^1 \in L^4(\R^2)$. On the other hand, 
$$ \left(\dee_z + ik \right)m^2 - c(z)  \in L^4_k(\R^2)$$
where
$$ 
c(z) = \lim_{|k| \rarr \infty} ik m^2(z,k) = \frac{i}{\pi} \int e_{-k}(z) \bfs(k) m^1(z,k) \, dk  
$$
Making a subtraction in \eqref{DSII:Lax1.pre1} we have
\begin{equation}
\label{DSII:Lax1.pre2}
\left\{
\begin{aligned}
\dee_{\kbar} w_1
	&=	e_{-k} \bfs(k) \overline{w_2 }\\
\dee_{\kbar}  w_2 
	&=	e_{-k} \bfs(k) \overline{ w_1 }
\end{aligned}
\right.
%}
\end{equation}
where
$$ w_1 = \dee_{\zbar} m^1 - \overline{c(z)} m^2, \quad w_2 = \left(\dee_z + ik\right)m^2 - c(z) m^1 $$
(see Exercise \ref{ex:DSII.dbar.zdiff}).
We can now apply Lemma \ref{lemma:DBAR.vanish} to conclude that  $m^1$ 
and $m^2$ satisfy \eqref{DSII:m.x1}--\eqref{DSII:m.x2} with $q$ as defined in 
\eqref{q.recon.bis}.
\end{proof}

\begin{remark}
\label{rem:mk.x.asy}
Since $q \in \scrS(\R^2)$, it follows that $m^1$ and $m^2$ have complete large-$z$ asymptotic expansions for each fixed $k$. 
\end{remark}

Next, we show that $m^1$ and $m^2$ satisfy \eqref{DSII:m.t1}--\eqref{DSII:m.t2} by a similar technique, now tracking the dependence of $m^1$ and $m^2$ on time.

\begin{proposition}
\label{prop:DSII.Lax2}
Suppose that $m^1(z,k,t)$ and $m^2(z,k,t)$ solve \eqref{DSII:m.kt}. Then $m^1$ and $m^2$ solve \eqref{DSII:m.t1}--\eqref{DSII:m.t2} where $q$ is defined by \eqref{q.recon.bis} and $g$ is given by 
$g = -4\dee_{\zbar}^{-1}\left( \dee_z |q|^2 \right)$. 
\end{proposition}

\begin{proof}
At top order the Lax equations \eqref{DSII:m.t1}--\eqref{DSII:m.t2} (taking $\eps=+1$ here and in what follows) imply that
$$
v_1 \coloneqq \left( \dee_t -  2i \dee_z^2 + 4k \dee_z \right) m_1 \sim 0, \quad
v_2 \coloneqq \left( \dee_t + 2i \dee_{\zbar}^2 -  2ik^2 \right) m^2 \sim 0
$$
where the corrections vanish as $|z| \to \infty$. 
Motivated by this observation, we differentiate \eqref{DSII:m.kt} and compute
\begin{align}
\label{DSII:dbart1}
\dee_{\kbar} v_1
	&=	e^{it\varphi} \bfs \overline{v_2}\\
\label{DSII:dbart2}
\dee_{\kbar} v_2 
	&=	e^{it\varphi} \bfs \overline{v_1}
\end{align}
(see Exercise \ref{ex:DSII.dbar.timediff})
where $\varphi$ is given by \eqref{DSII:phase}.
If $v_1$ and $v_2$ were decreasing at infinity as functions of $k$, Lemma \ref{lemma:DBAR.vanish} would allow us to conclude that $v_1=v_2=0$. This is not the case since $k^2 m^2$ is of  order $k$ as $k \to \infty$ and $4k\dee_z m_1$ is of order $1$ as $k \to \infty$. For this reason we must make a subtraction using the asymptotic expansions of $m^1$ and $m^2$ which will lead to the remaining, lower-order terms in \eqref{DSII:m.t1}--\eqref{DSII:m.t2}. From Exercise \ref{ex:DSII.m.asy.k},
we have 
\begin{align}
\label{m1.asy.k}
m^1(z,k)	
	&=	1	-	\frac{i  \dee_{\zbar}^{-1}\left(|q|^2\right)}{k} + \bigO{k^{-1}}\\
\label{m2.asy.k}
m^2(z,k)
	&=	\frac{-i  \qbar}{k} + \frac{-\qbar \dee_{\zbar}^{-1}\left(|q|^2\right) +%%%%i   error!
		\dee_z \qbar}{k^2} + \bigO{k^{-3}}
\end{align}
so that
\begin{align}
\label{v1.asy.k}
v_1 	&=	-4i   \dee_z  \dee_{\zbar}^{-1} \left( |q|^2 \right) + \bigO{k^{-1}},\\
\label{v2.asy.k}
v_2	&=	-2  \qbar k 
					-	2i \left(
								-q \dee_{\zbar}^{-1}\left(|q|^2\right) +  \dee_z \qbar
							\right) 
					+ \bigO{k^{-1}}.
\end{align}
Thus if
\begin{align}
\label{w1.def}
w_1	&=	v_1 - 2i(\dee_{\zbar} q )m^2 + 2iq \dee_{\zbar} m^2 - ig m^1,\\
\label{w2.def}
w_2	&=	v_2 +2i  (\dee_z \qbar) m^1 - 2i \qbar(\dee_z + ik)m^1+ i\gbar m_2,
\end{align}
it follows from the asymptotic expansions \eqref{m1.asy.k}--\eqref{m2.asy.k} and \eqref{v1.asy.k}--\eqref{v2.asy.k}
that $w_1 = \bigO{k^{-1}}$ and $w_2 = \bigO{k^{-1}}$ (see Exercise \ref{ex:DSII.w1w2}),
while a straightforward computation (Exercise \ref{ex:DSII.w1w2.k}) shows that 
\begin{equation}
\label{DSII:Lax2.pre1}
\left\{
\begin{aligned}
\dee_{\kbar} w_1 	&=	e^{it\varphi} \bfs \overline{w_2}\\
\dee_{\kbar} w_2	&=	e^{it\varphi} \bfs \overline{w_1}
\end{aligned}
\right.
%}
\end{equation}
We can now use Lemma \ref{lemma:DBAR.vanish} to conclude that $w_1=w_2=0$ and \eqref{DSII:m.t1}--\eqref{DSII:m.t2}  hold.
\end{proof}

\begin{proof}[Proof of Theorem \ref{thm:DSII.qinv.sol}]
It follows from Proposition \ref{prop:DSII.Sinv} that $	q_{\inv}(z,0;q_0)  = q_0(z)$, so it suffices to show that $q_{\inv}$ is a classical solution of \eqref{DSII:bis}. By Proposition \ref{prop:DSII.dense} it suffices to prove that this is the case for $q_0 \in \scrS(\R^2)$.  By Propositions \ref{prop:DSII.Lax1}--\ref{prop:DSII.Lax2}, the functions $m^1$ and $m^2$ solve \eqref{DSII:m.x1}--\eqref{DSII:m.t2}.  If we now set $\psi_1 = e^{i(kz-k^2t)} m^1$ and $\psi_2 = e^{i(kz-k^2t)}m^2$,   it will follow from the computations in Exercise \ref{ex:DSII.crossdiff} that
\eqref{DSII:crossdiff} holds with $q=q_{\inv}$ provided we can show that $\psi_1$ and $\psi_2$ satisfy conditions (i)--(iii) given there. 

Conditions (i) and (ii) are proved in Proposition \ref{prop:DSII.Lax1}. Condition (iii) is equivalent to the statement that, for each $(t,z)$, $m^2(z,k,t) \neq 0$ for at least one $k$. If not, it follows from Lemma \ref{lemma:DSII.m.k} that $\dee_{\kbar} m^1(z,k) \equiv 0$ so $m^1(z,k,t) \equiv 1$ and $m^2(z,k,t) \equiv 0$ for this fixed $(z,t)$ and all $k$. It then follows from the $\dee_{\kbar}$ equation for $m^2$ that $\bfs(k,t) \equiv 0$ for this fixed $t$. But then, since $\bfs$ evolves linearly, we get $\bfs(k,t) \equiv 0$, hence $q(z,t) \equiv 0$ by the invertibility of $\calS$. Hence, $m^2(z,k,t) \neq 0 $ for some $k$, and conditions (i)--(iii) hold.
\end{proof}

\subsection{\texorpdfstring{$L^2$ Global Well-Posedness and Scattering: The Work of Nachman, Regev, and Tataru}{L2 Global Well-Posedness and Scattering}}
\label{DSII:subsec.NRT}

In this section we discuss briefly the results of Nachman, Regev, and Tataru \cite{NRT:2017}.  Their first result is a remarkable strengthening of Theorem \ref{thm:DS.lip}. 

\begin{theorem}\cite[Theorem 1.2]{NRT:2017}
\label{thm:DS.lip.L2}
The scattering transform $\calS$ is a diffeomorphism from  $L^2(\R^2)$ onto itself with $\calS^{-1} = \calS$. 
Moreover, $\norm[L^2]{\calS q} = \norm[L^2]{q}$, and the pointwise bound
\begin{equation}
\label{DSII:S.max}
\left| \left(\calS q\right)(k) \right| \leq C\left(\norm[L^2]{q} \right) \calM\widehat{q}(k)
\end{equation}
holds.
\end{theorem}

Theorem \ref{thm:DS.lip.L2} rests on the following resolvent estimate which is proven using concentration compactness methods. Denote by $\dot{H}^{1/2}(\R^2)$ the homogeneous Sobolev space of order $1/2$, which embeds continuously into $L^4(\R^2)$, and denote by $\dot{H}^{-1/2}(\R^2)$ its topological dual. The authors consider the model equation
$$ L_q u = f, \quad L_q u = \dee_{\zbar} u + q \ubar$$ (compare \eqref{DSII:m.model}) for $u \in \dot{H}^{1/2}(\R^2)$, where $f \in \dot{H}^{-1/2}(\R^2)$.  

\begin{theorem}\cite[Theorem 1.1]{NRT:2017}
\label{thm:DS.l2.res}
The estimate
$$ \norm[\dot{H}^{-1/2}(\R^2) \to H^{1/2}(\R^2)]{L_q^{-1}} \leq C\left( \norm[L^2]{q} \right) $$
where $t \to C(t)$ is an increasing, locally bounded function on $[0,\infty)$. 
\end{theorem}

As an immediate consequence of Theorem \ref{thm:DS.lip.L2}, we have:

\begin{theorem}\cite[Theorem 1.4]{NRT:2017}
For any Cauchy data $a_0 \in L^2(\R^2)$, the defocussing DS II equation has a unique global solution in $C(\R,L^2(\R^2) \cap L^4(\R^2 \times \R)$. 
\end{theorem}

The authors' Theorem 1.4 also includes stability estimates, pointwise bounds, and a global bound on the $L^4$ norm of the solution in space and time.

The pointwise bound \eqref{DSII:S.max} plays a crucial role in the authors' analysis of scattering for the DS II equation. Applied to the solution $q_{\inv}(z,t)$
given by \eqref{DSII:IST-sol} it implies that
$$ |q_{\inv}(z,t)| \lesssim C\left(\norm[L^2]{\calS q_0}\right) \calM q_{\lin} (z,t) $$
where
$$ q_\lin(z,t) = \calF_a \left( e^{it\left( (\dotarg)^2 + (\overline{\dotarg})^2 \right)} (\calS q_0) \right)(z).
$$
which is exactly the solution to the linear problem \eqref{DSII:lin.bis} with initial data 
$$v_0 = \left(\calF_a \circ  \calS \right)(q_0).$$ 
Since the maximal function is bounded between $L^p$-spaces for $p \in (1,\infty)$, this implies immediately that 
$L^p$ estimates in space or mixed $L^p-L^q$ estimates in space and time which hold for the linear problem, automatically hold for the nonlinear problem for $q_0$ on bounded subsets of $L^2(\R^2)$. In particular, it follows that $q \in L^4(\R^2 \times \R)$.  

Using these estimates, Nachman, Regev and Tataru show that all solutions scatter in $L^2(\R^2)$ and that, indeed the scattering is trivial in the sense that past and future asymptotics are equal. Denote by $U(t)$ the nonlinear evolution
$$ U(t)f (z) = \calS^{-1} \left( e^{it\left( (\dotarg)^2 + (\overline{\dotarg})^2 \right)} \left( \calS f \right) \right)(z).$$
and by $V(t)$ the linear evolution
$$ V(t) f(z) =  \calF_a^{-1}  \left( e^{it\left( (\dotarg)^2 + (\overline{\dotarg})^2 \right)} \left( \calF_a f \right) \right)(z). $$
In scattering theory we seek initial data $v_\pm$ for the linear equation so that
$$ \lim_{t \to \pm \infty} \norm[L^2]{U(t) q_0  - V(t) v_\pm} = 0 $$
as vectors in $L^2(\R^2)$.
Formally we have
$$ v_\pm(t) = \lim_{t \to \pm \infty} V(-t) U(t) q_0 $$
if the limit exists. The limiting maps, if they exist, are the nonlinear wave operators $W^\pm$.  

To show the convergence, it suffices to show that
$$ \frac{d}{dt} V(-t) U(t) q_0 = V(-t) N(q(t)) $$
is integrable as an $L^2$-valued function of $t$, where $N(q) = q (g+\gbar)$ is the nonlinear term in the DS II equation. Since $q(z,t) \in L^4(\R^2 \times \R^2)$, it follows that $N(q(t)) \in L^{4/3}(\R^2 \times \R^2)$, since the Beurling operator $\dbar^{-1} \dee$ is bounded from $L^p$ to itself for any $p \in (1,\infty)$. From the Strichartz estimate
\eqref{DSII:S3},
it follows that $V(-t) N(q(t))$ is integrable as an $L^2$-valued function of $t$, so that
the asymptotes $v_\pm$ exist. Hence:

\begin{theorem}\cite[Lemma 5.5]{NRT:2017}
The nonlinear wave operators $W^\pm$ exist and are Lipschitz continuous maps on $L^2(\R^2)$. 
\end{theorem} 

It can be shown that $v_+ = v_- $ so that the scattering is trivial. 

\subsection*{Exercises for Lecture 3}

\begin{exercise}
\label{ex:DBAR.green}
Using integration by parts, show that  for any $f \in C_0^\infty(\R^2)$, 
$$ \frac{1}{\pi} \int \frac{1}{z-w} \left( \dee_{\wbar} f \right) (w) \, dw = f(z). $$
\emph{Hint}: Develop a Green's formula for the $\dbar$ operator analogous to the corresponding formula for the Laplacian, and use the fact that 
$$ 
\int \frac{1}{z-w} f(w) \, dw = \lim_{\eps \darr 0} \int_{\R^2 \setminus B(z,\eps)}  \frac{1}{z-w} f(w) \, dw . 
$$
\end{exercise}

\begin{exercise}
\label{ex:DBAR.asy}
Suppose that $f$ is a measurable function with $\int |x|^N |f(x)| \, dx < \infty$ for all nonnegative integers $N$. Show that $u(z) = \left(\dee_{\zbar}^{-1} f\right)(z)$ has a large-$z$ asymptotic expansion of the form 
$$ u(z) \sim \sum_{j \geq 1} \frac{a_j}{z^j} $$
and give an explicit remainder estimate for $ u - \sum_{j=1}^N a_j z^{-j}$. 

\emph{Remark}. The equation $\dee_{\zbar} u = f \in \scrS(\R^2)$ implies that $u$ is `almost' analytic near infinity; the expansion above shows that $f$ `almost' has a Taylor series near the point at infinity.
\end{exercise}

\begin{exercise}
\label{ex:DBAR.Vekua}
Prove \eqref{DBAR:Vekua} by writing
$$ \int \frac{1}{z-w} f(z)\, dz = \left( \int_{|z-w| \leq 1} + \int_{|z-w| > 1} \right) \frac{1}{z-w} f(w) \, dw$$
and using H\"{o}lder's inequality.
\end{exercise}

\begin{exercise}
\label{ex:NRT.frac}
(proof suggested by Adrian Nachman)
The purpose of this exercise is to prove Theorem 2.3 of \cite{NRT:2017}, which asserts the following. 
For $0< \alpha < n$ and $f \in L^p(\R^n)$, $1 < p \leq 2$, the estimate
\begin{equation}
\label{NRT.frac}
\left| |D|^{-\alpha} f (x) \right| \lesssim \left( \lam^{n-\alpha} \calM \widehat{f}(0) + \lam^\alpha \calM f(0) \right).
\end{equation}
holds for any $\lam > 0$, where $|D|^{-\alpha}$ is the Fourier multiplier with symbol $|\xi|^{-\alpha}$. 
\begin{enumerate}
\item[(a)]	Prove that
				$$ \left( |D|^{-\alpha} f \right)(x)= c_\alpha \int_0^\infty t^{\alpha-1} \left( P_t * f \right)(x) \, dx. $$
\item[(b)]	 Prove estimate \eqref{NRT.frac} by splitting $\dint_0^\infty \, = \left( \dint_0^\lam + \dint_\lam^\infty \right)$ and mimicking the proof of Theorem \ref{thm:DBAR.est}. 
\medskip
\item[(c)]	Optimize in $\lam$  to show that
				$$ \left| \left(|D|^{-\alpha} f\right)(x) \right| 
						\lesssim \left( \calM \widehat{f}(0) \right)^{\alpha/n}  \left( \calM f (x) \right)^{1-\alpha/n}.
				$$
\end{enumerate}
\end{exercise}

\begin{exercise}[Duhamel's formula]
\label{ex:DSII.Duhamel}
Suppose that $f \in \scrS(\R^2)$ and $F \in C(\R;\scrS(\R^2))$. Show that the initial value problem 
$$ iv_t + 2(\dee_z^2 + \dee_{\zbar}^2) v = F, \quad v(0) = f $$
is solved by $v \in C(\R;\scrS(\R^2))$ given by 
$$ v(t) = V(t) f - i \int_0^t V(t-s) F(s) \, ds. $$
\end{exercise}

\begin{exercise}
\label{ex:DSII.U}
The purpose of this exercise is to prove the basic dispersive estimate \eqref{DSII:U.disp}. In this exercise we define
$$\widehat{f}(\xi_1,\xi_2) =
	 \int f(x_1,x_2) e^{-i(\xi_1x_1 + \xi_2 x_2)} \, dx_1 \, dx_2$$
so that the inverse Fourier transform is
$$\widecheck{g}(x_1,x_2) = 
	\frac{1}{(2\pi)^2} \int  g(\xi_1,\xi_2)  e^{i(\xi_1x_1 + \xi_2 x_2)} \, d\xi_1 \, d\xi_2.
$$
\begin{enumerate}
\item[(a)]	Let $h(\xi) = \frac{1}{2}\left(\xi_2^2 - \xi_1^2\right)$, and let 
$$ \widehat{f}(\xi_1,\xi_2) = \int_{\R^2} e^{-ix\cdot \xi} f(x) \, dx $$
be the usual (unitary) Fourier transform on $L^2(\R^2)$. Show that, if $f \in \scrS(\R^2)$, then the unique solution to 
\eqref{DSII:lin.bis} with $v(x,0) = f(x)$ is given by
$$
v(x,t) = \frac{1}{(2\pi)^2} \int e^{i x \cdot \xi} e^{it h(\xi)} \widehat{f}(\xi) \, d\xi 
$$
\item[(b)]	Compute the distribution Fourier transform of $e^{ith(\xi)}$ using the result of Exercise \ref{ex:Schrodinger.prop1} and separation of variables.
\item[(c)] Conclude that
$$ v(x,t) = \int K(t,x-y) f(y) \, dy $$
where
$$ \left| K(t,x) \right| \lesssim t^{-1}. $$
\end{enumerate}
\end{exercise}

\begin{exercise}
\label{ex:DSII.S2}
The purpose of this exercise is to prove the Strichartz estimate \eqref{DSII:S2}. 
\begin{enumerate}
\item[(a)]	Using the dispersive estimate $\norm[L^\infty]{V(t)f} \lesssim t^{-1} \norm[L^1]{f}$,
the trivial estimate $\norm[L^2]{V(t)(f)}$, and real interpolation, prove that for any $p > 2$,
$$ \norm[L^p]{V(t)f} \lesssim_{\, p} t^{2/p-1} \norm[L^{p'}]{f}. $$
\item[(b)]	Regarding $\dint_{-\infty}^\infty V(t-s) g(s) \, ds$ as a joint convolution in $s,z$, use 
part (a) with $p=4$ and the Hardy-Littlewood-Sobolev inequality \eqref{HLS} with $n=1$ and $\alpha=1/2$ 
to prove \eqref{DSII:S2}.
\end{enumerate}
\end{exercise}

\begin{exercise}
\label{ex:DSII.time}
The purpose of this exercise is to find the correct normalization for time-dependent joint solutions of \eqref{DSII:x1}--\eqref{DSII:t2}. 
Suppose that 
$$
\Psi_1(z,k,t) = m_1(z,k,t) e^{ikz}, 	\quad
\Psi_2(z,k,t)=m_2(z,k,t) e^{ikz}$$ 
solves \eqref{DSII:x1}--\eqref{DSII:x2} with 
$$\left(m^1(z,k,t),m^2(z,k,t)\right) \to (1,0) \text{ as }|z| \to  \infty. $$
Let
$$\psi_1(z,k,t) = C_1(k,t) \Psi_1(z,k,t), \quad
\psi_2(z,k,t) = C_2(k,t) \Psi_2(z,k,t)$$
be a joint solution of \eqref{DSII:x1}--\eqref{DSII:t2}.  
Assuming that 
$$\dee m_1(z,k,t)/\dee t, \, \dee m_2(z,k,t)/\dee t \to 0 \text{ as }|z| \to \infty$$ and that $q(z,t), \, g(z,t) \to 0$ as $|z| \to \infty$, show that
$C_1(k,t) = C_2(k,t) = e^{-2ik^2 t}$. 
\end{exercise}

\begin{exercise}
\label{ex:DSII.m.asy.k}
Suppose that the solutions $m^1$ and $m^2$ of \eqref{DSII:m.x1}-\eqref{DSII:m.x2} admit (differentiable) asymptotic expansions of the form
\begin{align*}
m^1(z,k,t)	&\sim	1 + \sum_{j \geq 1} \frac{\alpha_j(z,t)}{k^j},\\
m^2(z,k,t)	&\sim	\sum_{j \geq 1} \frac{\beta_j(z,t)}{k^j}
\end{align*}
for each $(z,t)$. 
Use \eqref{DSII:m.x1}--\eqref{DSII:m.x2} to show that 
%\sidenote{The answers are: \\ \\
% $\beta_1 = -i\eps \qbar$,\\ 
% $\alpha_1 = -i\eps \dee_{\zbar}^{-1}\left(|q|^2\right)$,\\ 
% $\beta_2 = -\qbar \dee_{\zbar}^{-1}\left(|q|^2\right) + i\eps \dee_z \qbar$.}
\begin{align*}
i\beta_1(z,t)		&=	\eps \overline{q(z,t)},\\
\left(\dee_{\zbar} \alpha_j \right)(z,t)	&= q(z,t) \beta_j(z,t),\\
\left(\dee_z \beta_j \right)(z,t) + i \beta_{j+1}(z,t)	&=	 \eps \overline{q(z,t)} \alpha_j(z,t)
\end{align*}
and compute $\beta_1$, $\alpha_1$, and $\beta_2$.
\end{exercise}

\begin{exercise}
\label{ex:DSII.two-s}
The \emph{conjugate Solid Cauchy transform} is the integral operator
$$ \left(\dee_z^{-1} f \right)(z) = \frac{1}{\pi} \int \frac{1}{\zbar-\zetabar} f(\zeta) \, d\zeta $$
and is a solution operator for the equation $\dee u = f$.
Suppose that $q \in \scrS(\R^2)$ and that equations \eqref{DSII:m} admit bounded solutions $m^1,m^2$.
Show that, in the large-$z$ asymptotic expansion \eqref{DSII:m2.asy.z}, the function $\bfs(k)$ is given by 
\eqref{DSII:scatt}. \emph{Hint}: Remember the identity \eqref{dee-id}. 
\end{exercise}

\begin{exercise}
\label{ex:DSII.F}
The unitary normalization for the Fourier transform on $\R^2$ is given by
$$ (\calF_0 f)(\xi) = \frac{1}{2\pi} \int_{\R^2} e^{-ix \cdot \xi} f(x) \, dx $$
so that
$$ \left(\calF_0^{-1} g \right)(x) = \frac{1}{2\pi} \int_{\R^2} e^{ix \cdot \xi} g(\xi) \, d\xi. $$
With this normalization, $\calF_0$ is a unitary map from $L^2(\R^2)$ to itself. Show that if $k=k_1+ik_2$,
$$ \left( \calF_a f \right)(k) = 2 \left( \calF_0 \overline{f} \right)(2k_1,-2k_2)$$
so that $\calF_a$ is also an isometry of $L^2(\R^2)$. Show also that $\calF_a^{-1} = \calF_a$.
\end{exercise}

\begin{exercise}
\label{ex:DSII.q.nice}
Show that, given $q \in L^2(\R^2)$ and any $\eps>0$, we may write
$$ q = q_n + q_s $$
where $q_n$ is a smooth function of compact support, and $\norm[L^2]{q_s} < \eps$.
\emph{Hint}: mollify $q$ and truncate to a large ball using a smooth cutoff function.
\end{exercise}

\begin{exercise}
\label{ex:DSII.Bukhgeim}
Prove \eqref{Bukhgeim} for $f \in C_0^\infty(\R^2)$ using the result of Exercise \ref{ex:DBAR.green} and the identity $$(i\kbar)\dee_{\zbar}(e_k)(z) = e_k(z).$$
By a density argument and \eqref{DBAR:HLS}, show that the same identity holds true as functions in $L^p(\R^2)$
provided $f \in L^p(\R^2) \cap L^{2p/(p+2)}(\R^2)$ and $\dee_{\zbar} f \in L^{2p/(2p+2)}(\R^2)$. 
\end{exercise}

\begin{exercise}
\label{ex:DSII.L21} Denote by $L^{2,1}(\R^2)$ the space of measurable complex-valued functions with $$
\norm[L^{2,1}]{f} \coloneqq \left( \int (1+|x|^2) |f(x)|^2 \, dx\right)^{1/2} < \infty.
$$ 
Show that for any $f \in L^{2,1}(\R^2)$ and any $p \in (1,2]$,  $\norm[L^p]{f} \lesssim \norm[L^{2,1}]{f}.$
\end{exercise}

\begin{exercise}
\label{ex:DSII.H2}
Show that for any $f \in H^1(\R^2)$ and any $p \in [2,\infty)$, $\norm[L^p]{f} \lesssim_{\, p} \norm[H^1]{f}$.
\emph{Hint}: By the Hausdorff-Young inequality, it suffices to show that $\norm[L^p]{f} \lesssim_p \norm[L^{2,1}]{f}$ for $p \in (1,2]$ (why?). Then,  use the result of Exercise \ref{ex:DSII.L21}. 
\end{exercise}

\begin{exercise}
\label{ex:DSII.H11.compact}
Using Theorem \ref{thm:Kolmogorov-Riesz}, show that $H^{1,1}(\R^2)$ is compactly embedded in $L^2(\R^2)$. That is,  show
that bounded subsets of $H^{1,1}(\R^2)$ are identified with subsets of $L^2(\R^2)$ having compact closure. 
\end{exercise}

\begin{exercise}
\label{ex:DSII.L2}
Show that the map $q \mapsto q \, \dee_{\zbar}^{-1}(e_k \qbar)$ is Lipschitz continuous from $L^2(\R^2)$ into $L^4(\R^2_k, L^{4/3}(\R_z^2))$. \emph{Hint}: Use \eqref{DBAR.est.fourier}, bilinearity in $q$, and H\"{o}lder's inequality.
\end{exercise}

\begin{exercise}
\label{ex:DSII.dbar.zdiff}
Using the commutation relation
$$ \dee_{\zbar} e_{-k} = e_{-k} \left( \dee_{\zbar} - i\kbar \right) $$
show that \eqref{DSII:Lax1.pre1} and \eqref{DSII:Lax1.pre2} hold.
\end{exercise}

\begin{exercise}
\label{ex:DSII.dbar.timediff}
Recall that 
$$ \varphi(z,k,t) = 2\left(k^2 + \kbar^2\right)t + \frac{kz + \kbar \zbar}{t}. $$
Using the commutation relations
$$
\begin{gathered}
\dee_z e^{it\varphi} = e^{it\varphi} \left(\dee_z - ik\right), \quad 
\dee_{\zbar} e^{it\varphi} = e^{it\varphi}\left(\dee_{\zbar} - i\kbar \right), \\
\dee_t e^{it\varphi} = e^{it\varphi} \left( \dee_t + 2i(k^2+\kbar^2) \right), 
\end{gathered}
$$
derive \eqref{DSII:dbart1}--\eqref{DSII:dbart2} from the $\dee_{\kbar}$ equations \eqref{DSII:m.kt} and the time evolution \eqref{s.ev}.
\end{exercise}

\begin{exercise}
\label{ex:DSII.w1w2}
Use the asymptotic expansions \eqref{m1.asy.k}--\eqref{m2.asy.k} to show that 
$w_1$ and $w_2$ as defined in \eqref{w1.def}--\eqref{w2.def} are of order $k^{-1}$ as $k \to \infty$. 
\end{exercise}

\begin{exercise}
\label{ex:DSII.w1w2.k}
Show that \eqref{DSII:Lax2.pre1} holds using \eqref{DSII:dbart1}--\eqref{DSII:dbart2} and the fact that \eqref{DSII:m.kt} holds. Note that
$e_{-k} \bfs(k,t) = e^{it\varphi} \bfs$. 
\end{exercise}

\bibliography{fields}

\providecommand{\bysame}{\leavevmode\hbox to3em{\hrulefill}\thinspace}
\providecommand{\MR}{\relax\ifhmode\unskip\space\fi MR }
\providecommand{\MRhref}[2]{%
  \href{http://www.ams.org/mathscinet-getitem?mr=#1}{#2}
}
\providecommand{\href}[2]{#2}
\begin{thebibliography}{10}

\bibitem{Ablowitz:2011}
Mark~J. Ablowitz, \emph{Nonlinear dispersive waves}, Cambridge Texts in Applied
  Mathematics, Cambridge University Press, New York, 2011, Asymptotic analysis
  and solitons. \MR{2848561}

\bibitem{AKNS:1974}
Mark~J. Ablowitz, David~J. Kaup, Alan~C. Newell, and Harvey Segur, \emph{The
  inverse scattering transform-{F}ourier analysis for nonlinear problems},
  Studies in Appl. Math. \textbf{53} (1974), no.~4, 249--315. \MR{0450815}

\bibitem{AS:1981}
Mark~J. Ablowitz and Harvey Segur, \emph{Solitons and the inverse scattering
  transform}, SIAM Studies in Applied Mathematics, vol.~4, Society for
  Industrial and Applied Mathematics (SIAM), Philadelphia, Pa., 1981.
  \MR{642018}

\bibitem{AFR:2015}
Kari Astala, Daniel Faraco, and Keith~M. Rogers, \emph{On {P}lancherel's
  identity for a two-dimensional scattering transform}, Nonlinearity
  \textbf{28} (2015), no.~8, 2721--2729. \MR{3382582}

\bibitem{AIM:2009}
Kari Astala, Tadeusz Iwaniec, and Gaven Martin, \emph{Elliptic partial
  differential equations and quasiconformal mappings in the plane}, Princeton
  Mathematical Series, vol.~48, Princeton University Press, Princeton, NJ,
  2009. \MR{2472875}

\bibitem{BC:1998}
R.~Beals and R.~R. Coifman, \emph{Scattering and inverse scattering for first
  order systems [ {MR}0728266 (85f:34020)]}, Surveys in differential geometry:
  integral systems [integrable systems], Surv. Differ. Geom., vol.~4, Int.
  Press, Boston, MA, 1998, pp.~467--519. \MR{1726933}

\bibitem{BC:1984}
Richard Beals and Ronald~R. Coifman, \emph{Scattering and inverse scattering
  for first order systems}, Comm. Pure Appl. Math. \textbf{37} (1984), no.~1,
  39--90. \MR{728266}

\bibitem{BC:1985a}
\bysame, \emph{Multidimensional inverse scatterings and nonlinear partial
  differential equations}, Pseudodifferential operators and applications
  ({N}otre {D}ame, {I}nd., 1984), Proc. Sympos. Pure Math., vol.~43, Amer.
  Math. Soc., Providence, RI, 1985, pp.~45--70. \MR{812283}

\bibitem{BC:1986}
\bysame, \emph{The {D}-bar approach to inverse scattering and nonlinear
  evolutions}, Phys. D \textbf{18} (1986), no.~1-3, 242--249, Solitons and
  coherent structures (Santa Barbara, Calif., 1985). \MR{838329}

\bibitem{BC:1989}
\bysame, \emph{Linear spectral problems, nonlinear equations and the
  {$\overline\partial$}-method}, Inverse Problems \textbf{5} (1989), no.~2,
  87--130. \MR{991913}

\bibitem{Brown:2001}
R.~M. Brown, \emph{Estimates for the scattering map associated with a
  two-dimensional first-order system}, J. Nonlinear Sci. \textbf{11} (2001),
  no.~6, 459--471. \MR{1871279}

\bibitem{BOP:2016}
R.~M. Brown, K.~A. Ott, and P.~A. Perry, \emph{Action of a scattering map on
  weighted {S}obolev spaces in the plane}, J. Funct. Anal. \textbf{271} (2016),
  no.~1, 85--106. \MR{3494243}

\bibitem{BU:1997}
Russell~M. Brown and Gunther~A. Uhlmann, \emph{Uniqueness in the inverse
  conductivity problem for nonsmooth conductivities in two dimensions}, Comm.
  Partial Differential Equations \textbf{22} (1997), no.~5-6, 1009--1027.
  \MR{1452176}

\bibitem{CW:1989}
Thierry Cazenave and Fred~B. Weissler, \emph{Some remarks on the nonlinear
  {S}chr\"odinger equation in the critical case}, Nonlinear semigroups, partial
  differential equations and attractors ({W}ashington, {DC}, 1987), Lecture
  Notes in Math., vol. 1394, Springer, Berlin, 1989, pp.~18--29. \MR{1021011}

\bibitem{DZ:1993}
P.~Deift and X.~Zhou, \emph{A steepest descent method for oscillatory
  {R}iemann-{H}ilbert problems. {A}symptotics for the {MK}d{V} equation}, Ann.
  of Math. (2) \textbf{137} (1993), no.~2, 295--368. \MR{1207209}

\bibitem{Deift:2018}
Percy Deift, \emph{Fifty years of kdv: An integrable system (coxeter lectures,
  august 2017)}, 2018.

\bibitem{DZ:2003}
Percy Deift and Xin Zhou, \emph{Long-time asymptotics for solutions of the
  {NLS} equation with initial data in a weighted {S}obolev space}, Comm. Pure
  Appl. Math. \textbf{56} (2003), no.~8, 1029--1077, Dedicated to the memory of
  J\"urgen K. Moser. \MR{1989226}

\bibitem{DM:2008}
M.~{Dieng} and K.~D.~T.~. {McLaughlin}, \emph{{Long-time Asymptotics for the
  NLS equation via dbar methods}}, ArXiv e-prints (2008).

\bibitem{DMM:2018}
M.~{Dieng}, P.~Miller, and K.~D.~T. {McLaughlin}, \emph{Dispersive asymptotics
  for linear and integrable equations by the $\dbar$ method of steepest
  descent}, 2018.

\bibitem{Faddeev:1965}
L.~D. Faddeev, \emph{Increasing solutions of the {S}chr\"{o}dinger equation},
  Sov.\ Phys.\ Doklady \textbf{10} (1965), 1033--1035.

\bibitem{FA:1983c}
A.~S. Fokas and M.~J. Ablowitz, \emph{Method of solution for a class of
  multidimensional nonlinear evolution equations}, Phys. Rev. Lett. \textbf{51}
  (1983), no.~1, 7--10. \MR{711737}

\bibitem{FA:1984a}
\bysame, \emph{On the inverse scattering transform of multidimensional
  nonlinear equations related to first-order systems in the plane}, J. Math.
  Phys. \textbf{25} (1984), no.~8, 2494--2505. \MR{751539}

\bibitem{FA:1983b}
Athanassios~S. Fokas and Mark~J. Ablowitz, \emph{The inverse scattering
  transform for multidimensional {$(2+1)$}\ problems}, Nonlinear phenomena
  ({O}axtepec, 1982), Lecture Notes in Phys., vol. 189, Springer, Berlin, 1983,
  pp.~137--183. \MR{727861}

\bibitem{FL:2012}
Rupert~L. Frank and Elliott~H. Lieb, \emph{A new, rearrangement-free proof of
  the sharp {H}ardy-{L}ittlewood-{S}obolev inequality}, Spectral theory,
  function spaces and inequalities, Oper. Theory Adv. Appl., vol. 219,
  Birkh\"auser/Springer Basel AG, Basel, 2012, pp.~55--67. \MR{2848628}

\bibitem{Gerard:1998}
Patrick G\'erard, \emph{Description du d\'efaut de compacit\'e de l'injection
  de {S}obolev}, ESAIM Control Optim. Calc. Var. \textbf{3} (1998), 213--233.
  \MR{1632171}

\bibitem{GS:1990}
Jean-Michel Ghidaglia and Jean-Claude Saut, \emph{On the initial value problem
  for the {D}avey-{S}tewartson systems}, Nonlinearity \textbf{3} (1990), no.~2,
  475--506. \MR{1054584}

\bibitem{Grinevich:2000}
P.~G. Grinevich, \emph{The scattering transform for the two-dimensional
  {S}chr\"{o}dinger operator with a potential that decreases at infinity at
  fixed nonzero energy}, Uspekhi Mat. Nauk \textbf{55} (2000), no.~6(336),
  3--70. \MR{1840357}

\bibitem{HH:2010}
Harald Hanche-Olsen and Helge Holden, \emph{The {K}olmogorov-{R}iesz
  compactness theorem}, Expo. Math. \textbf{28} (2010), no.~4, 385--394.
  \MR{2734454}

\bibitem{HH:2016}
\bysame, \emph{Addendum to ``{T}he {K}olmogorov-{R}iesz compactness theorem''
  [{E}xpo. {M}ath. 28 (2010) 385--394] [ {MR}2734454]}, Expo. Math. \textbf{34}
  (2016), no.~2, 243--245. \MR{3494284}

\bibitem{KVZ:2018}
Rowan Killip, Monica Vi\c{s}an, and Xiaoyi Zhang, \emph{Low regularity
  conservation laws for integrable {PDE}}, Geom. Funct. Anal. \textbf{28}
  (2018), no.~4, 1062--1090. \MR{3820439}

\bibitem{KV:2018}
Rowan {Killip} and Monica {Visan}, \emph{{KdV is wellposed in \$H\^\{-1\}\$}},
  arXiv e-prints (2018), arXiv:1802.04851.

\bibitem{KT:2018}
Herbert Koch and Daniel Tataru, \emph{Conserved energies for the cubic
  nonlinear {S}chr\"{o}dinger equation in one dimension}, Duke Math. J.
  \textbf{167} (2018), no.~17, 3207--3313. \MR{3874652}

\bibitem{Lieb:1983}
Elliott~H. Lieb, \emph{Sharp constants in the {H}ardy-{L}ittlewood-{S}obolev
  and related inequalities}, Ann. of Math. (2) \textbf{118} (1983), no.~2,
  349--374. \MR{717827}

\bibitem{LP:2015}
Felipe Linares and Gustavo Ponce, \emph{Introduction to nonlinear dispersive
  equations}, second ed., Universitext, Springer, New York, 2015. \MR{3308874}

\bibitem{MTT:2003}
Camil Muscalu, Terence Tao, and Christoph Thiele, \emph{A {C}arleson theorem
  for a {C}antor group model of the scattering transform}, Nonlinearity
  \textbf{16} (2003), no.~1, 219--246. \MR{1950785}

\bibitem{NRT:2017}
A.~I. {Nachman}, I.~{Regev}, and D.~I. {Tataru}, \emph{{A Nonlinear Plancherel
  Theorem with Applications to Global Well-Posedness for the Defocusing
  Davey-Stewartson Equation and to the Inverse Boundary Value Problem of
  Calderon}}, ArXiv e-prints (2017).

\bibitem{Ozawa:1992}
Tohru Ozawa, \emph{Exact blow-up solutions to the {C}auchy problem for the
  {D}avey-{S}tewartson systems}, Proc. Roy. Soc. London Ser. A \textbf{436}
  (1992), no.~1897, 345--349. \MR{1177134}

\bibitem{Perry:2016}
Peter~A. Perry, \emph{Global well-posedness and long-time asymptotics for the
  defocussing {D}avey-{S}tewartson {II} equation in {$H^{1,1}(\Bbb{C})$}}, J.
  Spectr. Theory \textbf{6} (2016), no.~3, 429--481, With an appendix by
  Michael Christ. \MR{3551174}

\bibitem{Sung:1994a}
Li-Yeng Sung, \emph{An inverse scattering transform for the
  {D}avey-{S}tewartson {II} equations. {I}}, J. Math. Anal. Appl. \textbf{183}
  (1994), no.~1, 121--154. \MR{1273437}

\bibitem{Sung:1994b}
\bysame, \emph{An inverse scattering transform for the {D}avey-{S}tewartson
  {II} equations. {II}}, J. Math. Anal. Appl. \textbf{183} (1994), no.~2,
  289--325. \MR{1274142}

\bibitem{Sung:1994c}
\bysame, \emph{An inverse scattering transform for the {D}avey-{S}tewartson
  {II} equations. {III}}, J. Math. Anal. Appl. \textbf{183} (1994), no.~3,
  477--494. \MR{1274849}

\bibitem{Tao:2006}
Terence Tao, \emph{Nonlinear dispersive equations}, CBMS Regional Conference
  Series in Mathematics, vol. 106, Published for the Conference Board of the
  Mathematical Sciences, Washington, DC; by the American Mathematical Society,
  Providence, RI, 2006, Local and global analysis. \MR{2233925}

\bibitem{TO:2016}
Thomas Trogdon and Sheehan Olver, \emph{Riemann-{H}ilbert problems, their
  numerical solution, and the computation of nonlinear special functions},
  Society for Industrial and Applied Mathematics (SIAM), Philadelphia, PA,
  2016. \MR{3450072}

\bibitem{Vekua:1962}
I.~N. Vekua, \emph{Generalized analytic functions}, Pergamon Press,
  London-Paris-Frankfurt; Addison-Wesley Publishing Co., Inc., Reading, Mass.,
  1962. \MR{0150320}

\bibitem{ZM:1976}
V.~E. Zakharov and S.~V. Manakov, \emph{Asymptotic behavior of non-linear wave
  systems integrated by the inverse scattering method}, Z. \`Eksper. Teoret.
  Fiz. \textbf{71} (1976), no.~1, 203--215. \MR{0673411}

\bibitem{ZS:1972}
V.~E. Zakharov and A.~B. Shabat, \emph{Exact theory of two-dimensional
  self-focusing and one-dimensional self-modulation of waves in nonlinear
  media}, \v Z. \`Eksper. Teoret. Fiz. \textbf{61} (1971), no.~1, 118--134.
  \MR{0406174}

\bibitem{Zhou:1998}
Xin Zhou, \emph{{$L^2$}-{S}obolev space bijectivity of the scattering and
  inverse scattering transforms}, Comm. Pure Appl. Math. \textbf{51} (1998),
  no.~7, 697--731. \MR{1617249}

\end{thebibliography}
\bibliographystyle{amsplain}

\end{document}